\documentclass[8pt,oneside,english,reqno]{amsart}
\usepackage[T1]{fontenc}
\usepackage[latin9]{inputenc}
\usepackage{prettyref}
\usepackage{amsthm}
\usepackage{amstext}
\usepackage{graphicx}
\usepackage{setspace}
\usepackage{amssymb}
\usepackage{esint}
\usepackage{xargs}[2008/03/08]
\makeatletter
\numberwithin{equation}{section}
\numberwithin{figure}{section}
\theoremstyle{plain}
\newtheorem{thm}{Theorem}[section]
  \theoremstyle{plain}
  \newtheorem{lem}[thm]{Lemma}
  \theoremstyle{plain}
  \newtheorem{algorithm}[thm]{Algorithm}
  \theoremstyle{definition}
  \newtheorem{defn}[thm]{Definition}
  \theoremstyle{remark}
  \newtheorem{claim}[thm]{Claim}
  \theoremstyle{plain}
  \newtheorem*{lem*}{Lemma}
  \theoremstyle{plain}
  \newtheorem{cor}[thm]{Corollary}
  \theoremstyle{remark}
  \newtheorem{rem}[thm]{Remark}
 \theoremstyle{definition}
  \newtheorem{example}[thm]{Example}
  \theoremstyle{plain}
  \newtheorem{prop}[thm]{Proposition}

\usepackage{etex}
\usepackage{constants}
\usepackage{dimadefs}
\usepackage{fullpage}
\usepackage{concmath}
\usepackage{pst-all}
\usepackage[dvips]{hyperref}
\usepackage{enumitem}

\DeclareMathOperator{\shift}{E}
\DeclareMathOperator{\id}{I}
\DeclareMathOperator{\diag}{diag}
\DeclareMathOperator{\diam}{diam}
\DeclareMathOperator{\combop}{\mathcal{D}}

\usepackage{prettyref}
\newrefformat{sub}{Subsection \ref{#1}}
\newrefformat{app}{Appendix \ref{#1}}
\newrefformat{cor}{Corollary \ref{#1}}
\newrefformat{alg}{Algorithm \ref{#1}}
\newrefformat{def}{Definition \ref{#1}}
\newrefformat{clm}{Claim \ref{#1}}
\newrefformat{prop}{Proposition \ref{#1}}

\newconstantfamily{kk}{symbol=K}
\newconstantfamily{rr}{symbol=R}

\@ifundefined{showcaptionsetup}{}{%
 \PassOptionsToPackage{caption=false}{subfig}}
\usepackage{subfig}
\makeatother

\usepackage{babel}

\begin{document}

\title{Algebraic Fourier reconstruction of piecewise smooth functions}

\author{Dmitry Batenkov}

\address{Department of Mathematics\\
Weizmann Institute of Science\\
Rehovot 76100\\
Israel}

\email{dima.batenkov@weizmann.ac.il}

\urladdr{http://www.wisdom.weizmann.ac.il/~dmitryb}

\author{Yosef Yomdin}

\email{yosef.yomdin@weizmann.ac.il}

\urladdr{http://www.wisdom.weizmann.ac.il/~yomdin}



\date{\today}

\maketitle
\newcommandx\fc[1][usedefault, addprefix=\global, 1=k]{\ensuremath{c_{#1}}}
\global\long\def\startcoeff{\ensuremath{M}}
\newcommandx\frsum[2][usedefault, addprefix=\global, 1=\fun, 2=\startcoeff]{\mathfrak{F}_{#2}\left(#1\right)}
\global\long\def\ord{\ensuremath{d}}
\global\long\def\fun{\ensuremath{f}}
\global\long\def\fcsmoothbound{R}
\global\long\def\vec#1{\ensuremath{\mathbf{#1}}}
\global\long\def\smooth{\ensuremath{\Psi}}
\global\long\def\sing{\ensuremath{\Phi}}
\global\long\def\jc{\ensuremath{A}}
\global\long\def\np{\ensuremath{K}}
\global\long\def\jp{\xi}
\global\long\def\nn#1{\widetilde{#1}}
\global\long\def\dimdim{\mathcal{F}}
\global\long\def\vec#1{\ensuremath{\mathbf{#1}}}
\global\long\def\vvec#1{\boldsymbol{#1}}
\global\long\def\lag#1#2{\mathcal{L}_{#2}^{(#1)}}

\begin{abstract}
Accurate reconstruction of piecewise-smooth functions from a finite
number of Fourier coefficients is an important problem in various
applications. This probelm exhibits an inherent inaccuracy, in particular
the Gibbs phenomenon, and it is being intensively investigated during
the last decades. Several nonlinear reconstruction methods have been
proposed in the literature, and it is by now well-established that
the ``classical'' convergence order can be completely restored up
to the discontinuities. Still, the maximal accuracy of determining
the positions of these discontinuities remains an open question.

In this paper we prove that the locations of the jumps (and subsequently
the pointwise values of the function) can be reconstructed with at
least ``half the classical accuracy''. In particular, we develop
a constructive approximation procedure which, given the first $k$
Fourier coefficients of a piecewise-$C^{2d+1}$ function, recovers
the locations of the jumps with accuracy $\sim k^{-\left(d+2\right)}$,
and the values of the function between the jumps with accuracy $\sim k^{-\left(d+1\right)}$
(similar estimates are obtained for the associated jump magnitudes).
A key ingredient of the algorithm is to start with the case of a single
discontinuity, where a modified version of one of the existing algebraic
methods (due to K.Eckhoff) may be applied. It turns out that the additional
orders of smoothness produce highly correlated error terms in the
Fourier coefficients, which eventually cancel out in the corresponding
algebraic equations. To handle more than one jump, we apply a localization
procedure via a convolution in the Fourier domain, which eventually
preserves the accuracy estimates obtained for the single jump. We
provide some numerical results which support the theoretical predictions.
\end{abstract}

\section{\label{sec:intro}Introduction}

Consider the problem of reconstructing a function $\fun:\left[-\pi,\pi\right]\to\reals$
from a finite number of its Fourier coefficients\[
\fc(\fun)\isdef\frac{1}{2\pi}\int_{-\pi}^{\pi}\fun(t)\ee^{-\imath kt}\dd t,\qquad k=0,1,\dots\startcoeff\]

It is well-known that for periodic smooth functions, the truncated
Fourier series\[
\frsum\isdef\sum_{|k|=0}^{\startcoeff}\fc(\fun)\ee^{\imath kx}\]
converges to $\fun$ very fast, subsequently making Fourier analysis
very attractive in a vast number of applications. We have by the classical
Lebesgue lemma (see e.g. \cite{Natanson1949}) that \[
\max_{-\pi\leq x\leq\pi}\left|\fun(x)-\frsum(x)\right|\leq\left(3+\ln\startcoeff\right)\cdot E_{\startcoeff}\left(\fun\right)\]
where $E_{\startcoeff}\left(\fun\right)$ is the error of the best
uniform approximation to $\fun$ by trigonometric polynomials of degree
at most $\startcoeff$. This number, in turn, depends on the smoothness
of the function. In particular:
\begin{enumerate}
\item If $\fun$ is $\ord$-times continuously differentiable (including
at the endpoints) and $\left|\der{\fun}{\ord}(x)\right|\leq R$, then
(see \cite[Vol.I, Chapter 3, Theorem 13.6]{zygmund1959trigonometric})\begin{equation}
E_{\startcoeff}\left(\fun\right)\leq C_{\ord}\cdot R\cdot\startcoeff^{-\ord}\label{eq:best-approximation-smooth}\end{equation}

\item If $\fun$ is analytic, then by classical results of S.Bernstein (see
e.g. \cite[Chapter IX]{Natanson1949}) there exist constants $C$
and $q<1$ such that\begin{equation}
E_{\startcoeff}\left(\fun\right)\leq C\cdot q^{\startcoeff}\label{eq:best-approximation-analytic}\end{equation}

\end{enumerate}
Yet many realistic phenomena exhibit discontinuities, in which case
the unknown function $\fun$ is only piecewise-smooth. As a result,
the trigonometric polynomial $\frsum$ no longer provides a good approximation
to $\fun$ due to the slow convergence of the Fourier series (one
of the manifestations of this fact is commonly known as the ``Gibbs
phenomenon''). It has very serious implications, for example when
using spectral methods to calculate solutions of PDEs with shocks.
Therefore an important question arises: \emph{``Can such piecewise-smooth
functions be reconstructed from their Fourier measurements, with accuracy
which is comparable to the 'classical' one (such as \eqref{eq:best-approximation-smooth}
or \eqref{eq:best-approximation-analytic})''?}

This problem has received much attention, especially in the last few
decades (\cite{guilpin2004eps,brezinski2004extrapolation,beckermann2008rgp,marchbarone:2000,gelb1999detection,bauer-band,mhaskar2000polynomial,eckhoff1995arf,eckhoff1998high,kvernadze2004ajd,banerjee1998exponentially,gottlieb1997gpa,driscoll2001pba,Boyd20091404,barkhudaryan2007asymptotic,wei2007detection}
would be only a partial list). It has long been known that the key
problem for Fourier series acceleration is the detection of the shock
locations. By now it is well-established that classical convergence
rates can be restored uniformly up to the discontinuities (see e.g.
\cite{gottlieb1997gpa}), but the corresponding question for the jump
locations themselves is still open. Notice that any linear approximation
procedure with free (a-priori unknown) jump locations will not be
able to achieve accuracy higher than $\frac{1}{\sqrt{\startcoeff}}$
- see \cite{ettinger2008lvn}.

Several partial results and conjectures in this direction are known,
in particular the following. The concentration method of Gelb\&Tadmor
\cite{gelb-segmentation,gelb1999detection} recovers the jumps with
first order accuracy, and it can be extended to higher orders. Kvernadze
\cite{kvernadze2004ajd} proves that his method can recover jumps
of a $C^{3}$ function with second order accuracy. In \cite{ettinger2008lvn,batenkov2009arp}
we have conjectured that the locations of the jumps of a piecewise
$C^{\ord}$ function can be recovered with accuracy $k^{-d}$ from
its first $k$ Fourier coefficients (a similar conjecture is made
in \cite{vindas2009local}). Both Eckhoff \cite{eckhoff1995arf} and
Banerjee\&Geer \cite{banerjee1998exponentially} made the same conjectures
with respect to their particular reconstruction methods. We would
also like to mention a related but different problem: reconstruction
of piecewise-smooth functions from point measurements. There, adaptive
approximations can achieve asymptotic accuracy $k^{-d}$ for piecewise
$C^{d}$ functions \cite{arandiga2005interpolation,Plaskota2009227,YaronLipman07082009}.

With this motivation, our main goal in this paper is to arrive at
a better understanding of the ``optimal'', or the ``best possible''
accuracy of reconstruction, especially with respect to the locations
and the magnitudes of the jumps. As a means to achieve this goal,
we develop a reconstruction method which allows for explicit accuracy
analysis. Our method is a ``hybrid'' between a Fourier filtering
technique which is first applied to localize the jumps, and the algebraic
approach of Eckhoff/Kvernadze which is used in order to resolve each
discontinuity one at a time to a high order of accuracy. It is precisely
this ``localization'' which makes the subsequent analysis tractable.

Our accuracy analysis is ``asymptotic'' in nature, although we provide
the explicit constants at every step. These constants in general depend
upon various a-priori estimates (such as the minimal distance between
the jumps, or the upper bound on the jump magnitudes), which are presumably
available. See discussion in \prettyref{sec:procedure-overview} below,
in particular \eqref{eq:desired-estimates}.

Let us now give a brief summary of the main results.
\begin{enumerate}
\item If a function with a single jump has at least $2d+1$ continuous derivatives
everywhere except the jump, then the jump location can be recovered
from the first $\startcoeff$ Fourier coefficients with error at most
$\sim\startcoeff^{-d-2}$ (\prettyref{thm:jump-final-accuracy}).
In addition, a jump in the $l$-th derivative can be recovered with
error at most $\sim\startcoeff^{l-d-1}$ (\prettyref{thm:magnitude-final-accuracy}).\emph{
}A key observation in the analysis is that the additional orders of
smoothness produce highly correlated error terms in the Fourier coefficients,
which eventually cancel out in the corresponding algebraic equations.
\item The localization step does not ``destroy'' the above accuracy estimates
(\prettyref{thm:localization-preserves-accuracy}). Thus, the pointwise
values of $f$ are recovered with the accuracy $\sim\startcoeff^{-d-1}$
(\prettyref{thm:final-accuracy}) up to the jumps.
\item Numerical simulations are consistent with the theoretical accuracy
predictions (\prettyref{sec:numerical-results}).
\end{enumerate}
By means of this constructive approximation procedure with provable
asymptotic convergence properties, we therefore demonstrate that \emph{the
algebraic reconstruction methods for piecewise-smooth data can be
at least ``half accurate'' compared to the classical approximation
theory for smooth data.}

We provide an overview of the reconstruction procedure in \prettyref{sec:procedure-overview}.
For expository reasons, the details of the localization step and the
analysis of its accuracy are postponed until \prettyref{sec:localization}.
The resolution method of a single jump is presented in \prettyref{sec:single-jump-reconstruction},
while \prettyref{sec:accuracy-single-jump} is devoted to proving
its asymptotic convergence order. Finally, the accuracy of the whole
reconstruction is analyzed in \prettyref{sec:final-accuracy}. Some
common notations used throughout the paper are summarized below.

\thanks{We would like to thank Ch.Fefferman, E.Tadmor and N.Zobin for useful
discussions.}

\subsection{Notation}
\begin{itemize}
\item $\naturals$ denotes the natural numbers, $\reals$ - the real numbers,
$\complexfield$ - the complex numbers.
\item $C^{d}$ denotes the class of smooth functions which are continuously
differentiable $d$ times everywhere. $C^{\infty}$ is the class of
smooth functions having continuous derivatives of all orders.
\item $B_{r}\left(z\right)$ is the ball of radius $r$ centered at $z$,
and $\partial B_{r}\left(z\right)$ is the boundary of such a ball.
\end{itemize}

\section{The algebraic reconstruction method\label{sec:procedure-overview}}

\global\long\def\jcbound{\jc^{*}}
\global\long\def\geom{\mathbb{G}}
\global\long\def\jcallboundu{L}
\global\long\def\magbounds{\mathbb{A}}

Let us assume that $\fun$ has $\np\geq0$ jump discontinuities $\left\{ \jp_{j}\right\} _{j=1}^{\np}$.
Furthermore, we assume that $\fun\in C^{\ord}$ in every segment $\left(\jp_{j-1},\jp_{j}\right)$,
and we denote the associated jump magnitudes at $\jp_{j}$ by\[
\jc_{l,j}\isdef\der{\fun}{l}(\jp_{j}^{+})-\der{\fun}{l}(\jp_{j}^{-})\]

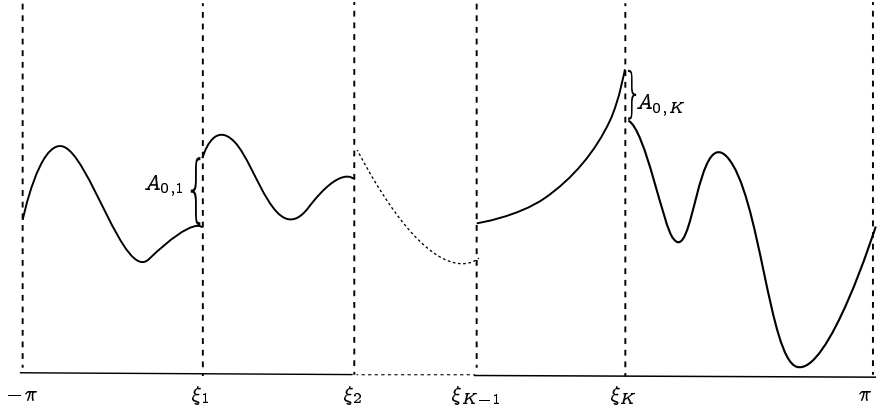
\begin{figure}
\newcommand{\piecewiselabels}{
\rput(17,20){$-\pi$}
\rput(153,20){$\jp_1$}
\rput(268,20){$\jp_2$}
\rput(361,20){$\jp_{\np-1}$}
\rput(472,20){$\jp_{\np}$}
\rput(655,20){$\pi$}
\rput(125,179){$\jc_{0,1}$}
\rput(500,240){$\jc_{0,\np}$}
}
\psset{xunit=.5pt,yunit=.5pt,runit=.5pt}
\begin{pspicture}(663.50164795,324.19845581)
{
\newrgbcolor{curcolor}{0 0 0}
\pscustom[linewidth=1.46649683,linecolor=curcolor,linestyle=dashed,dash=4.39949053 4.39949053]
{
\newpath
\moveto(268.253057,318.09554581)
\lineto(268.253057,35.87023581)
}
}
{
\newrgbcolor{curcolor}{0 0 0}
\pscustom[linewidth=1.46649683,linecolor=curcolor,linestyle=dashed,dash=4.39949053 4.39949053]
{
\newpath
\moveto(153.680877,34.96191581)
\lineto(153.680877,317.18722581)
}
}
{
\newrgbcolor{curcolor}{0 0 0}
\pscustom[linewidth=1.46649683,linecolor=curcolor,linestyle=dashed,dash=4.39949053 4.39949053]
{
\newpath
\moveto(17.360282,41.23989581)
\lineto(17.360282,323.46520781)
}
}
{
\newrgbcolor{curcolor}{0 0 0}
\pscustom[linewidth=1.08071089,linecolor=curcolor]
{
\newpath
\moveto(15.266137,37.40652581)
\lineto(267.002347,36.03931581)
}
}
{
\newrgbcolor{curcolor}{0 0 0}
\pscustom[linewidth=1.46649683,linecolor=curcolor,linestyle=dashed,dash=4.39949053 4.39949053]
{
\newpath
\moveto(660.627627,318.09554581)
\lineto(660.627627,35.87023581)
}
}
{
\newrgbcolor{curcolor}{0 0 0}
\pscustom[linewidth=1.46649683,linecolor=curcolor,linestyle=dashed,dash=4.39949053 4.39949053]
{
\newpath
\moveto(473.305157,34.96191581)
\lineto(473.305157,317.18722581)
}
}
{
\newrgbcolor{curcolor}{0 0 0}
\pscustom[linewidth=1.46649683,linecolor=curcolor,linestyle=dashed,dash=4.39949053 4.39949053]
{
\newpath
\moveto(360.837117,41.23989581)
\lineto(360.837117,323.46520781)
}
}
{
\newrgbcolor{curcolor}{0 0 0}
\pscustom[linewidth=1.15019548,linecolor=curcolor]
{
\newpath
\moveto(358.771217,35.35582581)
\lineto(662.926537,34.07405581)
}
}
{
\newrgbcolor{curcolor}{0 0 0}
\pscustom[linewidth=0.82122475,linecolor=curcolor,linestyle=dashed,dash=2.46367434 2.46367434]
{
\newpath
\moveto(359.517737,35.90957581)
\lineto(270.013357,35.90957581)
}
}
{
\newrgbcolor{curcolor}{0 0 0}
\pscustom[linewidth=1.46649683,linecolor=curcolor]
{
\newpath
\moveto(17.393624,153.14648581)
\curveto(17.393624,153.14648581)(29.731092,213.21692581)(47.494186,208.82872581)
\curveto(67.550637,203.87397581)(94.251267,105.51786581)(112.803847,123.36883581)
\curveto(147.787597,157.02960581)(153.353197,147.41223581)(153.353197,147.41223581)
}
}
{
\newrgbcolor{curcolor}{0 0 0}
\pscustom[linewidth=1.46649683,linecolor=curcolor]
{
\newpath
\moveto(153.353197,199.83325581)
\curveto(153.353197,199.83325581)(161.304047,229.88751581)(178.000837,211.85495581)
\curveto(194.697627,193.82239581)(209.009167,130.10736581)(233.656817,162.56596581)
\curveto(258.304447,195.02457581)(267.845467,184.20503581)(267.845467,184.20503581)
}
}
{
\newrgbcolor{curcolor}{0 0 0}
\pscustom[linewidth=1.46649683,linecolor=curcolor]
{
\newpath
\moveto(360.870437,150.54425581)
\curveto(360.870437,150.54425581)(391.878767,155.35294581)(412.550987,169.77898581)
\curveto(433.223207,184.20503581)(450.715077,204.64192581)(461.846267,228.68534581)
\curveto(469.002037,244.31355581)(472.977457,267.15480581)(472.977457,267.15480581)
}
}
{
\newrgbcolor{curcolor}{0 0 0}
\pscustom[linewidth=1.60117018,linecolor=curcolor]
{
\newpath
\moveto(475.587157,228.28361581)
\curveto(488.512897,221.95430581)(502.284417,146.56775581)(509.537827,139.14431581)
\curveto(524.677247,117.38028581)(527.735027,225.52330581)(550.359087,200.69774581)
\curveto(574.961637,173.70117581)(582.145597,48.39384581)(602.968927,41.93968581)
\curveto(626.859187,34.53492581)(662.732927,148.13321581)(662.732927,148.13321581)
}
}
{
\newrgbcolor{curcolor}{0 0 0}
\pscustom[linewidth=1,linecolor=curcolor,linestyle=dashed,dash=2 2]
{
\newpath
\moveto(270.523117,205.98520581)
\curveto(330.649117,91.46696581)(362.500877,124.34837581)(362.500877,124.34837581)
}
}
{
\newrgbcolor{curcolor}{0 0 0}
\pscustom[linestyle=none,fillstyle=solid,fillcolor=curcolor]
{
\newpath
\moveto(151.80040904,151.39123636)
\lineto(151.80040904,149.03967393)
\lineto(151.24525371,149.03967393)
\curveto(149.93634081,149.03968307)(149.02913667,149.71416154)(148.52363857,151.06311137)
\curveto(148.0090962,152.41207543)(147.75182935,155.10087474)(147.75183726,159.12951738)
\lineto(147.75183726,166.04748592)
\curveto(147.75182935,169.16466234)(147.44040107,171.38861838)(146.81755147,172.71936072)
\curveto(146.51062966,173.3755955)(146.08636504,173.70372016)(145.54475632,173.70373569)
\lineto(144.54276866,173.70373569)
\lineto(144.54276866,176.10998562)
\lineto(145.54475632,176.10998562)
\curveto(145.92388071,176.10996768)(146.34814534,176.43809234)(146.81755147,177.09436059)
\curveto(147.44040107,177.95111162)(147.75182935,180.15683851)(147.75183726,183.71154789)
\lineto(147.75183726,190.65686018)
\curveto(147.75182935,194.68546938)(148.0090962,197.36515412)(148.52363857,198.69592244)
\curveto(149.02913667,200.04483886)(149.93634081,200.71931733)(151.24525371,200.71935988)
\lineto(151.80040904,200.71935988)
\lineto(151.80040904,198.3951412)
\lineto(151.19109221,198.3951412)
\curveto(150.33352472,198.39510098)(149.77385649,197.966716)(149.51208586,197.10998499)
\curveto(149.2593228,196.2531761)(149.1329461,194.43937588)(149.13295539,191.6685789)
\lineto(149.13295539,183.76623539)
\curveto(149.1329461,180.90423358)(148.99302905,178.82611072)(148.7132038,177.53186058)
\curveto(148.44238773,176.23757172)(147.97298857,175.36257262)(147.30500492,174.90686065)
\curveto(147.97298857,174.41465693)(148.44690118,173.52142868)(148.72674418,172.22717324)
\curveto(148.9975425,170.93288968)(149.1329461,168.8638814)(149.13295539,166.02014217)
\lineto(149.13295539,158.09045491)
\curveto(149.1329461,155.33785366)(149.2593228,153.53316802)(149.51208586,152.67639257)
\curveto(149.77385649,151.81962812)(150.33352472,151.39124314)(151.19109221,151.39123636)
\lineto(151.80040904,151.39123636)
}
}
{
\newrgbcolor{curcolor}{0 0 0}
\pscustom[linestyle=none,fillstyle=solid,fillcolor=curcolor]
{
\newpath
\moveto(475.37344281,230.62605401)
\lineto(475.37344281,228.90241551)
\lineto(475.84238011,228.90241551)
\curveto(476.94801317,228.9024222)(477.71432457,229.39679864)(478.14131663,230.38554631)
\curveto(478.57594835,231.3743044)(478.79326054,233.34512939)(478.79325386,236.29802721)
\lineto(478.79325386,241.36873116)
\curveto(478.79326054,243.65354623)(479.05632266,245.28365233)(479.58244102,246.25905434)
\curveto(479.84169654,246.7400586)(480.20007103,246.98056605)(480.65756556,246.98057744)
\lineto(481.50394021,246.98057744)
\lineto(481.50394021,248.74430055)
\lineto(480.65756556,248.74430055)
\curveto(480.33732084,248.7442874)(479.97894635,248.98479486)(479.58244102,249.46582364)
\curveto(479.05632266,250.09380147)(478.79326054,251.71054604)(478.79325386,254.31606221)
\lineto(478.79325386,259.40680847)
\curveto(478.79326054,262.35968177)(478.57594835,264.32382601)(478.14131663,265.29924706)
\curveto(477.71432457,266.28797024)(476.94801317,266.78234668)(475.84238011,266.78237786)
\lineto(475.37344281,266.78237786)
\lineto(475.37344281,265.07878167)
\lineto(475.88813009,265.07878167)
\curveto(476.61251364,265.07875219)(477.08526297,264.76475634)(477.30637949,264.13679319)
\curveto(477.51988735,263.50877296)(477.6266372,262.17930118)(477.62662935,260.14837387)
\lineto(477.62662935,254.35614683)
\curveto(477.6266372,252.25836858)(477.74482453,250.73515469)(477.9811917,249.78650057)
\curveto(478.20994887,248.83781808)(478.6064483,248.19646486)(479.1706912,247.86243899)
\curveto(478.6064483,247.50166554)(478.20613637,246.8469508)(477.96975421,245.8982928)
\curveto(477.74101204,244.94961419)(477.6266372,243.43308106)(477.62662935,241.34868885)
\lineto(477.62662935,235.5364195)
\curveto(477.6266372,233.51882922)(477.51988735,232.19603821)(477.30637949,231.56804249)
\curveto(477.08526297,230.94005482)(476.61251364,230.62605898)(475.88813009,230.62605401)
\lineto(475.37344281,230.62605401)
}
}
\piecewiselabels
\end{pspicture}

\caption{A piecewise-smooth function}
\label{fig:piecewise}
\end{figure}

We write the piecewise smooth $\fun$ as the sum $\fun=\smooth+\sing$,
where $\smooth(x)$ is smooth and periodic and $\sing(x)$ is a piecewise
polynomial of degree $\ord$, uniquely determined by $\left\{ \jp_{j}\right\} ,\left\{ \jc_{i,j}\right\} $
such that it ``absorbs'' all the discontinuities of $\fun$ and its
first $\ord$ derivatives. This idea is very old and goes back at
least to A.N.Krylov (\cite{kantokryl62,barkhudaryan2007asymptotic}).
Eckhoff derives the following explicit representation for $\sing(x)$:
\begin{equation}
\begin{split}\sing(x) & =\sum_{j=1}^{\np}\sum_{l=0}^{\ord}\jc_{l,j}V_{l}(x;\jp_{j})\\
V_{n}\left(x;\jp_{j}\right) & =-\frac{\left(2\pi\right)^{n}}{\left(n+1\right)!}B_{n+1}\left(\frac{x-\jp_{j}}{2\pi}\right)\qquad\jp_{j}\leq x\leq\jp_{j}+2\pi\end{split}
\label{eq:sing-part-explicit-bernoulli}\end{equation}
where $V_{n}\left(x;\jp_{j}\right)$ is understood to be periodically
extended to $\left[-\pi,\pi\right]$ and $B_{n}(x)$ is the $n$-th
Bernoulli polynomial. For completeness, let us dervie the formula
for the Fourier coefficients of $\sing(x)$ (it can also be found
in \cite{eckhoff1995arf}).
\begin{lem}
Let $\sing(x)$ be a piecewise polynomial of degree $\ord$, with
jump discontinuities $\left\{ \jp_{j}\right\} _{j=1}^{\np}$ and the
associated jump magnitudes $\left\{ \jc_{l,j}\right\} _{l=0,\dots,\ord}^{j=1,\dots,\np}$.
For definiteness, let us assume that $\fc[0](\sing)=\int_{-\pi}^{\pi}\sing(x)\dd x=0$.
Then\begin{equation}
\fc(\sing)=\frac{1}{2\pi}\sum_{j=1}^{\np}\ee^{-\imath k\jp_{j}}\sum_{l=0}^{\ord}(\imath k)^{-l-1}\jc_{l,j}\label{eq:singular-fourier-explicit}\end{equation}
\end{lem}
\begin{proof}
One integration by parts yields for $k\neq0$:\begin{align*}
\begin{split}\fc(\sing) & =\frac{1}{2\pi}\int_{-\pi}^{\pi}\ee^{-\imath kx}\biggl(\sum_{j=0}^{\np}\sing_{j}(x)\biggr)\dd x\\
 & =\sum_{j=0}^{\np}\Biggl(\frac{\bigl(\sing_{j}(\jp_{j+1}^{-})\ee^{-\imath k\jp_{j+1}}-\sing_{j}(\jp_{j}^{+})\ee^{-\imath k\jp_{j}}\bigr)}{-2\pi\imath k}+\frac{1}{2\pi\imath k}\int_{-\pi}^{\pi}\ee^{-\imath kx}\sing'_{j}(x)\dd x\Biggr)\\
 & =\frac{1}{2\pi\imath k}\sum_{j=0}^{\np}\jc_{0,j}\ee^{-\imath k\jp_{j}}+\frac{1}{\imath k}\fc\biggl(\sum_{j=0}^{\np}\sing'_{j}\biggr)\end{split}
\end{align*}
and so after $\ord+1$-fold repetition we obtain (recall that $\der{\sing_{j}}{\ord+1}\equiv0$):\[
\fc(\sing)=\frac{1}{2\pi}\sum_{l=0}^{\ord}\sum_{j=0}^{\np}(\imath k)^{-l-1}\jc_{l,j}\ee^{-\imath k\jp_{j}}\qedhere\]

\end{proof}
A key observation is that if $\smooth$ is sufficiently smooth, then
the contribution of $\fc(\smooth)$ to $\fc(\fun)$ is negligible
\textbf{for large} $k$. Therefore, for some large enough $\startcoeff$
one can build from the equations \eqref{eq:singular-fourier-explicit}
an approximate system\[
\fc\left(\fun\right)\approx\frac{1}{2\pi}\sum_{j=1}^{\np}\omega_{j}^{k}\sum_{l=0}^{\ord}\frac{\jc_{l,j}}{(\imath k)^{l+1}}\qquad k=\startcoeff,\dots,\startcoeff+\ord+1\]
Here and in the rest of the paper we use the notation $\omega_{j}\isdef\ee^{-\imath\jp_{j}}$.

In fact, this system (up to a change of variables and the number of
equations) lies at the heart of the algebraic reconstruction methods
of Eckhoff \cite{eckhoff1995arf}, Banerjee\&Geer \cite{banerjee1998exponentially}
and Kvernadze \cite{kvernadze2004ajd}. Banerjee\&Geer solve it for
all the parameters at once by least squares minimization. Eckhoff
and Kvernadze eliminate all the $\left\{ \jc_{i,j}\right\} $ first,
resulting in a system of polynomial equations for the $\left\{ \jp_{j}\right\} $,
whose coefficients have nonlinear dependence on the initial data.

In contrast, we propose to solve this system separately for each $\jp=\jp_{j}$,
beacuse this case reduces to a single polynomial equation with respect
to $\jp$. We achieve this ``separation'' by filtering the original
Fourier coefficients such that only the part related to a particular
$\jp_{j}$ remains. This step requires some a-priori knowledge of
the approximate locations of the jumps. Fortunately, such an information
can easily be obtained by a variety of methods - see \prettyref{sec:localization}.

Let us finish this section by presenting the main steps of the reconstruction.
We denote the approximately reconstructed parameters with a tilde
sign. If not stated otherwise, it is understood that these approximations
depend on the index $\startcoeff$. It is important to note that we
distinguish between the actual smoothness of the function $\fun$
and the reconstruction order.
\begin{algorithm}
\label{alg:entire-reconstruction}Let $\fun$ be a piecewise-smooth
function with jumps at $\left\{ \jp_{j}\right\} _{j=1}^{\np}$, continuously
differentiable $\ord_{1}$ times between the jumps. Fix a reconstruction
order to be some nonnegative integer $d\leq\ord_{1}$. Let there be
given the first $\startcoeff+\ord+2$ Fourier coefficients of $\fun$.
\begin{enumerate}
\item Solve the system \eqref{eq:singular-part-system} by localization
(\prettyref{alg:localization-alg}) and resolution (\prettyref{alg:rec-single-jump}).
This will provide us with approximate values for the parameters $\left\{ \nn{\jp_{j}}\right\} $
and $\left\{ \nn{\jc}_{l,j}\right\} $.
\item Calculate the sequence\[
\fc(\widetilde{\sing})=\frac{1}{2\pi}\sum_{j=1}^{\np}\nn{\omega}_{j}^{k}\sum_{l=0}^{d}\frac{\widetilde{\jc}_{l,j}}{(\imath k)^{l+1}}\qquad\left|k\right|\leq\startcoeff\]
and subsequently recover the approximate Fourier coefficients of the
smooth part:\[
\fc(\widetilde{\smooth})\isdef\fc(\fun)-\fc(\widetilde{\sing})\qquad\left|k\right|\leq\startcoeff\]
Take the final approximation to be\begin{equation}
\begin{split}\widetilde{\fun} & =\widetilde{\smooth}+\widetilde{\sing}=\sum_{\left|k\right|\leq\startcoeff}\fc(\nn{\smooth})\ee^{\imath kx}+\sum_{j=1}^{\np}\sum_{l=0}^{d}\nn{\jc}_{l,j}V_{l}(x;\nn{\jp_{j}})\end{split}
\label{eq:final-approximation-explicit}\end{equation}

\end{enumerate}
\end{algorithm}
The rest of this paper is devoted to providing all the details of
the above algorithm and analyzing its accuracy. In particular, we
will seek estimates of the form\begin{equation}
\begin{split}\left|\nn{\jp_{j}}-\jp_{j}\right| & \leq C^{*}\left(\ord,\ord_{1},\np\right)\cdot F_{1}\left(\magbounds,\fcsmoothbound,\geom\right)\cdot\startcoeff^{\alpha\left(\ord,\ord_{1}\right)}\\
\left|\nn{\jc}_{l,j}-\jc_{l,j}\right| & \leq C^{**}\left(\ord,\ord_{1},\np\right)\cdot F_{2}\left(\magbounds,\fcsmoothbound,\geom\right)\cdot\startcoeff^{\beta\left(l,\ord,\ord_{1}\right)}\\
\left|\nn{\fun}\left(x\right)-\fun\left(x\right)\right| & \leq C^{***}\left(d,\ord_{1},\np\right)\cdot F_{3}\left(\magbounds,\fcsmoothbound,\geom\right)\cdot\startcoeff^{\gamma\left(\ord,\ord_{1}\right)}\end{split}
\label{eq:desired-estimates}\end{equation}
where
\begin{itemize}
\item $C^{*},C^{**},C^{***}$ are some absolute constants depending only
on the ``size'' of the problem;
\item $\geom=\geom\left(\jp_{1},\dots,\jp_{\np}\right)$ represents the
geometry of the jump points (such as minimal distance between two
adjacent jumps);
\item $\magbounds=\magbounds\left(\left|\jc_{0,1}\right|,\dots,\left|\jc_{\ord_{1},\np}\right|\right)$
represents some a-priori bounds on the jump magnitudes, such as lower
and upper bounds;
\item $\fcsmoothbound$ is an absolute bound for the Fourier coefficients
of the smooth component $\smooth$:\begin{equation}
\left|\fc\left(\smooth\right)\right|\leq\fcsmoothbound\cdot k^{-\ord-2}\label{eq:fcsmoothbound}\end{equation}

\item $F_{1},F_{2},F_{3}$ and $\alpha,\beta,\gamma$ are some ``simple''
functions.
\end{itemize}
In the course of our investigation we shall be defining more specific
bounds, but it will always be assumed that those can be expressed
in terms of the above quantities.

Since we are interested in ``asymptotic'' estimates, we will in general
allow the inequalities \eqref{eq:desired-estimates} to hold for all
$\startcoeff$ starting from some index $K^{*}$ which may be large
and depend on the parameters of the problem. However, if a particular
bound holds for all $k>K^{*}$ then it will in general hold for $k=1,2,\dots,K^{*}$
as well, with some larger multiplicative constants $\nn{C^{*}}\dots$,
but which are harder to compute explicitly.

\section{\label{sec:single-jump-reconstruction}Resolving a single jump}

Let $\omega\isdef\ee^{-\imath\jp}$. The goal is to recover $\sing$
from the approximate system of equations\begin{equation}
\fc(\fun)\approx\frac{\omega^{k}}{2\pi}\sum_{l=0}^{\ord}\frac{\jc_{l}}{(\imath k)^{l+1}}\left(=\fc(\sing)\right)\qquad k=\startcoeff,\dots,\startcoeff+\ord+1\label{eq:singular-part-system}\end{equation}

To find $\omega$, we eliminate $\left\{ \jc_{0},\dotsc,\jc_{\ord}\right\} $
from the equations. The result is a single polynomial equation having
the exact value $\omega$ as one of its solutions. In Eckhoff's paper,
this elimination is described in great detail, while here we present
only the end result.

Let\begin{eqnarray}
\begin{split}m_{k} & \isdef2\pi(\imath k)^{d+1}\fc(\sing)=\omega^{k}\sum_{l=0}^{d}(\imath k)^{d-l}\jc_{l}\\
p_{k}^{\ord}(z) & \isdef\sum_{j=0}^{d+1}\left(-1\right)^{j}\binom{d+1}{j}m_{k+j}z^{d+1-j}\end{split}
\label{eq:mk-sequence-def}\end{eqnarray}

\[
\]

\begin{lem}
\label{lem:omega-exact-root}The point $\omega$ satisfies:\[
p_{k}^{\ord}(\omega)=0\qquad\forall k\in\naturals\]
\end{lem}
\begin{proof}
The proof is an immediate consequence of \prettyref{lem:basic-recurrence}
(see \prettyref{app:difference-calculus-and-related-results}).
\end{proof}
Since the exact coefficients $m_{k}$ (and as a result the polynomials
$p_{k}^{d}\left(z\right)$) are unknown, we approximate these with
the known quantities\begin{equation}
\begin{split}r_{k} & \isdef2\pi\left(\imath k\right)^{d+1}\fc(\fun)\\
q_{k}^{d}(z) & \isdef\sum_{j=0}^{d+1}(-1)^{j}\binom{d+1}{j}r_{k+j}z^{d+1-j}\end{split}
\label{eq:q-def}\end{equation}

Now we are ready to formulate the procedure of recovering the parameters
of a single jump.
\begin{algorithm}
\label{alg:rec-single-jump}Let us be given the first $\startcoeff+\ord+2$
Fourier coefficients of the function $\fun$ which has a single jump
$\jp\in[-\pi,\pi]$.
\begin{enumerate}
\item Solve the polynomial equation\[
q_{\startcoeff}^{d}(z)=0\]
and take $\nn{\omega}$ to be the root \emph{which is closest to the
unit circle.} In \prettyref{sec:accuracy-single-jump} below, we shall
provide the justification for this choice.
\item The jump magnitudes $\jc_{0},\dotsc,\jc_{d}$ are reconstructed as
follows. By \eqref{eq:mk-sequence-def}, the exact values of ${\jc_{j}}$
satisfy \begin{equation}
m_{k}\omega^{-k}=\sum_{l=0}^{d}(\imath k)^{d-l}\jc_{l}\qquad\forall k\in\naturals\label{eq:exact-system-jump-magnitudes}\end{equation}
We use the approximations $r_{k}\approx m_{k}$, $\widetilde{\omega}\approx\omega$
and solve the system of linear equations \begin{equation}
r_{k}\widetilde{\omega}^{-k}=\sum_{l=0}^{d}(\imath k)^{d-l}\widetilde{\jc_{l}}\qquad k=\startcoeff,\dotsc,\startcoeff+\ord\label{eq:approximate-system-jump-magnitudes}\end{equation}
with respect to the unknowns $\left\{ \nn{\jc}_{l}\right\} $ by any
one of the standard methods.
\end{enumerate}
\end{algorithm}

\section{\label{sec:accuracy-single-jump}Accuracy analysis: a single jump}

Our goal in this section is to analyze \prettyref{alg:rec-single-jump}
and calculate its accuracy. We shall express all our estimates in
terms of the index $k$, keeping in mind that it should be replaced
with $\startcoeff$ to be consistent with the definitions of the previous
sections.

\subsection{Accuracy analysis: jump location}

\global\long\def\uz{\mathcal{T}}

We start with the determination of the jump point $\nn{\omega}$.
Our strategy will be to investigate the polynomial $q_{k}^{d}(z)$,
and determine the bounds on locations of its roots. We can informally
summarize the main results as follows:
\begin{enumerate}
\item Starting from some $k$, the roots of $q_{k}^{d}(z)$ are {}``separated''
from each other by at least $\sim k^{-1}$.
\item If the function $\fun$ is continuously differentiable at least $d_{1}\geq2\ord+1$
times everywhere except at $\jp$, then one of those roots deviates
from the ``true'' value $\omega$ by at most $\sim k^{-\ord-2}$.
\end{enumerate}
We regard $q_{k}^{d}(z)$ as a perturbation of $p_{k}^{d}(z)$. With
this point of view, we shall first describe the roots of $p_{k}^{d}(z)$,
and then calculate the {}``deviations'' due to the difference\[
e_{k}^{d}(z)=q_{k}^{d}(z)-p_{k}^{d}(z)\]

In the subsequent analysis we denote the roots of $p_{k}^{d}(z)$
by $z_{i}^{(k,d)}$ for $i=0,1,\dots,\ord$, with the convention that
$z_{0}^{(k,d)}=\omega$. Also, we denote the roots of $q_{k}^{d}(z)$
by $\kappa_{i}^{(k,d)}$.

It will be convenient to study $p_{k}^{d}\left(z\right)$ in a different
coordinate system. For this purpose, consider the following transformation
of the punctured $z$-plane:\[
u=\uz\left(z\right)=\frac{\omega}{z}-1\qquad z\neq0\]
Then the inverse map is given by

\begin{equation}
z=\uz^{-1}\left(u\right)=\frac{\omega}{u+1}\qquad u\neq-1\label{eq:u-z-map}\end{equation}
Now we translate the problem into the $u$-plane.
\begin{defn}
For all $k,d\in\naturals$ let\begin{equation}
s_{k}^{d}(u)\isdef\frac{p_{k}^{d}(z)}{\omega^{k}z^{\ord+1}}=\frac{(u+1)^{d+1}}{\omega^{k+d+1}}p_{k}^{d}\left(\frac{\omega}{u+1}\right)\label{eq:eliminant-shifted}\end{equation}
\end{defn}
\begin{claim}
\label{clm:skd-pkd-corr}$s_{k}^{d}(u)$ is a polynomial function.
Furthermore, if $u_{0}\neq-1$ is a root of $s_{k}^{d}(u)$, then
$z_{0}=\uz^{-1}\left(u_{0}\right)$ is a root of $p_{k}^{d}(u)$.
\end{claim}
Therefore it makes sense to study the roots of $s_{k}^{d}(u)$. We
denote these roots by $\sigma_{i}^{(k,d)},i=0,\dots,\ord$. The observation
below is an immediate consequence of \prettyref{lem:omega-exact-root}.
\begin{claim}
$s_{k}^{d}(0)=0$
\end{claim}
Therefore we will always take $\sigma_{0}^{(k,d)}=0.$

In what follows, we shall break $s_{k}^{d}(u)$ into a sum of terms
and subsequently apply a perturbation analysis to determine its roots.
We begin with some simplifications: \[
\begin{split}s_{k}^{d}(u) & =\frac{(u+1)^{d+1}}{\omega^{k+d+1}}\sum_{j=0}^{d+1}\left(-1\right)^{j}\binom{d+1}{j}\underbrace{\left\{ \omega^{k+j}\sum_{l=0}^{d}\left(\imath(k+j)\right)^{d-l}\jc_{l}\right\} }_{=m_{k+j}}\frac{\omega^{d+1-j}}{(u+1)^{d+1-j}}\\
 & =\sum_{j=0}^{d+1}\left(-1\right)^{j}\binom{d+1}{j}\left(u+1\right)^{j}\sum_{l=0}^{d}\left(\imath(k+j)\right)^{d-l}\jc_{l}=\sum_{l=0}^{d}\imath^{d-l}\jc_{l}\sum_{j=0}^{d+1}\left(-1\right)^{j}\binom{d+1}{j}\left(u+1\right)^{j}\left(k+j\right)^{d-l}\end{split}
\]
Now substitute the binomial expansions\[
\begin{split}(k+j)^{d-l} & =\sum_{m=0}^{d-l}k^{m}j^{d-l-m}\binom{d-l}{m}\\
(u+1)^{j} & =\sum_{s=0}^{j}u^{s}\binom{j}{s}\end{split}
\]
and get\[
\begin{split}s_{k}^{d}(u) & =\end{split}
\sum_{l=0}^{d}\imath^{d-l}\jc_{l}\sum_{m=0}^{d-l}k^{m}\sum_{j=0}^{d+1}\left(-1\right)^{j}\binom{d+1}{j}\sum_{s=0}^{j}u^{s}\binom{j}{s}j^{d-l-m}\binom{d-l}{m}\]
Now we make a change in indexing according to the following scheme:\[
\sum_{j=0}^{\ord+1}\sum_{s\-}^{j}=\sum_{s=0}^{\ord+1}\sum_{j=s}^{\ord+1}\qquad\sum_{l=0}^{\ord}\sum_{m=0}^{\ord-1}=\sum_{m=0}^{\ord}\sum_{l=0}^{\ord-m}\]
and continue:\[
\begin{split}s_{k}^{d}(u) & =\sum_{s=0}^{d+1}u^{s}\sum_{m=0}^{d}k^{m}\sum_{l=0}^{d-m}\binom{d-l}{m}\imath^{d-l}\jc_{l}\sum_{j=s}^{d+1}(-1)^{j}\binom{j}{s}\binom{d+1}{j}j^{d-l-m}\end{split}
\]

\begin{defn}
For all integers $t,s$ with $0\leq t\leq d$ and $0\leq s\leq d+1$
let \[
\dimdim(d,t,s)\isdef\sum_{j=s}^{d+1}(-1)^{j}\binom{j}{s}\binom{d+1}{j}j^{d-t}\]

\end{defn}
With this definition, we can write\[
s_{k}^{d}(u)=\sum_{s=0}^{d+1}u^{s}\sum_{m=0}^{d}k^{m}\sum_{l=0}^{d-m}\binom{d-l}{m}\imath^{d-l}\jc_{l}\cdot\dimdim(d,m+l,s)\]

We will need a technical result.
\begin{lem*}[\ref{lem:dimdim-special-values}]
For all $0\leq s\leq d+1$
\begin{enumerate}
\item If $m+l\geq s$ then $\dimdim(d,m+l,s)=0$
\item $\dimdim(d,s-1,s)=(-1)^{d+1}(d+1-s)!\binom{d+1}{s}$
\end{enumerate}
\end{lem*}
\begin{proof}
See \prettyref{app:difference-calculus-and-related-results}.
\end{proof}
The polynomial $s_{k}^{d}(u)$ must therefore be of the form \[
s_{k}^{d}(u)=\sum_{s=1}^{d+1}\left(a_{0,s}+a_{1,s}k+\dots+a_{s-1,s}k^{s-1}\right)u^{s}\]
where in particular\begin{equation}
a_{s-1,s}=\binom{d}{s-1}\imath^{d}\jc_{0}(-1)^{d+1}(d+1-s)!\binom{d+1}{s}\label{eq:first-perturbation-poly-coeffs}\end{equation}

We break up the polynomial $s_{k}^{d}(u)$ into a ``dominant'' and
a ``perturbation'' part: $s_{k}^{d}(u)=\nn{s_{k}^{d}}(u)+h_{k}^{d}(u)$
where\begin{equation}
\begin{split}\nn{s_{k}^{d}}(u) & \isdef\sum_{s=1}^{d+1}a_{s-1,s}k^{s-1}u^{s}\\
h_{k}^{d}(u) & \isdef\sum_{s=2}^{d+1}\left(a_{0,s}+a_{1,s}k+\dots+a_{s-2,s}k^{s-2}\right)u^{s}\end{split}
\label{eq:first-perturbation-decomposition}\end{equation}

Next we shall see that the dominant component $\nn{s_{k}^{d}}(u)$
determines the locations of the roots of $s_{k}^{d}(u)$ up to the
first order accuracy, while the other component $h_{k}^{d}(u)$ is
responsible for second-order perturbations of these roots.

We denote the roots of $\nn s_{k}^{d}(u)$ by $\nn{\sigma}_{i}^{(k,d)},i=0,\dots,\ord$
with $\nn{\sigma}_{0}^{(k,d)}=0$.

It turns out that $\nn{s_{k}^{d}}(u)$ can be completely characterized.
\begin{defn}
\label{def:laguerre}For every $\alpha>-1$ and $n=0,1,2,\dots$ let
$\lag{\alpha}n(x)$ denote the generalized Laguerre polynomial (\cite[Chapter 22]{abramowitz1965handbook},
\cite[Chapter 5.2]{szegHo1975op}):\[
\lag{\alpha}n(x)=\sum_{m=0}^{n}\binom{n+\alpha}{n-m}\frac{(-x)^{m}}{m!}\]
\end{defn}
\begin{lem}
With the above notations:
\begin{enumerate}
\item The polynomial $\nn{s_{k}^{d}}(u)$ satisfies\begin{equation}
\nn{s_{k}^{d}}(u)=\frac{1}{k}\nn{s_{1}^{d}}(ku)\label{eq:dominant-component-scaling}\end{equation}

\item Furthermore,\[
\nn{s_{1}^{d}}(u)=-(-\imath)^{d}\jc_{0}(d+1)!\lag{-1}{d+1}(-u)\]

\end{enumerate}
\end{lem}
\begin{proof}
The first part follows from \eqref{eq:first-perturbation-decomposition}:\[
\begin{split}k\cdot\nn{s_{k}^{d}}(u) & =k\cdot\sum_{s=1}^{d+1}a_{s-1,s}k^{s-1}u^{s}=\sum_{s=1}^{d+1}a_{s-1,s}\left(ku\right)^{s}=\nn s_{1}^{d}(ku)\end{split}
\]

To prove the second part, we substitute the expression \eqref{eq:first-perturbation-poly-coeffs}
into \eqref{eq:first-perturbation-decomposition}:\[
\begin{split}\nn{s_{1}^{d}}(u) & =\sum_{s=1}^{d+1}a_{s-1,s}u^{s}=\sum_{s=1}^{d+1}\binom{d}{s-1}\imath^{d}\jc_{0}(-1)^{d+1}(d+1-s)!\binom{d+1}{s}u^{s}\\
 & =-(-\imath)^{d}\jc_{0}(d+1)!\sum_{s=1}^{d+1}\binom{d}{d+1-s}\frac{u^{s}}{s!}=-(-\imath)^{d}\jc_{0}(d+1)!\lag{-1}{d+1}(-u)\qedhere\end{split}
\]
\end{proof}
\begin{cor}
\label{cor:skd-laguerre-roots-correspondence}For all $k\in\naturals$,
$\nn{s_{k}^{d}}(u^{*})=0$ if and only if $\lag{-1}{d+1}(-ku^{*})=0$.\end{cor}
\begin{lem}
\label{lem:skd-roots-characterization} The numbers $\left\{ \nn{\sigma}_{i}^{(k,d)}\right\} $
satisfy the following properties:
\begin{enumerate}
\item \label{enu:lag-p1}each $\nn{\sigma}_{i}^{(k,d)}$ is a simple root
of $\nn s_{k}^{d}\left(u\right)$;
\item \label{enu:lag-p2}$\nn{\sigma}_{i}^{(k,d)}<0$ for $i=1,2,\dots,\ord$;
\item \label{enu:lag-p3}there exist constants $\Cl{s-roots-lower},\Cl{s-roots-upper}$
such that for every $k\in\naturals$ and $0\leq i<j\leq\ord$ \begin{equation}
\Cr{s-roots-lower}k^{-1}\leq\left|\nn{\sigma}_{i}^{(k,d)}-\nn{\sigma}_{j}^{(k,d)}\right|\leq\Cr{s-roots-upper}k^{-1}\label{eq:sroots-bounds}\end{equation}

\end{enumerate}
\end{lem}
\begin{proof}
Following \prettyref{cor:skd-laguerre-roots-correspondence}, we only
need to characterize the roots of $\lag{-1}{d+1}(-u)$. By \cite[Chapter 5.2]{szegHo1975op},
for every integer $m\geq1$\[
\lag{-m}n(x)=(-x)^{m}\frac{(n-m)!}{n!}\lag m{n-k}(x)\]
and therefore\[
\lag{-1}{d+1}(-u)=u\frac{d!}{(d+1)!}\lag 1d(-u)\]
The polynomials $\left\{ \lag 1n(x)\right\} _{n=0}^{\infty}$ form
an orthogonal system on the interval $\left(0,\infty\right)$ (see
again \cite[Chapter 22]{abramowitz1965handbook}, \cite[Chapter 5.2]{szegHo1975op}).
Parts \eqref{enu:lag-p1} and \eqref{enu:lag-p2} follow immediately.
Part \eqref{enu:lag-p3} follows by taking $\Cr{s-roots-lower}$ and
$\Cr{s-roots-upper}$ to be the minimal and the maximal distance between
the roots of $\lag 1d(u)$, correspondingly.
\end{proof}
Now we show that $h_{k}^{d}(u)$ perturbes the zeros of $\nn{s_{k}^{d}}(u)$
by at most $\sim k^{-2}$. Since the coefficients $h_{k}^{d}\left(u\right)$
depend linearly on $\jc_{0}\dots,\jc_{\ord}$, we can expect that
the bound will depend on the quantity $\sum_{l=0}^{\ord}\left|\jc_{l}\right|$.
For convenience, let us therefore define\[
\jcbound\isdef\max\left(1,\sum_{l=0}^{\ord}\left|\jc_{l}\right|\right)\]

\begin{lem}
\label{lem:skd-roots-deviation}There exist constants $\Cl{skd-deviation},\Cl[kk]{kk-skd-deviation}$
such that for all $k>\Cr{kk-skd-deviation}\jcbound$ and for all $i=0,\dots,\ord$\begin{equation}
\left|\nn{\sigma}_{i}^{(k,d)}-\sigma_{i}^{(k,d)}\right|\leq\Cr{skd-deviation}\jcbound k^{-2}\label{eq:skd-roots-deviation}\end{equation}
\end{lem}
\begin{proof}
Our method of proof is based on Rouche's theorem (\prettyref{thm:rouche}).
We shall define a sequence $\rho(k)=\Cr{skd-deviation}\jcbound k^{-2}$
(where $\Cr{skd-deviation}$ is to be determined) and consider disks
of radius $\rho(k)$ around each one of the roots $\nn{\sigma}_{i}^{(k,d)}$.
Our goal is to find $\Cr{skd-deviation}$ so that $\left|\nn{s_{k}^{d}}(u_{\theta})\right|>\left|h_{k}^{d}\left(u_{\theta}\right)\right|$
for all points $u_{\theta}=\nn{\sigma}_{i}^{(k,d)}+\rho(k)\ee^{\imath\theta}$
on the boundaries of these disks.
\begin{itemize}
\item In order to bound $\left|\nn{s_{k}^{d}}(u_{\theta})\right|$ from
below, we shall use \prettyref{lem:lower-bound-values-circle-via-derivatives}.
We need to bound from below the first derivative at $\nn{\sigma}_{i}^{(k,d)}$,
as well as to bound from above the second derivative in the disk $B_{k^{-1}}\left(\nn{\sigma}_{i}^{(k,d)}\right)$.

\begin{enumerate}
\item We always have\[
\frac{\dd}{\dd u}\nn{s_{k}^{d}}(u)\Big|_{u=\nn{\sigma}_{i}^{(k,d)}}=\frac{\dd}{\dd u}\left(\frac{1}{k}\nn{s_{1}^{d}}(ku)\right)\Big|_{u=\nn{\sigma}_{i}^{(k,d)}}=\frac{\dd\nn s_{1}^{d}(w)}{\dd w}\Big|_{w=k\nn{\sigma}_{i}^{(k,d)}}\]
Now $k\nn{\sigma}_{i}^{(k,d)}$ is always a root of $\nn s_{1}^{d}(u)$,
therefore the value of $\frac{\dd}{\dd u}\nn{s_{k}^{d}}(u)\Big|_{u=\nn{\sigma}_{i}^{(k,d)}}$
is independent of $k$ and thus we can write\[
\left|\frac{\dd}{\dd u}\nn{s_{k}^{d}}(u)\Big|_{u=\nn{\sigma}_{i}^{(k,d)}}\right|\geq\Cl{uniform-bound-first-der}\isdef\min_{i}\left|\frac{\dd}{\dd u}\nn s_{1}^{d}(u)\Big|_{u=\nn{\sigma}_{i}^{(1,d)}}\right|\]
Since all the roots are simple, this is guaranteed to be a strictly
positive bound.
\item Now consider a point $u^{*}\in B_{k^{-1}}\left(\nn{\sigma}_{i}^{(k,d)}\right)$.
Then $\left|ku^{*}-k\nn{\sigma}_{i}^{(k,d)}\right|\leq1$ and therefore
$ku^{*}\in B_{1}\left(\nn{\sigma}_{i}^{(1,d)}\right)$. Using \eqref{eq:dominant-component-scaling}
and differentiating twice, we get\[
\frac{\dd^{2}}{\dd u^{2}}\nn s_{k}^{d}\left(u\right)\Big|_{u=u^{*}}=k\frac{\dd^{2}}{\dd w^{2}}\nn s_{1}^{d}\left(w\right)\Big|_{w=ku^{*}}\]
Let\[
\Cl{uniform-bound-second-der}\isdef\max_{i}\max_{w^{*}\in B_{1}\left(\nn{\sigma}_{i}^{(1,d)}\right)}\left|\frac{\dd^{2}}{\dd w^{2}}\nn s_{1}^{d}\left(w\right)\Big|_{w=w^{*}}\right|\]

\item The constants $\Cr{uniform-bound-first-der}$ and $\Cr{uniform-bound-second-der}$
therefore satisfy the assumptions of \prettyref{lem:lower-bound-values-circle-via-derivatives}.
We define $\Cl{skd-radius}\isdef\min\left(1,\frac{\Cr{uniform-bound-first-der}}{\Cr{uniform-bound-second-der}}\right)$.
The conclusion is that there exists a constant $\Cl{skd-below}$ such
that for every function $\eta(k):\naturals\to\reals$ satisfying $0<\eta(k)<\frac{\Cr{skd-radius}}{k}$
we have\[
\left|\nn s_{k}^{d}\left(\nn{\sigma}_{i}^{(k,d)}+\eta(k)\ee^{\imath\theta}\right)\right|>\Cr{skd-below}\eta(k)\]

\end{enumerate}
\item Now we shall bound $\left|h_{k}^{d}\left(u_{\theta}\right)\right|$
from above. Recall that\[
h_{k}^{d}(u)=\sum_{s=2}^{d+1}\left(a_{0,s}+a_{1,s}k+\dots+a_{s-2,s}k^{s-2}\right)u^{s}\]
where $a_{i,j}$ are some linear functions of $\jc_{0},\dots,\jc_{\ord}.$
Let $\zeta(k):\naturals\to\reals$ be any function satisfying $0<\zeta(k)<\frac{1}{k}$,
and consider $u_{\theta}=\nn{\sigma}_{i}^{(k,d)}+\zeta(k)\ee^{\imath\theta}$.
By \prettyref{lem:skd-roots-characterization}, $\left|\nn{\sigma}_{i}^{(k,d)}\right|<\Cr{s-roots-upper}k^{-1}$
and therefore $\left|u_{\theta}\right|<2\cdot\max\left(1,\Cr{s-roots-upper}\right)k^{-1}.$
But then \[
\begin{split}\left|h_{k}^{d}\left(u_{\theta}\right)\right| & \leq\left|a_{0,2}\right|\left|u_{\theta}\right|^{2}+\left(\left|a_{0,3}\right|+\left|a_{1,3}\right|k\right)\left|u_{\theta}\right|^{3}+\dots+\left(\left|a_{0,\ord+1}\right|+\dots+\left|a_{\ord-1,\ord+1}\right|k^{\ord-1}\right)\left|u_{\theta}\right|^{\ord+1}\leq\Cl{hkd-bound-above}\jcbound k^{-2}\end{split}
\]
for some constant $\Cr{hkd-bound-above}$.
\end{itemize}
We set\[
\Cr{skd-deviation}\isdef\frac{2\Cr{hkd-bound-above}}{\Cr{skd-below}}\]
and let $\rho(k)=\Cr{skd-deviation}\jcbound k^{-2}$. We need the
inequality $\rho(k)<\frac{\Cr{skd-radius}}{k}$ to be satisfied, and
this is obviously possible if\[
k>\underbrace{\frac{2\Cr{hkd-bound-above}}{\Cr{skd-below}\Cr{skd-radius}}}_{\isdef\Cr{kk-skd-deviation}}\jcbound\]
In this case we have shown that\[
\left|\nn s_{k}^{d}\left(\nn{\sigma}_{i}^{(k,d)}+\rho(k)\ee^{\imath\theta}\right)\right|>\Cr{skd-below}\rho(k)=2\Cr{hkd-bound-above}\jcbound k^{-2}\]
and also\[
\left|h_{k}^{d}\left(\nn{\sigma}_{i}^{(k,d)}+\rho(k)\ee^{\imath\theta}\right)\right|\leq\Cr{hkd-bound-above}\jcbound k^{-2}\]
Therefore\[
\left|\nn{s_{k}^{d}}\left(\nn{\sigma}_{i}^{(k,d)}+\rho(k)\ee^{\imath\theta}\right)\right|>\left|h_{k}^{d}\left(\nn{\sigma}_{i}^{(k,d)}+\rho(k)\ee^{\imath\theta}\right)\right|\]
which completes the proof. \qedhere

\end{proof}
\begin{rem}
We have in fact shown that for each $k>\Cr{kk-skd-deviation}$ the
polynomial $s_{k}^{d}(u)$ has precisely $\ord+1$ distinct roots.
\end{rem}
Now we can go back to the original polynomial $p_{k}^{d}(z)$ and
accurately describe the location of its roots $\left\{ z_{i}^{\left(k,d\right)}\right\} $.
Recall from \prettyref{clm:skd-pkd-corr} that $z_{i}^{\left(k,d\right)}=\uz^{-1}\left(\sigma_{i}^{\left(k,d\right)}\right)$.
Being careful to avoid the singularity $\sigma_{i}^{\left(k,\ord\right)}=-1$
(by choosing large enough $k$), we now show that the geometry of
the roots $\sigma_{i}^{\left(k,d\right)}$ is preserved under $\uz^{-1}$.
In particular, the numbers $z_{i}^{\left(k,d\right)}$ remain separated
from each other (following \eqref{eq:sroots-bounds}), each of them
being close (following \prettyref{lem:skd-roots-deviation}) to one
of the numbers

\begin{equation}
y_{i}^{(k,d)}\isdef\uz^{-1}\left(\nn{\sigma}_{i}^{\left(k,d\right)}\right)=\frac{\omega}{\nn{\sigma}_{i}^{(k,d)}+1}\label{eq:ykd-def}\end{equation}
The only thing which is different are the constants.
\begin{lem}
\label{lem:first-perturbation-characterization} Let $y_{i}^{\left(k,d\right)}$
be defined by \eqref{eq:ykd-def}. Then
\begin{enumerate}
\item \label{enu:first-perturb-i1}there exist constants $\Cl{rootlower},\Cl{rootupper},\Cl[kk]{kk-rooty}$
such that for all $k>\Cr{kk-rooty}$ and $0\leq i<j\leq\ord$\[
\Cr{rootlower}k^{-1}\leq\left|y_{i}^{(k,d)}-y_{j}^{(k,d)}\right|\leq\Cr{rootupper}k^{-1}\]

\item \label{enu:first-perturb-i2}there exist constants $\Cl{firstdev},\Cl[kk]{kk-firstdev}$
such that for all $k>\Cr{kk-firstdev}\jcbound$\[
\left|z_{i}^{(k,d)}-y_{i}^{(k,d)}\right|<\Cr{firstdev}\jcbound\cdot k^{-2}\]

\item \label{enu:first-perturb-i3}there exist constants $\Cl{finalfirstdevlower},\Cl{finalfirstdevupper},\Cl[kk]{kk-finalfirstdev}$
such that for all $k>\Cr{kk-finalfirstdev}\jcbound$ and $0\leq i<j\leq\ord$\[
\Cr{finalfirstdevlower}k^{-1}\leq\left|z_{i}^{(k,d)}-z_{j}^{(k,d)}\right|\leq\Cr{finalfirstdevupper}k^{-1}\]

\end{enumerate}
\end{lem}
\begin{proof}
If $k>2\Cr{s-roots-upper}$ then $\left|\nn{\sigma}_{i}^{(k,d)}\right|<\frac{1}{2}$
(see \eqref{eq:sroots-bounds}). It follows that $\frac{1}{2}<\left|\nn{\sigma}_{i}^{(k,d)}+1\right|\leq1$
and so by \eqref{eq:ykd-def} \[
\Cr{s-roots-lower}k^{-1}<\left|y_{i}^{(k,d)}-y_{j}^{(k,d)}\right|<4\Cr{s-roots-upper}k^{-1}\]
This proves \eqref{enu:first-perturb-i1} with $\Cr{rootlower}=\Cr{s-roots-lower}$,
$\Cr{rootupper}=4\Cr{s-roots-upper}$ and $\Cr{kk-rooty}=2\Cr{s-roots-upper}$.

If in addition $k>\frac{2\Cr{skd-deviation}}{\Cr{s-roots-lower}}$
then $\left|\nn{\sigma}_{i}^{(k,d)}-\sigma_{i}^{(k,d)}\right|<\frac{\left|\nn{\sigma}_{i}^{(k,d)}\right|}{2}<\frac{1}{4}$
and therefore $\left|\sigma_{i}^{(k,d)}+1\right|>\frac{1}{4}$. It
follows from \eqref{eq:skd-roots-deviation} that\[
\left|z_{i}^{(k,d)}-y_{i}^{(k,d)}\right|=\frac{\left|\nn{\sigma}_{i}^{(k,d)}-\sigma_{i}^{(k,d)}\right|}{\left|\nn{\sigma}_{i}^{(k,d)}+1\right|\left|\sigma_{i}^{(k,d)}+1\right|}<4\Cr{skd-deviation}k^{-2}\qquad k>\underbrace{\max\left(2\Cr{s-roots-upper},\frac{2\Cr{skd-deviation}}{\Cr{s-roots-lower}},\Cr{kk-skd-deviation}\right)}_{\isdef\Cr{kk-firstdev}}\jcbound\]
and this proves \eqref{enu:first-perturb-i2} with $\Cr{firstdev}=4\Cr{skd-deviation}$
and $\Cr{kk-firstdev}$ as above.

Let $k>\underbrace{\max\left(\Cr{kk-rooty},\Cr{kk-firstdev},\frac{4\Cr{firstdev}}{\Cr{rootlower}}\right)}_{\isdef\Cr{kk-finalfirstdev}}\jcbound$.
Using \eqref{enu:first-perturb-i1} and \eqref{enu:first-perturb-i2},
we have one one hand\[
\begin{split}\left|z_{i}^{(k,d)}-z_{j}^{(k,d)}\right| & <\left|y_{i}^{(k,d)}-y_{j}^{(k,d)}\right|+\left|z_{i}^{(k,d)}-y_{i}^{(k,d)}\right|+\left|y_{j}^{(k,d)}-z_{j}^{(k,d)}\right|\\
 & <\left|y_{i}^{(k,d)}-y_{j}^{(k,d)}\right|+\frac{\Cr{firstdev}}{k^{2}}+\frac{\Cr{firstdev}}{k^{2}}\\
 & <\left|y_{i}^{(k,d)}-y_{j}^{(k,d)}\right|+2\cdot\frac{\Cr{rootlower}}{4k}<\frac{3}{2}\left|y_{i}^{(k,d)}-y_{j}^{(k,d)}\right|<\underbrace{\frac{3}{2}\Cr{rootupper}}_{\isdef\Cr{finalfirstdevupper}}k^{-1}\end{split}
\]
and on the other hand\[
\left|z_{i}^{(k,d)}-z_{j}^{(k,d)}\right|>\left|y_{i}^{(k,d)}-y_{j}^{(k,d)}\right|-\left|y_{j}^{(k,d)}-z_{j}^{(k,d)}\right|-\left|y_{j}^{(k,d)}-z_{j}^{(k,d)}\right|>\frac{1}{2}\left|y_{i}^{(k,d)}-y_{j}^{(k,d)}\right|>\underbrace{\frac{1}{2}\Cr{rootlower}}_{\isdef\Cr{finalfirstdevlower}}k^{-1}\]
That proves \eqref{enu:first-perturb-i3}.
\end{proof}
\global\long\def\fsbb{\fcsmoothbound^{*}}
\global\long\def\jccbound{\jc^{**}}
\global\long\def\sab{H}

Remaining in the $z$-plane, we now turn to investigate $q_{k}^{d}(z)$
and its roots $\left\{ \kappa_{i}^{\left(k,d\right)}\right\} $. Recall
that we consider $q_{k}^{d}\left(z\right)$ to be a ``perturbation''
of $p_{k}^{d}(z)$ by another polynomial $e_{k}^{d}(z)$, i.e.\[
q_{k}^{d}(z)=p_{k}^{d}(z)+e_{k}^{d}(z)\]

The coefficients of $e_{k}^{d}(z)$ depend on the Fourier coefficients
of the ``smooth part'' of our pieceiwise-smooth function $\fun$.
It turns out that in the general setting, the coefficients of $e_{k}^{d}(z)$
are large compared to those of $p_{k}^{d}\left(z\right)$ and therefore
the perturbations of the roots are large too. If, however, there is
enough structure in those coefficients due to additional orders of
smoothness, then the perturbation of the roots is small. This is the
essense of the key \prettyref{lem:qkd-root-deviation} below.

Recall that $\fun$ has in fact $\ord_{1}\geq\ord$ continuous derivatives
everywhere in $[-\pi,\pi]\setminus\left\{ \jp\right\} $, and denote
the additional jump magnitudes at $\jp$ by $\jc_{\ord+1},\dots,\jc_{\ord_{1}}$.
For every $l\leq d_{1}$, let $\sing_{l}$ denote the piecewise polynomial
of degree $l$ with jump point $\jp$ and jump magnitudes $\jc_{0},\dotsc,\jc_{l}$.
Then we write

\begin{equation}
\fun=\sing_{d}+\left(\sing_{d_{1}}-\sing_{d}\right)+\smooth^{*}\label{eq:decomp-additional-smoothness}\end{equation}
where $\smooth^{*}$ is $\ord_{1}$-times smooth everywhere in $\left[-\pi,\pi\right]$.
Thus there exists a constant $\fsbb$ such that \begin{equation}
\left|\fc\left(\smooth^{*}\right)\right|\leq\fsbb k^{-\ord_{1}-2}\label{eq:add-smooth-fc-bound}\end{equation}
Let us also denote\[
\begin{split}\jccbound & \isdef\max\left(1,\sum_{l=\ord+1}^{\ord_{1}}\left|\jc_{l}\right|\right)\\
\sab & \isdef\left(\jcbound+\jccbound+\fsbb\right)\end{split}
\]

\begin{lem}
\label{lem:qkd-root-deviation}Let $d_{1}\geq2\ord+1$. Then there
exist constants $\Cl{qkd-deviation-reg},\Cl{qkd-deviation-omega},\Cl[kk]{kk-qkd-deviation}$
such that for all $k>\Cr{kk-qkd-deviation}\sab$\[
\left|\kappa_{i}^{(k,d)}-z_{i}^{(k,d)}\right|\leq\begin{cases}
\Cr{qkd-deviation-reg}\cdot\sab\cdot k^{-2} & i=1,2,\dots,\ord\\
\Cr{qkd-deviation-omega}\cdot\sab\cdot k^{-\ord-2} & i=0\end{cases}\]
\end{lem}
\begin{proof}
The idea of the proof is the same as in \prettyref{lem:skd-roots-deviation}.
Namely, we shall seek the constants $\Cr{qkd-deviation-reg},\Cr{qkd-deviation-omega}$
and $\Cr{kk-qkd-deviation}$ such that if $\rho_{1}(k)=\Cr{qkd-deviation-reg}\cdot\sab\cdot k^{-2}$
and $\rho_{2}(k)=\Cr{qkd-deviation-omega}\cdot\sab\cdot k^{-\ord-2}$
then for $i=0,1,\dots,\ord$ and $k>\Cr{kk-qkd-deviation}\sab$ there
exist neighborhoods $D_{i}^{(k)}$ of $z_{i}^{(k,d)}$ such that $\left|p_{k}^{d}(z)\right|>\left|e_{k}^{d}(z)\right|$
on the boundary of $D_{i}^{(k)}$ and
\begin{itemize}
\item $\diam D_{i}^{(k)}=\rho_{1}(k)$ for $i=1,2,\dots,\ord$; 
\item $\diam D_{0}^{(k)}=\rho_{2}(k)$.\end{itemize}
\begin{enumerate}
\item In order to show that $\left|p_{k}^{d}(z)\right|$ is large in some
neighborhood of $z_{i}^{(k,d)}$, let us first show that $\left|s_{k}^{d}(u)\right|$
is large in some neighborhood of $\sigma_{i}^{(k,d)}$. Recall that
$s_{k}^{d}(u)=\nn s_{k}^{d}(u)+h_{k}^{d}(u)$. We have shown in the
proof of \prettyref{lem:skd-roots-deviation} that if $\eta(k)$ is
any function satisfying $0<\eta(k)<\frac{\Cr{skd-radius}}{k}$, then
$\left|\nn s_{k}^{d}(u)\right|>\Cr{skd-below}\eta(k)$ everywhere
on $\partial B_{\eta(k)}\left(\nn{\sigma}_{i}^{(k,d)}\right)$. Furthermore,
in this case $\left|h_{k}^{d}(u)\right|<\Cr{hkd-bound-above}\jcbound k^{-2}$.
We now require that\[
\Cr{hkd-bound-above}\jcbound k^{-2}<\frac{1}{2}\Cr{skd-below}\eta(k)\]
 which is true if $k>\frac{2\Cr{hkd-bound-above}\jcbound}{\Cr{skd-below}\Cr{skd-radius}}=\Cr{kk-skd-deviation}\jcbound$.
In that case we have\[
\left|s_{k}^{d}(u)\right|>\frac{1}{2}\Cr{skd-below}\eta(k)\qquad\forall u\in B_{\eta\left(k\right)}\left(\nn{\sigma}_{i}^{(k,d)}\right)\]
This is almost what we want - we would like to have such a bound on
the boundary of a neighborhood of $\sigma_{i}^{(k,d)}$ instead of
$\nn{\sigma}_{i}^{(k,d)}$. If $i=0$ then these values coincide,
and so we're done. Otherwise, recall that for $k>\Cr{kk-skd-deviation}\jcbound$
we also have $\left|\nn{\sigma}_{i}^{(k,d)}-\sigma_{i}^{(k,d)}\right|\leq\Cr{skd-deviation}\jcbound k^{-2}$.
So in order to make sure that $\sigma_{i}^{(k,d)}$ belongs to $B_{\eta(k)}\left(\nn{\sigma}_{i}^{(k,d)}\right)$,
we just require that $\eta(k)\geq\Cr{skd-deviation}\jcbound k^{-2}$.

We have thus shown the following: 
\begin{enumerate}
\item For every function $0<\eta(k)<\frac{\Cr{skd-radius}}{k}$ and for
every $k>\Cr{kk-skd-deviation}\jcbound$, the following bound holds
for every $u$ on the boundary of a neighborhood of $\sigma_{0}^{(k,d)}=0$
of diameter $2\eta(k)$:\begin{equation}
\left|s_{k}^{d}(u)\right|>\frac{1}{2}\Cr{skd-below}\eta(k)\label{eq:skd-lower-bound}\end{equation}

\item For every $k>\Cr{kk-skd-deviation}\jcbound$ and $\eta(k)$ as above,
which additionally satisfies $\eta(k)\geq\Cr{skd-deviation}\jcbound k^{-2}$,
the above bound holds for every $u$ on the boundary of a neighborhood
of $\sigma_{i}^{(k,d)}$ of the same diameter $2\eta(k)$, for every
$i=0,1,\dots,\ord$.
\end{enumerate}
\item We can now show that similar bounds hold for $p_{k}^{d}(z)$. Again,
only the constants will be different. The map $\uz^{-1}$ \eqref{eq:u-z-map},
being a Möbius transformation, maps $B_{\eta(k)}\left(\nn{\sigma}_{i}^{(k,d)}\right)$
to a circular neighborhood of $z_{i}^{(k,d)}$ (which is not necessarily
centered at $z_{i}^{(k,d)}$) . Let $u^{*}\in B_{\eta(k)}\left(\nn{\sigma}_{i}^{(k,d)}\right)$
. Now $\left|u^{*}-\nn{\sigma}_{i}^{(k,d)}\right|\leq\Cr{skd-radius}k^{-1}$
and also $-\frac{\Cr{s-roots-upper}}{k}<\nn{\sigma}_{i}^{(k,d)}<0$.
Therefore if $k>2\left(\Cr{s-roots-upper}+\Cr{skd-radius}\right)$
then $\Re\left(u^{*}\right)>-\frac{1}{2}$ and so $\left|u^{*}+1\right|>\frac{1}{2}$.
On the other hand, in this case $\left|u^{*}+1\right|<2$.

Now let $u_{1}$ and $u_{2}$ be two points in the $u$-plane, such
that $\left|u_{1}-u_{2}\right|=r$ and $\frac{1}{2}<\left|u_{1}\right|,\left|u_{2}\right|<2$.
They are mapped to the $z$-plane such that\[
\frac{r}{4}<\left|\frac{\omega}{u_{1}+1}-\frac{\omega}{u_{2}+1}\right|<4r\]

Recalling \eqref{eq:skd-lower-bound} and \eqref{eq:eliminant-shifted},
we conclude:
\begin{enumerate}
\item For every function $0<\eta(k)<\frac{\Cr{skd-radius}}{k}$ and every
$k>\underbrace{\max\left(\Cr{kk-skd-deviation},2\left(\Cr{s-roots-upper}+\Cr{skd-radius}\right)\right)}_{\isdef\Cl[kk]{kk-pkd-upper}}\jcbound$,
there exists a circular neighborhood of $\omega$ having diameter
between $\frac{\eta(k)}{2}$ and $8\eta(k)$, such that the magnitude
of $p_{k}^{d}(z)$ on the boundary of this neighborhood satisfies
\[
\left|p_{k}^{d}(z)\right|=\left|z^{\ord+1}\right|\left|s_{k}^{d}(u)\right|>2^{-\ord-2}\Cr{skd-below}\eta(k)=\Cl{pkd-below}\eta\left(k\right)\]

\item For every $k>\Cr{kk-pkd-upper}\jcbound$ and $\eta(k)$ as above,
which additionally satisfies $\eta(k)\geq\Cr{skd-deviation}\jcbound k^{-2}$,
the above bound holds for every $z$ on the boundary of a neighborhood
of $z_{i}^{(k,d)}$ of the same diameter as above, for every $i=0,1,\dots,\ord$.
\end{enumerate}
\item Once we have the lower bound for $\left|p_{k}^{d}(z)\right|$ on circles
of diameter at most $8\eta\left(k\right)<\frac{8\Cr{skd-radius}}{k}$
containing $z_{i}^{(k,d)}$, let us now establish an upper bound for
$\left|e_{k}^{d}(z)\right|$ on these circles. Let $z_{*}$ belong
to such a circle. On one hand, its distance from $z_{i}^{(k,d)}$
is at most $\frac{8\Cr{skd-radius}}{k}$. On the other hand, by \prettyref{lem:first-perturbation-characterization}
$\left|z_{i}^{(k,d)}-\omega\right|<\frac{\Cr{finalfirstdevupper}}{k}$
for all $k>\Cr{kk-finalfirstdev}\jcbound$. Therefore $\left|z_{*}-\omega\right|<\frac{8\Cr{skd-radius}+\Cr{finalfirstdevupper}}{k}$.
Denote $\Cl{zeta-bound}\isdef8\Cr{skd-radius}+\Cr{finalfirstdevupper}$
and let $z_{\theta}=\omega+\zeta(k)\ee^{\imath\theta}$ where $\zeta(k)$
is some function satisfying $0<\zeta(k)<\frac{\Cr{zeta-bound}}{k}$.
Our goal now is to find a uniform upper bound for $\left|e_{k}^{d}(z_{\theta})\right|$.

\begin{enumerate}
\item By \eqref{eq:decomp-additional-smoothness} we have\begin{align}
\begin{split}r_{k}-m_{k} & =2\pi(\imath k)^{d+1}\left\{ \fc(\fun)-\fc(\sing_{d})\right\} =2\pi(\imath k)^{d+1}\left\{ \fc(\sing_{d_{1}})-\fc(\sing_{d})+\fc(\smooth^{*})\right\} \\
 & =2\pi(\imath k)^{d+1}\left\{ \frac{\omega^{k}}{2\pi}\cdot\sum_{l=d+1}^{d_{1}}\frac{A_{l}}{(\imath k)^{l+1}}+\fc(\smooth^{*})\right\} =\omega^{k}\cdot\sum_{l=1}^{d_{1}-d}\frac{A_{d+l}}{(\imath k)^{l}}+\underbrace{2\pi(\imath k)^{d+1}\fc(\smooth^{*})}_{\isdef\delta_{k}}\end{split}
\label{eq:poly-coeffs-perturbation}\end{align}
Therefore\[
\begin{split}e_{k}^{d}(z_{\theta}) & =\sum_{j=0}^{d+1}(-1)^{j}\binom{d+1}{j}\left\{ \omega^{k+j}\cdot\sum_{l=1}^{d_{1}-d}\frac{A_{d+l}}{(\imath(k+j))^{l}}+\delta_{k+j}\right\} z_{\theta}^{d+1-j}\\
 & =\underbrace{\sum_{l=1}^{d_{1}-d}(-\imath)^{l}\jc_{d+l}\sum_{j=0}^{d+1}\frac{(-1)^{j}}{\left(k+j\right)^{l}}\binom{d+1}{j}\omega^{k+j}z_{\theta}^{d+1-j}}_{\isdef\Lambda_{k}\left(z_{\theta}\right)}+\underbrace{\sum_{j=0}^{d+1}(-1)^{j}\binom{d+1}{j}\delta_{k+j}z_{\theta}^{d+1-j}}_{\isdef\Delta_{k}(z_{\theta})}\end{split}
\]

\item On one hand, we have the bound \eqref{eq:add-smooth-fc-bound}. On
the other hand, $\left|z_{\theta}\right|<2$ and therefore\[
\left|\Delta_{k}(z_{\theta})\right|\leq\C2^{\ord+1}\cdot2\pi\cdot k^{\ord+1}\left|\fc\left(\smooth\right)\right|\leq\frac{\Cl{delta-bound}\fsbb}{k^{d_{1}-d+1}}\]
for some $\Cr{delta-bound}$.
\item Now we need to estimate $\Lambda_{k}\left(z_{\theta}\right)$. First
\[
\begin{split}z_{\theta}^{d+1-j} & =\left(\omega+\zeta(k)\ee^{\imath\theta}\right){}^{d+1-j}\\
 & =\omega^{d+1-j}+\left(\ord+1-j\right)\omega^{d-j}\zeta(k)\ee^{\imath\theta}+\alpha_{j}(k)\end{split}
\]
where $\left|\alpha_{j}(k)\right|\leq\Cl{ztheta-1}\zeta^{2}\left(k\right)$
for some constant $\Cr{ztheta-1}$. Furthermore, using the estimate
of \prettyref{lem:binomial-sum-fractions} we have \begin{eqnarray*}
\Lambda_{k}\left(z_{\theta}\right) & = & \omega^{k+d+1}\sum_{l=1}^{d_{1}-d}\frac{A_{d+l}}{\imath^{l}}\underbrace{\sum_{j=0}^{d+1}(-1)^{j}\binom{d+1}{j}\frac{1}{(k+j)^{l}}}_{\left|\cdot\right|\leq\C\cdot k^{-d-l-1}}\\
 &  & +\zeta(k)\ee^{\imath\theta}(d+1)\omega^{k+d}\sum_{l=1}^{d_{1}-d}\frac{A_{d+l}}{\imath^{l}}\underbrace{\sum_{j=0}^{d}(-1)^{j}\binom{d}{j}\frac{1}{(k+j)^{l}}}_{\left|\cdot\right|\leq\C\cdot k^{-d-l}}\\
 &  & \underbrace{\sum_{l=1}^{d_{1}-d}(-\imath)^{l}\jc_{d+l}\sum_{j=0}^{d+1}\frac{(-1)^{j}}{\left(k+j\right)^{l}}\binom{d+1}{j}\omega^{k+j}\alpha_{j}\left(k\right)}_{\left|\cdot\right|\leq\C\cdot k^{-1}\zeta^{2}\left(k\right)}\end{eqnarray*}
for all $k>\Cl[kk]{kk-lambdak}$ where $\Cr{kk-lambdak}$ is an explicit
constant (see \prettyref{lem:binomial-sum-fractions}). Therefore\[
\left|\Lambda_{k}\left(z_{\theta}\right)\right|<\jccbound\times\left\{ \Cl{lambda-1}\cdot k^{-\ord-2}+\Cl{lambda-2}\zeta(k)k^{-\ord-1}+\Cl{lambda-3}\cdot k^{-1}\zeta^{2}\left(k\right)\right\} \]
Combining all the above estimates we therefore have for $k>\max\left(\Cr{kk-finalfirstdev},\Cr{kk-lambdak}\right)\jcbound$\begin{equation}
\left|e_{k}^{d}(z_{\theta})\right|<\jccbound\times\left(\frac{\Cr{lambda-1}}{k^{\ord+2}}+\frac{\Cr{lambda-2}}{k^{\ord+1}}\zeta(k)+\frac{\Cr{lambda-3}}{k}\zeta^{2}\left(k\right)\right)+\frac{\Cr{delta-bound}\fsbb}{k^{d_{1}-d+1}}\label{eq:ekd-upper-bound}\end{equation}

\end{enumerate}
\item We can finally compare $\left|p_{k}^{d}\left(z\right)\right|$ and
$\left|e_{k}^{d}\left(z\right)\right|$. Let $k>\underbrace{\max\left(\Cr{kk-finalfirstdev},\Cr{kk-lambdak},\Cr{kk-pkd-upper}\right)}_{\isdef\Cl[kk]{kk-comparison-initial}}\jcbound$
and consider two cases.

\begin{enumerate}
\item Suppose $z_{i}^{(k,d)}\neq\omega$. Set $\rho\left(k\right)=\frac{\Cr{qkd-deviation-reg}\sab}{8k^{2}}$
where $\Cr{qkd-deviation-reg}$ is to be determined, and suppose also
that\begin{equation}
\Cr{skd-deviation}\jcbound k^{-2}\leq\rho\left(k\right)<\frac{\Cr{skd-radius}}{k}\label{eq:rhok-condition}\end{equation}
We have shown above that there exists a neighborhood $D_{i}^{(k)}$
containing $z_{i}^{(k,d)}$ of diameter at most $8\rho\left(k\right)=\Cr{qkd-deviation-reg}\cdot\sab\cdot k^{-2}$
such that for every $z^{*}\in\partial D_{i}^{(k,d)}$ we have $\left|p_{k}^{d}(z^{*})\right|>\Cr{pkd-below}\rho\left(k\right)=\frac{\Cr{pkd-below}\Cr{qkd-deviation-reg}\sab}{8k^{2}}$.
On the other hand, for every such $z^{*}$ we have by \eqref{eq:ekd-upper-bound}\[
\begin{split}\left|e_{k}^{d}(z^{*})\right| & <\jccbound\times\left(\frac{\Cr{lambda-1}}{k^{\ord+2}}+\frac{\Cr{lambda-2}}{k^{\ord+1}}\rho(k)+\frac{\Cr{lambda-3}}{k}\rho^{2}\left(k\right)\right)+\frac{\Cr{delta-bound}\fsbb}{k^{d_{1}-d+1}}\\
 & <\frac{\jccbound\times\left(\Cr{lambda-1}+\Cr{lambda-2}\Cr{skd-radius}+\Cr{lambda-3}\Cr{skd-radius}^{2}\right)+\Cr{delta-bound}\fsbb}{k^{2}}<\left(\jccbound+\fsbb\right)\frac{\Cl{ekd-reg-upper}}{k^{2}}\end{split}
\]
Therefore we must choose $\Cr{qkd-deviation-reg}$ and $k$ for which
the condition $\frac{\Cr{pkd-below}\Cr{qkd-deviation-reg}\sab}{8k^{2}}>\frac{\Cr{ekd-reg-upper}\left(\jccbound+\fsbb\right)}{k^{2}}$
is satisfied, together with \eqref{eq:rhok-condition}. For example:\[
\begin{split}\Cr{qkd-deviation-reg} & \isdef\max\left(\frac{8\Cr{ekd-reg-upper}}{\Cr{pkd-below}},8\Cr{skd-deviation}\right)\\
k & >\underbrace{\max\left(\Cr{kk-comparison-initial},\frac{\Cr{qkd-deviation-reg}}{8\Cr{skd-radius}}\right)}_{\isdef\Cr{kk-qkd-deviation}}\times\sab\end{split}
\]
In this case, $q_{k}^{d}(z)$ has a simple zero $\kappa_{i}^{(k,d)}$
in $D_{i}^{(k)}$ so that $\left|\kappa_{i}^{(k,d)}-z_{i}^{(k,d)}\right|<\Cr{qkd-deviation-reg}\left(\jcbound+\jccbound+\fsbb\right)k^{-2}$.
\item Now consider the case $z_{0}^{(k,d)}=\omega$. Set $\rho\left(k\right)=\frac{\Cr{qkd-deviation-omega}\sab}{8k^{\ord+2}}$
where $\Cr{qkd-deviation-omega}$ is to be determined. We again require
that $\rho\left(k\right)<\frac{\Cr{skd-radius}}{k}$. We have shown
that whenever $k>\Cr{kk-pkd-upper}\jcbound$, there exists a neighborhood
$D_{0}^{\left(k\right)}$containing $\omega$ such that for every
$z^{*}\in\partial D_{0}^{(k)}$ we have $\left|p_{k}^{d}\left(z^{*}\right)\right|>\Cr{pkd-below}\rho\left(k\right)=\frac{\Cr{qkd-deviation-omega}\Cr{pkd-below}\sab}{8k^{\ord+2}}$.
On the other hand, by \eqref{eq:ekd-upper-bound} for every such $z^{*}$
we have\[
\begin{split}\left|e_{k}^{d}(z^{*})\right| & <\jccbound\times\left(\frac{\Cr{lambda-1}}{k^{\ord+2}}+\frac{\Cr{lambda-2}}{k^{\ord+1}}\rho(k)+\frac{\Cr{lambda-3}}{k}\rho^{2}\left(k\right)\right)+\frac{\Cr{delta-bound}\fsbb}{k^{d_{1}-d+1}}\\
 & <\frac{\jccbound\times\left(\Cr{lambda-1}+\Cr{lambda-2}\Cr{skd-radius}+\Cr{lambda-3}\Cr{skd-radius}^{2}\right)+\Cr{delta-bound}\fsbb}{k^{\ord+2}}<\frac{\Cr{ekd-reg-upper}\left(\jccbound+\fsbb\right)}{k^{\ord+2}}\end{split}
\]
So we require $\frac{\Cr{qkd-deviation-omega}\Cr{pkd-below}\sab}{8k^{\ord+2}}>\frac{\Cr{ekd-reg-upper}\left(\jccbound+\fsbb\right)}{k^{\ord+2}}$
together with $\frac{\Cr{qkd-deviation-omega}\sab}{8k^{\ord+2}}<\frac{\Cr{skd-radius}}{k}$.
This is possible for example when\[
\begin{split}\Cr{qkd-deviation-omega} & =\frac{8\Cr{ekd-reg-upper}}{\Cr{pkd-below}}\\
k & >\Cr{kk-qkd-deviation}\sab\geq\left(\frac{\Cr{qkd-deviation-omega}\sab}{8\Cr{skd-radius}}\right)^{\frac{1}{\ord+1}}\end{split}
\]

\end{enumerate}
\end{enumerate}
Thus we have completed the proof of \prettyref{lem:qkd-root-deviation}.\qedhere

\end{proof}
We can finally combine everything and prove the main result of this
section.
\begin{thm}
\label{thm:jump-final-accuracy}Let $\fun$ have $d_{1}\geq2d+1$
continuous derivatives everywhere in $\left[-\pi,\pi\right]\setminus\left\{ \jp\right\} $.
Let $q_{k}^{d}(z)$ be as defined in \eqref{eq:q-def}, and let $\left\{ \kappa_{i}^{\left(k,d\right)}\right\} _{i=0}^{\ord}$
denote its roots, such that $\left|\kappa_{0}^{\left(k,d\right)}\right|\leq\dots\left|\kappa_{\ord}^{\left(k,d\right)}\right|$
. Let $\left\{ \phi_{i}\right\} _{i=1}^{\ord}$ denote the roots of
the Laguerre polynomial $\lag 1{\ord}$, such that $\left|\phi_{1}\right|<\dots\left|\phi_{\ord}\right|$.
Let $y_{0}^{\left(k,d\right)}=\omega$ and $y_{i}^{\left(k,d\right)}=\uz^{-1}\left(-\frac{\phi_{i}}{k}\right)$
for $i=1,\dots,\ord$ (see \eqref{eq:ykd-def}). Then there exist
constants $\Cr{rootlower},\Cr{qkd-deviation-omega},\Cl{qkd-distance-skd}$
and $\Cl[kk]{kk-qkd-final}$ such that for every $k>\Cr{kk-qkd-final}\sab$
the following statemenets are true:
\begin{enumerate}
\item \label{enu:jump-final-i1}The numbers $\left\{ y_{i}^{\left(k,d\right)}\right\} $
lie on the ray $O\omega$, so that $\left|y_{i}^{\left(k,d\right)}\right|\geq1$,
and:\[
\Cr{rootlower}k^{-1}\leq\left|y_{i}^{(k,d)}-y_{j}^{(k,d)}\right|\qquad0\leq i<j\leq\ord\]

\item \label{enu:jump-final-i2}Each of the numbers $\left\{ \kappa_{i}^{\left(k,d\right)}\right\} _{i=1}^{\ord}$
is close to some $y_{i}^{\left(k,d\right)}$: \[
\left|\kappa_{i}^{\left(k,d\right)}-y_{i}^{\left(k,d\right)}\right|\leq\Cr{qkd-distance-skd}\cdot\sab\cdot k^{-2}\]

\item \label{enu:jump-final-i3}The smallest $\kappa_{0}^{\left(k,d\right)}$
is very close to $\omega$:\[
\left|\kappa_{0}^{\left(k,d\right)}-\omega\right|\leq\Cr{qkd-deviation-omega}\cdot\sab\cdot k^{-d-2}\]

\item \label{jump-final-i4}\prettyref{alg:rec-single-jump} provides an
approximation for $\omega$ which is accurate up to order $k^{-\left(\ord+2\right)}$.\end{enumerate}
\begin{proof}
We have already proved \eqref{enu:jump-final-i1} (for $k>\Cr{kk-rooty}$)
and \eqref{enu:jump-final-i3} (for $k>\Cr{kk-firstdev}\jcbound$)
- see \prettyref{lem:qkd-root-deviation}. \eqref{enu:jump-final-i2}
follows from \prettyref{lem:qkd-root-deviation} and \prettyref{lem:first-perturbation-characterization}
by choosing $\Cr{qkd-distance-skd}=\Cr{qkd-deviation-reg}+\Cr{firstdev}$
and $k>\Cr{kk-qkd-deviation}\sab$. In order to prove \eqref{jump-final-i4},
we need to show that no root $\kappa_{i}^{\left(k,d\right)}$ is closer
to the unit circle than $\kappa_{0}^{\left(k,d\right)}$. From geometric
considerations (see \prettyref{fig:root-geometry}), it is sufficient
to require that\[
\left|\kappa_{i}^{\left(k,d\right)}-y_{i}^{\left(k,d\right)}\right|\leq\Cr{qkd-distance-skd}\cdot\sab\cdot k^{-2}<\frac{1}{2}\Cr{rootlower}k^{-1}<\frac{1}{2}\min_{j\neq i}\left|y_{i}^{\left(k,d\right)}-y_{j}^{\left(k,d\right)}\right|\]
which is true whenever\[
k>\frac{2\Cr{qkd-distance-skd}\sab}{\Cr{rootlower}}\]
Therefore we choose\[
\Cr{kk-qkd-final}=\max\left(\Cr{kk-rooty},\Cr{kk-firstdev},\Cr{kk-qkd-deviation},\frac{2\Cr{qkd-distance-skd}}{\Cr{rootlower}}\right)\qedhere\]

\end{proof}
\end{thm}
\begin{figure}
\newcommand{\writelabelsroots}{
\rput(115,131){Roots of $\lag{1}{\ord}$}
\rput(75,90){$\phi_1$}
\rput(100,90){$\phi_2$}
\rput(205,90){$\phi_{\ord}$}
\rput(325,170){\small $y_i^{(k,\ord)}=\uz^{-1}\left(-\frac{\phi_i}{k}\right)$}
\rput(479,145){\tiny $\omega$}
\rput(490,40){\tiny $\omega$}
\rput(517,143){\tiny $y_{1}$}
\rput(610,185){\tiny $y_{\ord}$}
\rput(590,110){\tiny $y_{i}$}
\rput(609,79){\tiny $\kappa_{i}$}
\rput(611,134){\tiny $z_{i}$}
\rput(517,37){\tiny $\kappa_{0}$}
\rput(560,10){\tiny $r\sim k^{-\ord-2}$}
\rput(643,62){\tiny $r\sim k^{-2}$}
\rput(517,185){\tiny $\sim k^{-1}$}
\rput(456,166){\tiny $|z|=1$}
}
\input{roots.tex}

\caption{The geometry of $\left\{ \kappa_{i}^{\left(k,d\right)}\right\} ,\left\{ y_{i}^{\left(k,d\right)}\right\} ,\left\{ z_{i}^{\left(k,d\right)}\right\} $.
The superscripts $\left(k,\ord\right)$ are omitted. The picture on
the right shrinks towards the unit circle as $k\to\infty$.}
\label{fig:root-geometry}

\end{figure}

\subsection{\label{sub:jump-magnitudes-accuracy-analysis}Accuracy analysis:
jump magnitudes}

\global\long\def\nc{\ensuremath{B}}
\global\long\def\dfer{\ensuremath{\eta}}
\global\long\def\cfer{\ensuremath{\varepsilon}}
\newcommandx\mult[2][usedefault, addprefix=\global, 1=d, 2=k]{\ensuremath{S_{#2,#1}}}
\newcommandx\partialrow[1][usedefault, addprefix=\global, 1=k]{\ensuremath{\vec{v_{#1}}}}

Suppose that $d_{1}\geq2d+1$ and let $k>\Cr{kk-qkd-final}\sab$ so
that our algorithm gives an approximation $\nn{\omega}_{(k)}$ with
error at most $\Cr{qkd-deviation-omega}\sab\cdot k^{-d-2}$, in accordance
with \prettyref{thm:jump-final-accuracy}. Our goal is to analyze
the accuracy of calculating the approximate jump magnitudes $\jc_{l}^{(k)}$,
given by the solution of the linear system \eqref{eq:approximate-system-jump-magnitudes}.
For convenience, we denote\begin{eqnarray*}
\nc_{l} & \isdef\imath^{l}\jc_{d-l}\\
\nn{\nc}_{l}^{(k)} & \isdef\imath^{l}\nn{\jc}_{d-l}^{(k)}\end{eqnarray*}

We consider only the case of exactly $d+1$ equations. Thus we can
write this system in the following form: \begin{equation}
\left[\begin{array}{c}
r_{k}\widetilde{\omega}_{(k)}^{-k}\\
\vdots\\
r_{k+d}\widetilde{\omega}_{(k)}^{-k-d}\end{array}\right]=V_{k}\times\left[\begin{array}{c}
\nn{\nc}_{0}^{(k)}\\
\vdots\\
\nn{\nc}_{d}^{(k)}\end{array}\right]\label{eq:approx-coeffs-matrix-system}\end{equation}
where $V_{k}$ is the $(d+1)\times(d+1)$ system matrix\[
V_{k}\isdef\left[\begin{array}{cccc}
1 & k & \dotsc & k^{d}\\
1 & (k+1) & \dotsc & (k+1)^{d}\\
\vdots & \vdots & \vdots & \vdots\\
1 & (k+d) & \dotsc & (k+d)^{d}\end{array}\right]\]

By \eqref{eq:exact-system-jump-magnitudes}, the ``true'' coefficients
$\nc_{j}$ satisfy\begin{equation}
\left[\begin{array}{c}
m_{k}\omega^{-k}\\
\vdots\\
m_{k+d}\omega^{-k-d}\end{array}\right]=V_{k}\times\left[\begin{array}{c}
\nc_{0}\\
\vdots\\
\nc_{d}\end{array}\right]\label{eq:exact-coeffs-matrix-system}\end{equation}

Our goal is to estimate the error $\cfer_{j}^{(k)}\isdef\nc_{j}-\nn{\nc}_{j}^{(k)}$
for $j=0,1,\dots,\ord$. Let\[
\dfer_{j}^{(k)}\isdef r_{k+j}\nn{\omega}_{(k)}^{-k-j}-m_{k+j}\omega^{-k-j}\]
Then subtracting \eqref{eq:exact-coeffs-matrix-system} from \eqref{eq:approx-coeffs-matrix-system}
gives\begin{equation}
\left[\begin{array}{c}
\cfer_{0}^{(k)}\\
\cfer_{1}^{(k)}\\
\vdots\\
\cfer_{d}^{(k)}\end{array}\right]=V_{k}^{-1}\times\left[\begin{array}{c}
\dfer_{0}^{(k)}\\
\dfer_{1}^{(k)}\\
\vdots\\
\dfer_{d}^{(k)}\end{array}\right]\label{eq:coeffs-accuracy-matrix-estimate}\end{equation}

This is the key relation of this section. In order to estimate the
magnitude of $\cfer_{j}^{(k)}$, we shall first write out explicit
expansions for the quantities $\dfer_{j}^{(k)}$, and then investigate
how these expansions are transformed when multiplied by the matrix
$V_{k}^{-1}$. Our analysis will show that the special combination
of the structures of both this matrix and the expansion coefficients
results in remarkable cancellations.

Let us start by investigating the structure of the matrix $V_{k}$.
\begin{defn}
Let $\mult$ denote the $(d+1)\times(d+1)$ square matrix with entries:\[
\left(\mult[r]\right)_{m,n}=(-k)^{n-m}\binom{n-1}{n-m}\]
\end{defn}
\begin{example}
For $d=4$, we have\[
\mult[4]=\left(\begin{array}{ccccc}
1 & -k & k^{2} & -k^{3} & k^{4}\\
0 & 1 & -2k & 3k^{2} & -4k^{3}\\
0 & 0 & 1 & -3k & 6k^{2}\\
0 & 0 & 0 & 1 & -4k\\
0 & 0 & 0 & 0 & 1\end{array}\right)\]
\end{example}
\begin{defn}
For every $k\in\naturals$ let the symbol $\partialrow$ denote the
following $1\times(d+1)$ row vector\[
\partialrow\isdef\left[\begin{array}{cccc}
1 & k & \dots & k^{d}\end{array}\right]\]

\end{defn}
With this definition, we have\begin{equation}
V_{k}=\left[\begin{array}{c}
\partialrow\\
\partialrow[k+1]\\
\vdots\\
\partialrow[k+d]\end{array}\right]\label{eq:vv-def}\end{equation}

\begin{lem}
Let $k\geq0$, then\begin{equation}
V_{k}^{-1}=\mult\times V_{0}^{-1}\label{eq:vv-dependence-m}\end{equation}
\end{lem}
\begin{proof}
Let $1\leq m\leq d+1$ and $0\leq t\leq d$. The $m$-th entry of
the vector $\partialrow[k+t]\times\mult$ equals to\[
\left(\partialrow[k+t]\times\mult\right)_{m}=\sum_{l=0}^{m-1}(k+t)^{l}(-k)^{m-1-l}\binom{m-1}{m-1-l}=t^{m-1}\]
and therefore\[
\partialrow[k+t]\times\mult=\partialrow[t]\]
\eqref{eq:vv-dependence-m} then follows from \eqref{eq:vv-def}.
\end{proof}
Now we would like to expand $\dfer_{j}^{(k)}$. We can obviously assume
the equality \[
\nn{\omega}_{(k)}=\omega+\frac{\alpha\left(k\right)}{k^{\ord+2}}\qquad\text{such that}\qquad\left|\alpha\left(k\right)\right|\leq\Cr{qkd-deviation-omega}\sab\]

Now we estimate $\nn{\omega}_{\left(k\right)}$ by \prettyref{prop:neg-exp-estimate}
as follows:\[
\begin{split}\left(\omega+\frac{\alpha(k)}{k^{\ord+2}}\right)^{-\left(k+j\right)} & =\omega^{-k-j}\left(1+\frac{\alpha(k)\omega^{-1}}{k^{d+2}}\right)^{-k-j}=\omega^{-k-j}\left(1-\left(k+j\right)\frac{\alpha\left(k\right)\omega^{-1}}{k^{\ord+2}}+\Cl[rr]{rr-newomega}\left(k,j\right)\right)\end{split}
\]
where $k$ is large enough so that $\frac{\alpha(k)\omega^{-1}}{k^{d+2}}<\frac{3}{k+j+2}$
is satisfied, and \[
\left|\Cr{rr-newomega}\left(k,j\right)\right|<\frac{\left(k+j\right)\left(k+j+1\right)\alpha^{2}\left(k\right)\omega^{-2}}{2k^{2\left(\ord+2\right)}\left(1-\frac{\alpha(k)\omega^{-1}\left(k+j+2\right)}{3k^{\ord+2}}\right)}<\C\cdot\sab^{2}k^{-2\ord-3}\]
Obviously, $\left|r_{k}\right|\leq\C\cdot\sab\cdot k^{\ord}$. Now
by \eqref{eq:poly-coeffs-perturbation}, we have\[
\begin{split}\dfer_{j}^{(k)} & =\left(m_{k+j}+\omega^{k+j}\cdot\sum_{l=1}^{d_{1}-d}\frac{A_{d+l}}{\left(\imath\left(k+j\right)\right)^{l}}+\delta_{k+j}\right)\nn{\omega}_{(k)}^{-k-j}-m_{k+j}\omega^{-k-j}\\
 & =\left(m_{k+j}+\omega^{k+j}\cdot\sum_{l=1}^{d_{1}-d}\frac{A_{d+l}}{\left(\imath\left(k+j\right)\right)^{l}}+\delta_{k+j}\right)\omega^{-k-j}\left(1-\frac{\left(k+j\right)\alpha\left(k\right)\omega^{-1}}{k^{d+2}}\right)-m_{k+j}\omega^{-k-j}+r_{k+j}\omega^{-k-j}\Cr{rr-newomega}\left(k\right)\\
 & =\frac{\beta\left(k\right)}{k^{d+2}}\sum_{l=0}^{d}\nc_{l}\left(k+j\right)^{l+1}+\sum_{l=1}^{d_{1}-d}\frac{A_{d+l}}{\left(\imath\left(k+j\right)\right)^{l}}+\Cl[rr]{rr-etak}\left(k,j\right)\end{split}
\]

where $\left|\Cr{rr-etak}\left(k,j\right)\right|\leq\C\cdot\sab^{2}k^{-\ord-2}$
and $\left|\beta\left(k\right)\right|\leq\C\cdot\sab$.

Therefore we can write \begin{equation}
\begin{split}\left[\begin{array}{c}
\dfer_{0}^{(k)}\\
\vdots\\
\dfer_{j}^{(k)}\\
\vdots\\
\dfer_{d}^{(k)}\end{array}\right] & =\beta\left(k\right)\nc_{0}\left[\begin{array}{c}
\frac{k}{k^{\ord+2}}\\
\vdots\\
\frac{k+j}{k^{\ord+2}}\\
\vdots\\
\frac{k+\ord}{k^{\ord+2}}\end{array}\right]+\dots+\beta\left(k\right)\nc_{l}\left[\begin{array}{c}
\frac{k^{l}}{k^{\ord+2}}\\
\vdots\\
\frac{\left(k+j\right)^{l}}{k^{\ord+2}}\\
\vdots\\
\frac{\left(k+\ord\right)^{l}}{k^{\ord+2}}\end{array}\right]+\dots+\beta\left(k\right)\nc_{\ord}\left[\begin{array}{c}
\frac{k^{\ord+1}}{k^{\ord+2}}\\
\vdots\\
\frac{\left(k+j\right)^{\ord+1}}{k^{\ord+2}}\\
\vdots\\
\frac{\left(k+\ord\right)^{\ord+1}}{k^{\ord+2}}\end{array}\right]\\
 & +\frac{\jc_{\ord+1}}{\imath}\left[\begin{array}{c}
\frac{1}{k}\\
\vdots\\
\frac{1}{k+j}\\
\vdots\\
\frac{1}{k+\ord}\end{array}\right]+\dots+\frac{\jc_{\ord+l}}{\imath^{l}}\left[\begin{array}{c}
\frac{1}{k^{l}}\\
\vdots\\
\frac{1}{\left(k+j\right)^{l}}\\
\vdots\\
\frac{1}{\left(k+\ord\right)^{l}}\end{array}\right]+\dots+\frac{\jc_{d_{1}}}{\imath^{\ord}}\left[\begin{array}{c}
\frac{1}{k^{\ord+1}}\\
\vdots\\
\frac{1}{\left(k+j\right)^{\ord+1}}\\
\vdots\\
\frac{1}{\left(k+\ord\right)^{\ord+1}}\end{array}\right]+\left[\begin{array}{c}
\Cr{rr-etak}\left(k,0\right)\\
\vdots\\
\Cr{rr-etak}\left(k,j\right)\\
\vdots\\
\Cr{rr-etak}\left(k,\ord\right)\end{array}\right]\end{split}
\label{eq:error-expansion-at-infinity}\end{equation}

In light of \eqref{eq:coeffs-accuracy-matrix-estimate}, we would
now like to examine the action of the matrix $V_{0}^{-1}$ on the
vectors in the right hand side of \eqref{eq:error-expansion-at-infinity}.
\begin{lem}
Let $j=0,1,\dots,\ord$.
\begin{enumerate}
\item If $l=1,2,\dots,\ord$ then \begin{equation}
\left[\begin{array}{c}
k^{l}\\
\vdots\\
\left(k+j\right)^{l}\\
\vdots\\
\left(k+\ord\right)^{l}\end{array}\right]=V_{0}\times\left[\begin{array}{c}
k^{l}\binom{l}{0}\\
\vdots\\
k^{l-j}\binom{l}{j}\\
\vdots\\
1\cdot\binom{l}{l}\\
0\\
\vdots\\
0\end{array}\right]\label{eq:action-v0-1-1}\end{equation}

\item Otherwise, there exists a function $\Cl[rr]{rr-powers}:\left\{ 0,1,\dots,\ord\right\} \to\reals$
such that \begin{equation}
\left[\begin{array}{c}
k^{\ord+1}\\
\vdots\\
\left(k+j\right)^{\ord+1}\\
\vdots\\
\left(k+\ord\right)^{\ord+1}\end{array}\right]=V_{0}\times\left[\begin{array}{c}
k^{\ord+1}\binom{\ord+1}{0}\\
\vdots\\
k^{\ord+1-j}\binom{\ord+1}{j}\\
\vdots\\
k\binom{\ord+1}{1}\end{array}\right]+\left[\begin{array}{c}
\Cr{rr-powers}\left(0\right)\\
\vdots\\
\Cr{rr-powers}\left(j\right)\\
\vdots\\
\Cr{rr-powers}\left(\ord\right)\end{array}\right]\label{eq:action-v0-1-2}\end{equation}

\end{enumerate}
\end{lem}
\begin{proof}
Straightforward application of the binomial theorem.
\end{proof}
\global\long\def\co{\ensuremath{\tau}}

\begin{lem}
\label{lem:action-v0}For $j=1,2,\dotsc$ and $i=1,\dotsc,d+1$ denote\[
\co_{j}^{i}\isdef(-1)^{i-1}\binom{j+i-2}{j-1}\]
Then there exists a bounded function $\Cl[rr]{rr-powers2}:\left\{ 0,1,\dots,\ord\right\} \times\naturals\to\reals$
such that\begin{equation}
\left[\begin{array}{c}
\frac{1}{k^{j}}\\
\frac{1}{\left(k+1\right)^{j}}\\
\vdots\\
\frac{1}{\left(k+d\right)^{j}}\end{array}\right]=V_{0}\times\left[\begin{array}{c}
\frac{\co_{j}^{1}}{k^{j}}\\
\frac{\co_{j}^{2}}{k^{j+1}}\\
\vdots\\
\frac{\co_{j}^{d+1}}{k^{j+d}}\end{array}\right]+\frac{1}{k^{\ord+j+1}}\left[\begin{array}{c}
\Cr{rr-powers2}\left(0,j\right)\\
\vdots\\
\Cr{rr-powers2}\left(l,j\right)\\
\vdots\\
\Cr{rr-powers2}\left(\ord,j\right)\end{array}\right]\label{eq:action-v0-2}\end{equation}
\end{lem}
\begin{proof}
First recall the well-known%
\footnote{It can be proven by induction on $j$, using the identity $\sum_{k=0}^{n}\binom{r+k}{r}=\binom{r+n+1}{r+1}$.%
} power series expansion\[
\frac{1}{(1+x)^{j}}=\sum_{n=0}^{\infty}(-1)^{n}\binom{j-1+n}{j-1}x^{n}\]

Now let $l=0,1,\dotsc,d$.
\begin{enumerate}
\item On one hand, the $(l+1)$-st entry in the product on the right-hand
side of \eqref{eq:action-v0-2} equals to\[
g_{j,l}\isdef\sum_{i=0}^{d}(-1)^{i}\frac{\binom{j-1+i}{i}}{k^{i+j}}l^{i}\]

\item On the other hand, by \prettyref{prop:neg-exp-estimate} we have for
some bounded function $\Cr{rr-powers2}:\left\{ 0,1,\dots,\ord\right\} \times\naturals\to\reals$
\[
\begin{split}\frac{1}{(k+l)^{j}} & =\frac{1}{k^{j}}\cdot\frac{1}{\left(1+\frac{l}{k}\right)^{j}}=\frac{1}{k^{j}}\left\{ \sum_{i=0}^{d}(-1)^{i}\binom{j-1+i}{j-1}\left(\frac{l}{k}\right)^{i}+\frac{\Cr{rr-powers2}\left(l,j\right)}{k^{d+1}}\right\} \\
 & =\underbrace{\sum_{i=0}^{d}(-1)^{i}\frac{\binom{j-1+i}{i}}{k^{i+j}}l^{i}}_{=g_{j,l}}+\frac{\Cr{rr-powers2}\left(l,j\right)}{k^{j+d+1}}\end{split}
\]

\end{enumerate}
Thus \eqref{eq:action-v0-2} is proved.\qedhere

\end{proof}
It is now easily seen that the multiplication by $V_{0}^{-1}$ {}``orders
up'' the vectors in \eqref{eq:error-expansion-at-infinity} by decreasing
powers of $k$. Further multiplication by $\mult$ from the left preserves
this structure, as is evident from the following calculation.

\global\long\def\oco{\ensuremath{\gamma}}

\begin{lem}
\label{lem:action-s}Let $c_{i,j}$ be arbitrary constants. Then there
exist constants $\oco_{i,j}$ such that\begin{equation}
\mult\times\left[\begin{array}{c}
\frac{c_{1,j}}{k^{j}}\\
\frac{c_{2,j}}{k^{j+1}}\\
\vdots\\
\frac{c_{d+1,j}}{k^{j+d}}\end{array}\right]=\left[\begin{array}{c}
\frac{\oco_{1,j}}{k^{j}}\\
\frac{\oco_{2,j}}{k^{j+1}}\\
\vdots\\
\frac{\oco_{d+1,j}}{k^{j+d}}\end{array}\right]\label{eq:action-s}\end{equation}
\end{lem}
\begin{proof}
Let $i=1,\dotsc,d+1$ and consider the $i$-th entry of the product,
say $y_{i}$:

\[
\begin{split}y_{i} & =\sum_{l=1}^{d+1}\left(\mult\right)_{i,l}\times\frac{c_{l,j}}{k^{j+l-1}}=\sum_{l=0}^{d}(-k)^{l+1-i}\binom{l}{l+1-i}\times\frac{c_{l+1,j}}{k^{l+j}}=\frac{1}{k^{j+i-1}}\oco_{i,j}\end{split}
\]

where $\oco_{i,j}=\sum_{l=i-1}^{d}(-1)^{l+1-i}\binom{l}{l-(i-1)}c_{l+1,j}$.
This proves the claim.
\end{proof}
We can now prove the main result of this section. 
\begin{thm}
\label{thm:magnitude-final-accuracy}Assume that $d_{1}\geq2d+1$
and $k>\Cr{kk-qkd-final}\sab$, so that by \prettyref{thm:jump-final-accuracy}
we have $\left|\nn{\omega}_{(k)}-\omega\right|\leq\Cr{qkd-deviation-omega}\cdot\sab\cdot k^{-d-2}$.
Then there exist constants $\Cl{magnitude-final},\Cl[kk]{kk-magnitude-final}$
such that for every $k>\Cr{kk-magnitude-final}\sab$ and $l=0,1,\dotsc d$
the error in determining $\jc_{l}$ is \[
\left|\nn{\jc}_{l}^{(k)}-\jc_{l}\right|\leq\Cr{magnitude-final}\cdot\sab^{2}\cdot k^{l-d-1}\]
\end{thm}
\begin{proof}
Combine \eqref{eq:coeffs-accuracy-matrix-estimate}, \eqref{eq:error-expansion-at-infinity},
\eqref{eq:action-v0-1-1}, \eqref{eq:action-v0-1-2}, \eqref{eq:action-v0-2}
and \eqref{eq:action-s}.
\end{proof}

\section{Localizing the discontinuities\label{sec:localization}}

\global\long\def\jmin{J_{1}}
\global\long\def\jmax{J_{2}}
\global\long\def\jdist{J_{3}}
\global\long\def\fsbbn{\widehat{\fsbb}}

As we have seen, both the location and the magnitudes of the jump
can be reconstructed with high accuracy. The remaining ingredient
in our method is to divide the initial function into regions containing
a single jump, and subsequently apply the reconstruction algorithm
in each region.

Our approach is to multiply the initial function $\fun$ by a {}``bump''
$g_{j}$ which vanishes outside some neighborhood of the $j$-th jump.
This step requires a-priori estimates of the jump positions, which
can fortunately be obtained by a variety of methods, for example:
\begin{enumerate}
\item The concentration method of Gelb\&Tadmor \cite{gelb1999detection};
\item The method of partial sums due to Banerjee\&Geer \cite{banerjee1998exponentially};
\item Eckhoff's method with order zero (the Prony method).
\end{enumerate}
All the above methods provide accurate estimates of $\left\{ \jp_{j}\right\} $
up to first order. For definiteness, we present the description of
the last method and a rigorous proof of its convergence in \prettyref{app:initial-estimates-prony}.

Now, the multiplication is implemented as Fourier domain convolution.
Because of the Fourier uncertainty principle, the Fourier series of
our bump will have inifinite support and therefore every practically
computable convolution will always be an approximation to the exact
one. Nevertheless, an error of order at most $k^{-\ord_{1}-2}$ in
the Fourier coefficients will be ``absorbed'' in the constant $\fsbb$
\eqref{eq:add-smooth-fc-bound} and therefore we will still have accurate
estimates for the reconstruction of each separate jump. This will
require us to use bump functions which are $C^{\infty}$. An explicit
construction of such a function is provided in \prettyref{app:mollifier-explicit-construction}.

We assume that the following quantities are known a-priori:
\begin{itemize}
\item the lower and upper bounds for the jump magnitudes of order zero:
$\jmin\leq\left|\jc_{0,j}\right|\leq\jmax$;
\item the minimal distance between any two jumps $\left|\jp_{i}-\jp_{j}\right|\geq\jdist>0$;
\item a constant T for which\begin{equation}
\left|2\pi\left(\imath k\right)\fc\left(f\right)-\sum_{j=1}^{\np}\jc_{0,j}\omega_{j}^{k}\right|\leq T\cdot k^{-1}\label{eq:order0-fc-bound}\end{equation}

\end{itemize}
Our localization algorithm can be summarized as follows.

\global\long\def\nnn#1{\widehat{#1}}

\begin{algorithm}
\label{alg:localization-alg}Let $\fun$ be a piecewise-smooth function
of order $d_{1}\geq2\ord+1$ with $\np$ jumps $\left\{ \jp_{j}\right\} _{j=1}^{\np}$
and jump magnitudes $\left\{ \jc_{l,j}\right\} _{l=0,\dots,\ord}^{j=1,\dots,\np}$.
Let there be given the Fourier coefficients $\left\{ \fc\left(\fun\right)\right\} _{\left|k\right|=0}^{2\startcoeff+\ord+1}$,
where $\startcoeff$ is large enough (see below). 
\begin{enumerate}
\item Using the a-priori bounds $\jmin,\jmax,\jdist,T$, obtain approximate
locations of the jumps $\left\{ \nnn{\jp_{j}}\right\} $ via \prettyref{alg:prony-initial-estimates}.
In particular, the error $\left|\nnn{\jp_{j}}-\jp_{j}\right|$ should
not exceed $\frac{\jdist}{3}$, and this will be possible if $\startcoeff$
is not smaller than required by \prettyref{thm:initial-estimates-prony}.
\item For each $\nnn{\jp_{j}}$:

\begin{enumerate}
\item Construct the bump $g_{j}$ centered at $\nnn{\jp_{j}}$ with parameters
$t=2\cdot\frac{\jdist}{3}$ and $E=\jdist$, according to \prettyref{app:mollifier-explicit-construction}.
Calculate its Fourier coefficients in the range $k=-3\startcoeff\dots3\startcoeff$
according to \eqref{eq:bump-fc-formula}. 
\item Now let $h_{j}=f\cdot g_{j}$. For each $k=0,1,\dots,\startcoeff+\ord+1$
calculate\begin{equation}
\widetilde{\fc}^{(\startcoeff)}(h_{j})=\sum_{i=-2\startcoeff}^{2\startcoeff}\fc[i](f)\fc[k-i](g_{j})\label{eq:convolution-coeffs-approximate}\end{equation}

\item Use the above approximate Fourier coefficients $\nn{\fc}^{(\startcoeff)}\left(h_{j}\right)$
as the input to \prettyref{alg:rec-single-jump} for reconstructing
all the parameters of a single jump.
\end{enumerate}
\end{enumerate}
\end{algorithm}
\begin{thm}
\label{thm:localization-preserves-accuracy} \prettyref{alg:localization-alg}
will produce estimates of all these parameters with the accuracy as
stated in Theorems \ref{thm:jump-final-accuracy} and \ref{thm:magnitude-final-accuracy},
$\fsbb$ being replaced with some other constant $\fsbbn=\fsbbn\left(\fsbb,T,\jmax,\jdist\right)$.\end{thm}
\begin{proof}
It is clear that the exact function $h_{j}=f\cdot g_{j}$ has exactly
one jump at $\jp$ and jump magnitudes $\jc_{0,j},\dots,\jc_{\ord,j}$.
In order to prove that Theorems \ref{thm:jump-final-accuracy} and
\ref{thm:magnitude-final-accuracy} can be applied, it is sufficient
to show that the error $\left|\nn{\fc}^{\left(\startcoeff\right)}\left(h_{j}\right)-\fc\left(h_{j}\right)\right|$
is of the order $k^{-\left(\ord_{1}+2\right)}$. By the Fourier convolution
theorem the exact Fourier coefficients of $h_{j}$ are equal to:\[
\fc(h_{j})=\sum_{i=-\infty}^{\infty}\fc[i](\fun)\fc[k-i](g_{j})\]
while our algorithm approximates these by the truncated convolution
\eqref{eq:convolution-coeffs-approximate}. Let us estimate the convolution
tail\[
\Delta\fc^{(\startcoeff)}(h_{j})=\sum_{i=-\infty}^{-2\startcoeff}\fc[i](f)\fc[k-i](g_{j})+\sum_{i=2\startcoeff}^{\infty}\fc[i](f)\fc[k-i](g_{j})\]
On one hand, the Fourier coefficients of $\fun$ can be bounded using
\eqref{eq:order0-fc-bound}: \[
\left|\fc\left(\fun\right)\right|\leq\Cl{jcbound-all}\left(\jmax+T\right)k^{-1}\]
On the other hand, taking $\alpha=\ord_{1}+1$ we have by \prettyref{thm:bump-fourier-decay}\[
\left|\fc\left(g_{j}\right)\right|\leq\frac{\Cl{each-bump-fc-bound}}{\jdist^{\ord_{1}+1}}\cdot\frac{1}{k^{\ord_{1}+2}}\]
Finally\[
\left|\Delta\fc^{(\startcoeff)}(h_{j})\right|\leq\frac{\Cl{deltack-loc}\left(\jmax+T\right)}{\jdist^{\ord_{1}+1}}\sum_{i=2\startcoeff}^{\infty}\frac{1}{i^{\ord_{1}+3}}\leq\frac{\Cr{deltack-loc}\left(\jmax+T\right)}{\jdist^{\ord_{1}+1}}\zeta(\ord_{1}+3,2\startcoeff)\leq\frac{\C\cdot\left(\jmax+T\right)}{\jdist^{\ord_{1}+1}}\startcoeff^{-\ord_{1}-2}\]
where $\zeta(s,q)$ is the Hurwitz zeta function. Therefore \prettyref{alg:rec-single-jump}
will produce estimates of $\left\{ \nn{\jp}_{j}\right\} $ and $\left\{ \nn{\jc}_{l,j}\right\} $
with accuracy as guaranteed by Theorems \ref{thm:jump-final-accuracy}
and \ref{thm:magnitude-final-accuracy} where\[
\fsbbn=\fsbb+\frac{\left(\jmax+T\right)}{\jdist^{\ord_{1}+1}}\qedhere\]

\end{proof}

\section{\label{sec:final-accuracy}Final accuracy}

\global\long\def\nsmooth{\ensuremath{\smooth}}
\global\long\def\singerr{\Theta}
\newcommandx\jfr[1][usedefault, addprefix=\global, 1=r]{\mathbf{D_{#1}}}

In this section we are going to calculate the overall accuracy of
approximation. Let us therefore suppose that $d_{1}\geq2d+1$, and
so using the Fourier coefficients ${\fc[-2\startcoeff](\fun),\dotsc,\fc[2\startcoeff](\fun)}$
we have reconstructed the singular part $\sing(x)$ with accuracy
specified by \prettyref{thm:localization-preserves-accuracy}.  Recall
from \eqref{eq:final-approximation-explicit} that our final approximation
is defined by\[
\begin{split}\nn{\fun} & =\sum_{\left|k\right|\leq\startcoeff}\left(\fc(\fun)-\fc(\widetilde{\sing})\right)\ee^{\imath kx}+\nn{\sing}\end{split}
\]
where $\nn{\sing}=\sum_{j=1}^{\np}\sum_{l=0}^{d}\nn{\jc}_{l,j}V_{l}(x;\nn{\jp_{j}})$.
Intuitively, the approximation error function $\nn{\fun}-\fun$ will
look as depicted in \prettyref{fig:overall-error} - very small almost
everywhere except in some shrinking neighborhoods of the jump points.
Let $y\in\left[-\pi,\pi\right]\setminus\left\{ \jp_{j}\right\} _{j=1}^{\np}$.
If we take $\startcoeff$ large enough so that the error estimate
of \prettyref{thm:jump-final-accuracy} will be less than the distance
to the nearest jump $\left|y-\jp_{j}\right|$, then $y$ will lie
in the ``flat'' region of \prettyref{fig:overall-error} and the
error $\left|\nn{\fun}\left(y\right)-\fun\left(y\right)\right|$ will
be small. This is precisely the content of our final theorem.

Let us denote the ``jump-free'' region by\[
\jfr\isdef\left[-\pi,\pi\right]\setminus\left(\bigcup_{j=1}^{\np}B_{r}\left(\jp_{j}\right)\right)\]

\begin{thm}
\label{thm:final-accuracy}Let $\fun:\left[-\pi,\pi\right]\to\reals$
have $\np$ jump discontinuities $\left\{ \jp_{j}\right\} _{j=1}^{\np}$,
and let it be $d_{1}$-times continuously differentiable between the
jumps. Let $r>0$. Then for every integer $d$ satisfying $2d+1\leq\ord_{1}$,
there exist explicit functions $F=F\left(\magbounds,\fsbbn\right),G=G\left(\magbounds,\fsbbn,r\right)$
depending on all the a-priori bounds such that for all $M>G$ \prettyref{alg:entire-reconstruction}
reconstructs the locations and the magnitudes of the jumps with accuracy
provided by \prettyref{thm:localization-preserves-accuracy}, and
with the pointwise accuracy \[
\left|\nn{\fun}(y)-\fun(y)\right|\leq F\cdot\startcoeff^{-\ord-1}\qquad y\in\jfr\]
\end{thm}
\begin{proof}
Define\[
\begin{split}\fun_{\startcoeff} & \isdef\sum_{\left|k\right|\leq\startcoeff}\left(\fc(\fun)-\fc(\sing)\right)\ee^{\imath kx}+\sing\end{split}
\]
We write the overall approximation error as \begin{equation}
\begin{split}\left|\nn{\fun}(y)-\fun(y)\right| & \leq\left|\nn{\fun}(y)-\fun_{\startcoeff}(y)\right|+\left|\fun_{\startcoeff}(y)-\fun(y)\right|\end{split}
\label{eq:final-error-breakup}\end{equation}
Let us examine the two terms on the right-hand side separately.

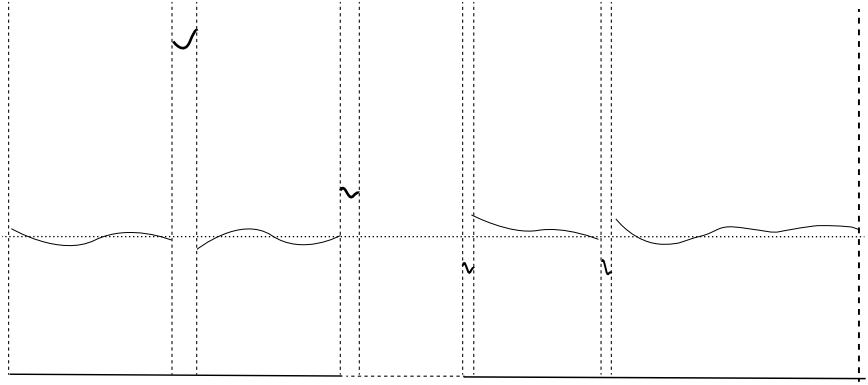
\begin{figure}
\psset{xunit=.5pt,yunit=.5pt,runit=.5pt}
\begin{pspicture}(663.50164795,324.19845581)
{
\newrgbcolor{curcolor}{0 0 0}
\pscustom[linewidth=0.80000001,linecolor=curcolor,linestyle=dashed,dash=2.4 2.4]
{
\newpath
\moveto(268.253057,323.46520781)
\lineto(268.253057,41.23989581)
}
}
{
\newrgbcolor{curcolor}{0 0 0}
\pscustom[linewidth=0.80000001,linecolor=curcolor,linestyle=dashed,dash=2.4 2.4]
{
\newpath
\moveto(159.636407,41.23989581)
\lineto(159.636407,323.46518781)
}
}
{
\newrgbcolor{curcolor}{0 0 0}
\pscustom[linewidth=0.80000001,linecolor=curcolor,linestyle=dashed,dash=2.4 2.4]
{
\newpath
\moveto(17.360282,41.23989581)
\lineto(17.360282,323.46520781)
}
}
{
\newrgbcolor{curcolor}{0 0 0}
\pscustom[linewidth=1.08071089,linecolor=curcolor]
{
\newpath
\moveto(18.073601,41.61772581)
\lineto(269.809807,40.25051581)
}
}
{
\newrgbcolor{curcolor}{0 0 0}
\pscustom[linewidth=1.46649683,linecolor=curcolor,linestyle=dashed,dash=4.39949053 4.39949053]
{
\newpath
\moveto(660.627627,318.09554581)
\lineto(660.627627,35.87023581)
}
}
{
\newrgbcolor{curcolor}{0 0 0}
\pscustom[linewidth=0.80000001,linecolor=curcolor,linestyle=dashed,dash=2.4 2.4]
{
\newpath
\moveto(473.305157,41.23989581)
\lineto(473.305157,323.46520781)
}
}
{
\newrgbcolor{curcolor}{0 0 0}
\pscustom[linewidth=0.80000001,linecolor=curcolor,linestyle=dashed,dash=2.4 2.4]
{
\newpath
\moveto(360.837117,41.23989581)
\lineto(360.837117,323.46520781)
}
}
{
\newrgbcolor{curcolor}{0 0 0}
\pscustom[linewidth=1.15019548,linecolor=curcolor]
{
\newpath
\moveto(361.578677,39.56702581)
\lineto(665.733997,38.28525581)
}
}
{
\newrgbcolor{curcolor}{0 0 0}
\pscustom[linewidth=0.82122475,linecolor=curcolor,linestyle=dashed,dash=2.46367434 2.46367434]
{
\newpath
\moveto(362.325197,40.12077581)
\lineto(272.820817,40.12077581)
}
}
{
\newrgbcolor{curcolor}{0 0 0}
\pscustom[linewidth=2,linecolor=curcolor]
{
\newpath
\moveto(141.881767,293.54776581)
\curveto(141.881767,293.54776581)(149.883957,282.14883581)(154.516817,293.54776581)
\curveto(159.149657,304.94668581)(159.886707,301.68985581)(159.886707,301.68985581)
}
}
{
\newrgbcolor{curcolor}{0 0 0}
\pscustom[linewidth=2,linecolor=curcolor]
{
\newpath
\moveto(268.340477,180.63114581)
\curveto(268.340477,180.63114581)(269.316207,183.84512581)(271.365227,181.91672581)
\curveto(273.414267,179.98834581)(275.170577,173.17470581)(278.195327,176.64580581)
\curveto(281.220087,180.11690581)(282.390957,178.95987581)(282.390957,178.95987581)
}
}
{
\newrgbcolor{curcolor}{0 0 0}
\pscustom[linewidth=0.80000001,linecolor=curcolor,linestyle=dashed,dash=2.4 2.4]
{
\newpath
\moveto(140.947527,41.23989581)
\lineto(140.947527,323.46520781)
}
}
{
\newrgbcolor{curcolor}{0 0 0}
\pscustom[linewidth=0.80000001,linecolor=curcolor,linestyle=dashed,dash=2.4 2.4]
{
\newpath
\moveto(282.639507,41.23989581)
\lineto(282.639507,323.46518781)
}
}
{
\newrgbcolor{curcolor}{0 0 0}
\pscustom[linewidth=0.80000001,linecolor=curcolor,linestyle=dashed,dash=2.4 2.4]
{
\newpath
\moveto(369.242827,41.23989581)
\lineto(369.242827,323.46520781)
}
}
{
\newrgbcolor{curcolor}{0 0 0}
\pscustom[linewidth=0.80000001,linecolor=curcolor,linestyle=dashed,dash=2.4 2.4]
{
\newpath
\moveto(465.523887,41.23989581)
\lineto(465.523887,323.46520781)
}
}
{
\newrgbcolor{curcolor}{0 0 0}
\pscustom[linewidth=1.59905803,linecolor=curcolor]
{
\newpath
\moveto(360.902937,123.74071581)
\curveto(360.902937,123.74071581)(361.492217,127.14262581)(362.729687,125.10146581)
\curveto(363.967177,123.06033581)(365.027867,115.84829581)(366.854617,119.52235581)
\curveto(368.681377,123.19641581)(369.388497,121.97172581)(369.388497,121.97172581)
}
}
{
\newrgbcolor{curcolor}{0 0 0}
\pscustom[linewidth=1.454656,linecolor=curcolor]
{
\newpath
\moveto(466.278547,126.68205581)
\curveto(466.278547,126.68205581)(466.768177,129.59789581)(467.796417,126.57322581)
\curveto(468.824657,123.54856581)(469.705997,115.51563581)(471.223867,117.71073581)
\curveto(472.741747,119.90585581)(473.329307,118.11939581)(473.329307,118.11939581)
}
}
{
\newrgbcolor{curcolor}{0 0 0}
\pscustom[linewidth=0.60000002,linecolor=curcolor]
{
\newpath
\moveto(19.436016,151.81925581)
\curveto(19.436016,151.81925581)(56.308789,130.13069581)(82.646477,143.14383581)
\curveto(108.984177,156.15697581)(140.589397,143.14383581)(140.589397,143.14383581)
}
}
{
\newrgbcolor{curcolor}{0 0 0}
\pscustom[linewidth=1.00463331,linecolor=curcolor,linestyle=dashed,dash=1.00463331 2.00926661]
{
\newpath
\moveto(12.295301,145.81947581)
\lineto(664.774767,145.69570581)
}
}
{
\newrgbcolor{curcolor}{0 0 0}
\pscustom[linewidth=0.60000002,linecolor=curcolor]
{
\newpath
\moveto(159.820427,136.13884581)
\curveto(159.820427,136.13884581)(192.841357,162.31246581)(216.427727,146.60828581)
\curveto(240.014117,130.90410581)(268.317757,146.60828581)(268.317757,146.60828581)
}
}
{
\newrgbcolor{curcolor}{0 0 0}
\pscustom[linewidth=0.60000002,linecolor=curcolor]
{
\newpath
\moveto(367.457737,162.26926581)
\curveto(367.457737,162.26926581)(396.703657,146.95122581)(417.593597,150.47132581)
\curveto(438.483547,153.99141581)(463.551457,143.69110581)(463.551457,143.69110581)
}
}
{
\newrgbcolor{curcolor}{0 0 0}
\pscustom[linewidth=0.61028296,linecolor=curcolor]
{
\newpath
\moveto(660.767967,151.05379581)
\curveto(660.767967,151.05379581)(659.111427,151.71157581)(656.153087,152.74475581)
\curveto(653.199597,153.77624581)(632.268197,154.58806581)(626.908167,153.87331581)
\curveto(595.443967,149.67757581)(601.422777,148.06467581)(590.536357,149.78463581)
\curveto(584.605877,150.72159581)(571.379067,152.84117581)(565.087367,153.18136581)
\curveto(553.892747,153.78663581)(552.259777,148.99145581)(541.629767,146.34771581)
\curveto(524.691067,142.13497581)(527.728877,140.08627581)(513.071677,139.98263581)
\curveto(491.241297,139.82827581)(476.713337,159.27650581)(476.713337,159.27650581)
}
}
\end{pspicture}

\caption{The approximation error}
\label{fig:overall-error}
\end{figure}

According to our previous notation, $\nsmooth=\fun-\sing$ is a $d$-times
continuously differentiable everywhere function. The term $\left|\fun_{\startcoeff}-\fun\right|$
is easily seen to be the usual Fourier truncation error of $\nsmooth$,
since\[
\begin{split}\left|\fun_{\startcoeff}(y)-\fun(y)\right| & =\left|\sum_{\left|k\right|\leq\startcoeff}\left(\fc(\fun)-\fc(\sing)\right)\ee^{\imath ky}+\sing(y)-f(y)\right|=\left|\sum_{\left|k\right|\leq\startcoeff}\fc(\nsmooth)\ee^{\imath ky}-\nsmooth(y)\right|=\left|\sum_{\left|k\right|>\startcoeff}\fc(\nsmooth)\ee^{\imath ky}\right|\end{split}
\]
Now recall that the Fourier coefficients of $\nsmooth\in C^{d}$ are
bounded by \eqref{eq:fcsmoothbound}, therefore\begin{equation}
\left|\fun_{\startcoeff}(y)-\fun(y)\right|\leq\C\cdot\fcsmoothbound\cdot\startcoeff^{-d-1}\label{eq:truncation-error-smooth}\end{equation}

Let $\singerr\isdef\nn{\sing}-\sing$ denote the ``singular error
function'' (see \prettyref{fig:overall-error}). Then the second
term can be written as:\[
\begin{split}\left|\nn{\fun}(y)-\fun_{\startcoeff}(y)\right| & =\left|\sum_{\left|k\right|\leq\startcoeff}\left(\fc(\sing)-\fc(\widetilde{\sing})\right)\ee^{\imath ky}+\left(\nn{\sing}(y)-\sing(y)\right)\right|\end{split}
=\left|\sum_{\left|k\right|\leq\startcoeff}\fc(\singerr)\ee^{\imath ky}-\singerr\left(y\right)\right|\]
Write \[
\begin{split}\nn{\jp}_{j}=\jp_{j}+\alpha\left(\startcoeff\right)\qquad & \left|\alpha\left(M\right)\right|\leq F_{\alpha}\left(\jcbound,\jccbound,\fsbbn\right)\cdot\startcoeff^{-\ord-2}\\
\nn{\jc}_{l,j}=\jc_{l,j}+\beta_{l}\left(\startcoeff\right)\qquad & \left|\beta_{l}\left(\startcoeff\right)\right|\leq F_{\beta}\left(\jcbound,\jccbound,\fsbbn\right)\cdot\startcoeff^{l-\ord-1}\end{split}
\]
where $F_{\alpha}$ and $F_{\beta}$ are provided by \prettyref{thm:localization-preserves-accuracy}.
For every $\epsilon<r$, we define\[
U_{l,\epsilon}\left(y\right)\isdef V_{l}\left(y;\jp_{j}+\epsilon\right)-V_{l}\left(y;\jp_{j}\right)\]
 Using the formula \eqref{eq:sing-part-explicit-bernoulli} we therefore
have \[
\begin{split}\singerr\left(y\right) & =\sum_{j=1}^{\np}\sum_{l=0}^{\ord}\left\{ \nn{\jc}_{l,j}V_{l}\left(y;\nn{\jp}_{j}\right)-\jc_{l,j}V_{l}\left(y;\jp_{j}\right)\right\} =\sum_{j=1}^{\np}\sum_{l=0}^{\ord}\left\{ \left(\jc_{l,j}+\beta_{l}\left(\startcoeff\right)\right)\left(V_{l}\left(y;\jp_{j}\right)+U_{l,\alpha\left(\startcoeff\right)}\left(y\right)\right)-\jc_{l,j}V_{l}\left(y;\jp_{j}\right)\right\} \\
 & =\underbrace{\sum_{j=1}^{\np}\sum_{l=0}^{\ord}\beta_{l}\left(\startcoeff\right)V_{l}\left(y;\jp_{j}\right)}_{\isdef Z\left(y\right)}+\underbrace{\sum_{j=1}^{\np}\sum_{l=0}^{\ord}\nn{\jc}_{l,j}U_{l,\alpha\left(\startcoeff\right)}\left(y\right)}_{\isdef W\left(y\right)}\end{split}
\]

and so\begin{equation}
\left|\nn{\fun}(y)-\fun_{\startcoeff}(y)\right|\leq\left|\sum_{\left|k\right|>\startcoeff}\fc\left(Z\right)\ee^{\imath ky}\right|+\left|\sum_{\left|k\right|<\startcoeff}\fc\left(W\right)\ee^{\imath ky}-W\left(y\right)\right|\label{eq:final-error-z-w-sum}\end{equation}
The functions $V_{l}$ belong to $C^{l}$, and therefore by the well-known
estimate (see also \eqref{eq:best-approximation-smooth}), there exist
constants $S_{l}$ such that\[
\left|\sum_{\left|k\right|>\startcoeff}\fc\left(V_{l}\right)\ee^{\imath ky}\right|\leq S_{l}\cdot\startcoeff^{-l}\]

and therefore\begin{equation}
\left|\sum_{\left|k\right|>\startcoeff}\fc\left(Z\right)\ee^{\imath ky}\right|\leq\C\cdot F_{\beta}\cdot\startcoeff^{-\ord-1}\label{eq:z-truncation-error}\end{equation}

Let us now investigate $W\left(y\right)$. Let $\jcallboundu$ denote
an upper bound for the magnitudes of the jumps:\[
\left|\jc_{l,j}\right|<\jcallboundu\qquad j=1,\dots,\np;\quad l=0,1,\dots,\ord_{1}\]
 Clearly the functions $U_{l,\epsilon}\left(y\right)$ satisfy:
\begin{enumerate}
\item Since $y\in\jfr$, then $\left|U_{l,\epsilon}\left(y\right)\right|\leq\Cl{bern}\epsilon$
for some absolute constant $\Cr{bern}$. This bound can be obtained
by just Taylor-expanding the functions $V_{l}\left(y;\jp_{j}+\epsilon\right)$
at $\epsilon=0$. In particular for $\epsilon=\alpha\left(\startcoeff\right)$
we have\begin{equation}
\left|W\left(y\right)\right|\leq\sum_{j=1}^{\np}\sum_{l=0}^{\ord}\left|\nn{\jc}_{l,j}\right|\left|U_{l,\alpha\left(\startcoeff\right)}\left(y\right)\right|\leq\C\cdot\jcallboundu\cdot F_{\alpha}\cdot\startcoeff^{-\ord-2}\label{eq:w-sum1}\end{equation}

\item In the ``no man's land'' of length $\alpha\left(\startcoeff\right)$
between $\jp_{j}$ and $\nn{\jp_{j}}$, $U_{l,\epsilon}$ is bounded
by $\C\cdot\jcallboundu$. Furthermore, as we have just seen, in the
flat regions $U_{l,\alpha\left(\startcoeff\right)}$ is bounded by
$\Cr{bern}F_{\alpha}\startcoeff^{-\ord-2}$. Therefore the Fourier
coefficients of $W$ are certainly bounded by\[
\left|\fc\left(W\right)\right|\leq\C\cdot\jcallboundu\cdot F_{\alpha}\cdot\startcoeff^{-\ord-2}\]
and so\begin{equation}
\left|\sum_{\left|k\right|<\startcoeff}\fc\left(W\right)\ee^{\imath ky}\right|\leq\C\cdot\jcallboundu\cdot F_{\alpha}\cdot\startcoeff^{-\ord-1}\label{eq:w-sum2}\end{equation}

\end{enumerate}
Combining \eqref{eq:final-error-breakup}, \eqref{eq:truncation-error-smooth},
\eqref{eq:final-error-z-w-sum}, \eqref{eq:z-truncation-error}, \eqref{eq:w-sum1}and
\eqref{eq:w-sum2} completes the proof. \qedhere

\end{proof}

\section{Numerical results\label{sec:numerical-results}}

In this section we present results of various numerical simulations
whose primary goal is to validate the asymptotic accuracy predictions
for large $\startcoeff$. We have used a straightforward implementation
and made no attempt to optimize it further. In particular, the Fourier
coefficients are assumed to be known with arbitrary precision (this
is important for the localization, see below).

\subsection{Recovery of a single jump}

Given $d,d_{1}$ and $\startcoeff$, a piecewise function with one
discontinuity is generated according to the formula\[
\fun(x)=\sum_{l=0}^{\ord_{1}}\jc_{l}V_{l}(x;\jp)+\sum_{k=-\startcoeff}^{\startcoeff}f_{k}\ee^{\imath kx}\]
where the numbers $\jp\in[-\pi,\pi],\left\{ \jc_{l}\in\reals\right\} _{l=0}^{d_{1}}$
and $\left\{ f_{k}\in\complexfield\right\} _{k=-\startcoeff}^{\startcoeff}$
are chosen at random, such that $f_{k}\sim k^{-d_{1}-2}$ and $f_{-k}=\bar{f_{k}}$.
The Fourier coefficients are calculated with the exact formula\[
\fc(\fun)=\frac{\ee^{-\imath k\jp}}{2\pi}\sum_{l=0}^{d_{1}}\frac{\jc_{l}}{\left(\imath k\right)^{l+1}}+f_{k}\]

These coefficients are then passed to the reconstruction routine for
a single jump, of order $d$. This routine implements \prettyref{alg:rec-single-jump}
in a standard MATLAB environment with double-precision calculations.

\begin{figure}
\subfloat[Accuracy of reconstruction with $d=3$ and $d_{1}=11$, as a function
of $\startcoeff$.]{\includegraphics[scale=0.5]{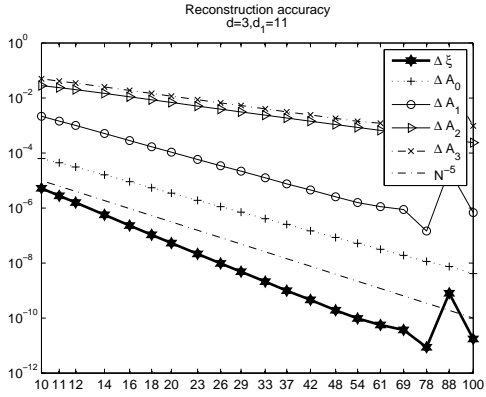}

}\subfloat[The roots of $q_{\startcoeff}^{d}(z)$]{\includegraphics[scale=0.5]{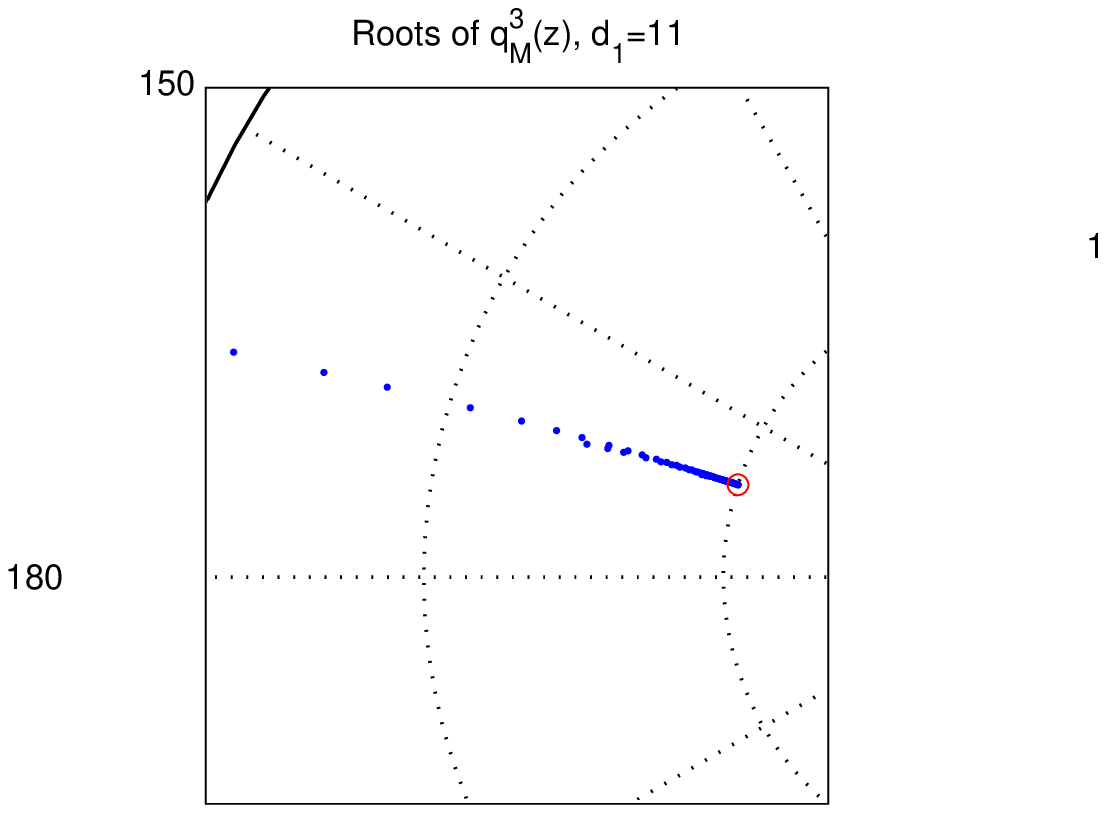}

}

\caption{Reconstruction of a single jump}
\label{fig:single-jump}
\end{figure}

\begin{figure}
\subfloat[$d_{1}=8$]{\includegraphics[scale=0.5]{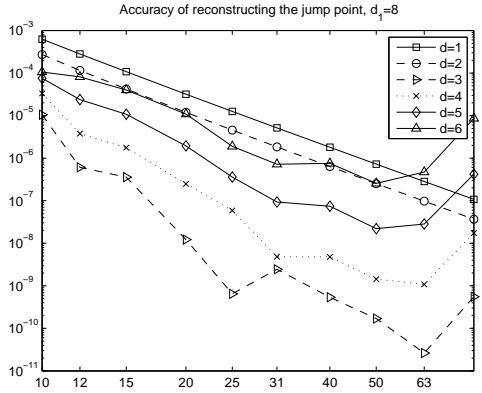}

}\subfloat[$d_{1}=12$]{\includegraphics[scale=0.5]{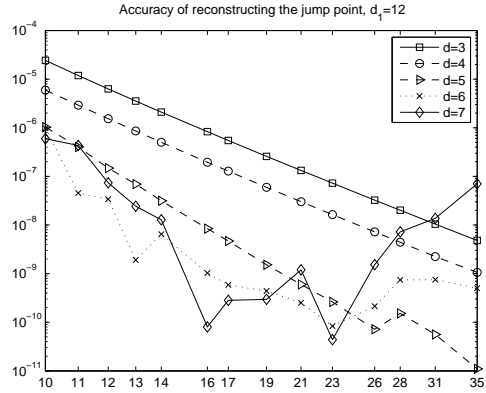}

}

\caption{Dependence of the accuracy on the order with fixed smoothness, with
increasing $\startcoeff$.}
\label{fig:single-jump-fixed-smoothness}
\end{figure}

\begin{figure}
\subfloat[$d=3$]{\includegraphics[scale=0.5]{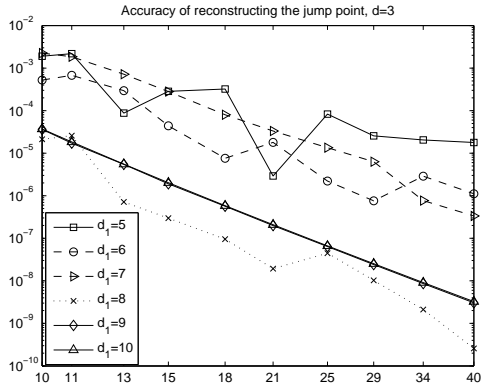}

}\subfloat[$d=4$]{\includegraphics[scale=0.5]{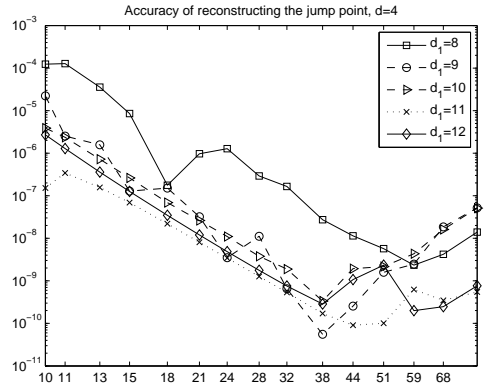}

}

\caption{Dependence of the accuracy on the smoothness with fixed order, with
increasing $\startcoeff$.}
\label{fig:single-jump-fixed-order}
\end{figure}

The following experiments were carried out:
\begin{enumerate}
\item Keeping $d$ and $d_{1}$ fixed, compare the accuracy of recovering
the jump location and all the jump magnitudes for different values
of $\startcoeff$. The results can be seen in \prettyref{fig:single-jump}.
We also plot the distribution of roots of the corresponding polynomials
$q_{\startcoeff}^{d}(z)$ - compare with \prettyref{fig:root-geometry}.
\item Keeping $d_{1}$ fixed, compare the accuracy for different reconstruction
orders $d=1,\dots,d_{1}$. The results are presented in \prettyref{fig:single-jump-fixed-smoothness}.
\item Keeping the reconstruction order $d$ fixed, compare the accuracy
for different smoothness values $d_{1}$. The results are presented
in \prettyref{fig:single-jump-fixed-order}.
\end{enumerate}
The optimality of $d=\frac{d_{1}}{2}-1$, as well as the asymptotic
order of convergence, are clearly seen to fit the theoretical predictions.
The instability and eventual breakup of the measured accuracy for
large values of $\startcoeff$ is due to the finite-precision calculations.

\subsection{Localization}

We have restricted ourselves to the following simplified setting:
the function has two jumps at $\jp_{1}=0$ and $\jp_{2}=3$, and we
localize the jump at the origin by a bump having width $\frac{8}{3}$
around the initial approximation $\widehat{\jp_{1}}=\frac{1}{40}$.
The explicit formulas for the Fourier coefficients of the bump are
derived in \prettyref{app:mollifier-explicit-construction}. We have
used Mathematica in order to carry out the computations with arbitrary
precision.

The results can be seen in \prettyref{fig:localization}. Localization
convergence can clearly be seen here, although it starts from very
large coefficients.

\begin{figure}
\subfloat[The case $\ord=1,$ $\ord_{1}=6$]{\includegraphics{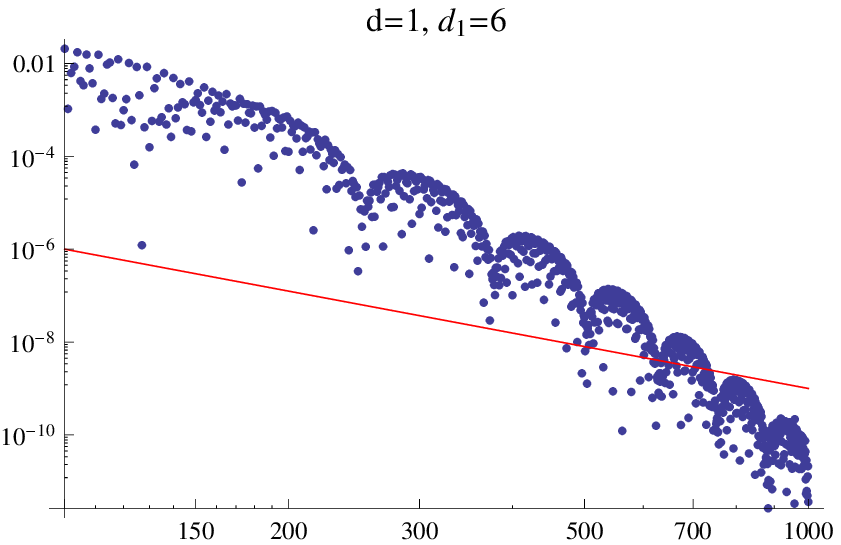}

}\subfloat[The case $\ord=2$, $\ord_{1}=6$]{\includegraphics{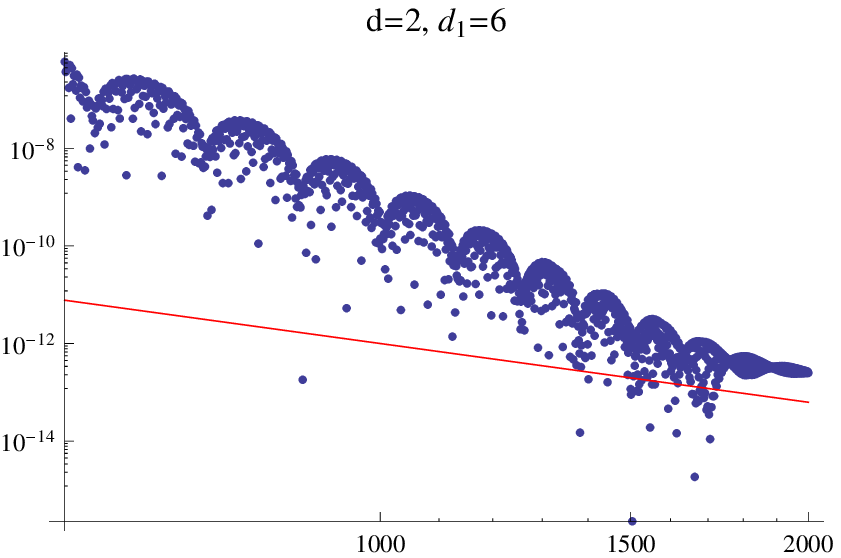}

}

\caption{Localization: accuracy of recovering the jump. The predicted accuracy
$\startcoeff^{-d-2}$ is drawn for comparison.}
\label{fig:localization}
\end{figure}

\section{Discussion}

In this paper we have demonstrated that nonlinear Fourier reconstruction
of piecewise-smooth functions can achieve accuracy with asymptotic
order of at least half the order of smoothness. As indicated by our
theoretical results as well as the numerical simulations, a reconstruction
method whose order is more than half the order of smoothness becomes
less accurate. So it appears that the algebraic approach has certain
limitations, and the interesting question is whether these limitations
are inherent or superficial. We hope that our results may provide
a clue towards obtaining sharp upper bounds.

In addition, it seems that Eckhoff's conjecture is false as stated
in \cite{eckhoff1995arf}, namely that the jumps of a piecewise-smooth
$C^{\ord}$ function can be reconstructed with accuracy $k^{-\ord-2}$.
Using a method of highest possible order doesn't take into account
the stiffness of the problem. In fact, it can be shown that the Lipschitz
constant of the solution map $\left\{ \fc(\fun)\right\} _{k=\startcoeff}^{\startcoeff+\ord+1}\to\left\{ \jp_{j},\jc_{l,j}\right\} $
of order $\ord$ is proportional to $\startcoeff^{d}$. We plan to
present these results elsewhere.

Hopefully, our analysis can be related to the algebraic reconstruction
schemes of Kvernadze and Banerjee\&Geer as well.

We would like to point out the connection of the algebraic system
\eqref{eq:singular-fourier-explicit} as well as the well-known Prony
system of equations  \eqref{eq:prony-system} (which plays a central
role in many branches of mathematics - see \cite{sig_ack} and \cite{lyubich2004sylvester})
to other recent nonlinear reconstruction methods in Signal Processing,
in particular: finite rate of innovation techniques \cite{vetterli2002ssf,dragotti2007sma},
reconstruction of shapes from moments \cite{gustafsson2000rpd,golub2000snm}
and piecewise $D$-finite moment inversion \cite{bat2008,batenkov2009arp}.
We therefore hope that our results can be extended to these subjects
as well.

\bibliographystyle{plain}
\bibliography{../../bibliography/all-bib}

\begin{thebibliography}{10}

\bibitem{abramowitz1965handbook}
M.~Abramowitz and I.A. Stegun.
\newblock {Handbook of mathematical functions: with formulas, graphs, and
  mathematical tables}.
\newblock 1965.

\bibitem{arandiga2005interpolation}
F.~Arandiga, A.~Cohen, R.~Donat, and N.~Dyn.
\newblock {Interpolation and approximation of piecewise smooth functions}.
\newblock {\em SIAM Journal on Numerical Analysis}, 43:41, 2005.

\bibitem{banerjee1998exponentially}
N.S. Banerjee and J.F. Geer.
\newblock {Exponentially accurate approximations to periodic Lipschitz
  functions based on Fourier series partial sums}.
\newblock {\em Journal of Scientific Computing}, 13(4):419--460, 1998.

\bibitem{barkhudaryan2007asymptotic}
A.~Barkhudaryan, R.~Barkhudaryan, and A.~Poghosyan.
\newblock {Asymptotic behavior of Eckhoff's method for Fourier series
  convergence acceleration}.
\newblock {\em Analysis in Theory and Applications}, 23(3):228--242, 2007.

\bibitem{marchbarone:2000}
P.~Barone and R.~March.
\newblock {Reconstruction of a Piecewise Constant Function from Noisy Fourier
  Coefficients by Pad{\'e} Method}.
\newblock {\em SIAM Journal on Applied Mathematics}, 60(4):1137--1156, 2000.

\bibitem{bat2008}
D.~Batenkov.
\newblock {Moment inversion problem for piecewise D-finite functions}.
\newblock {\em Inverse Problems}, 25(10):105001, October 2009.

\bibitem{batenkov2009arp}
D.~Batenkov, N.~Sarig, and Y.~Yomdin.
\newblock {An ``algebraic'' reconstruction of piecewise-smooth functions from
  integral measurements}.
\newblock {\em Proc. of Sampling Theory and Applications (SAMPTA)}, 2009.
\newblock Arxiv preprint arXiv:0901.4659.

\bibitem{bauer-band}
R.~Bauer.
\newblock {\em {Band filters for determining shock locations}}.
\newblock PhD thesis, Department of Applied Mathematics, Brown University,
  Providence, RI, 1995.

\bibitem{beckermann2008rgp}
B.~Beckermann, A.C. Matos, and F.~Wielonsky.
\newblock {Reduction of the Gibbs phenomenon for smooth functions with jumps by
  the $\varepsilon$-algorithm}.
\newblock {\em Journal of Computational and Applied Mathematics},
  219(2):329--349, 2008.

\bibitem{Boyd20091404}
John~P. Boyd.
\newblock Acceleration of algebraically-converging fourier series when the
  coefficients have series in powers of 1/n.
\newblock {\em Journal of Computational Physics}, 228(5):1404 -- 1411, 2009.

\bibitem{brezinski2004extrapolation}
C.~Brezinski.
\newblock {Extrapolation algorithms for filtering series of functions, and
  treating the Gibbs phenomenon}.
\newblock {\em Numerical Algorithms}, 36(4):309--329, 2004.

\bibitem{dragotti2007sma}
P.L. Dragotti, M.~Vetterli, and T.~Blu.
\newblock {Sampling Moments and Reconstructing Signals of Finite Rate of
  Innovation: Shannon meets Strang-Fix}.
\newblock {\em IEEE Transactions on Signal Processing}, 55(5):1741, 2007.

\bibitem{driscoll2001pba}
T.A. Driscoll and B.~Fornberg.
\newblock {A Pad{\'e}-based algorithm for overcoming the Gibbs phenomenon}.
\newblock {\em Numerical Algorithms}, 26(1):77--92, 2001.

\bibitem{eckhoff1995arf}
K.S. Eckhoff.
\newblock {Accurate reconstructions of functions of finite regularity from
  truncated Fourier series expansions}.
\newblock {\em Mathematics of Computation}, 64(210):671--690, 1995.

\bibitem{eckhoff1998high}
K.S. Eckhoff.
\newblock {On a high order numerical method for functions with singularities}.
\newblock {\em Mathematics of Computation}, 67(223):1063--1088, 1998.

\bibitem{elaydi2005ide}
S.~Elaydi.
\newblock {\em {An Introduction to Difference Equations}}.
\newblock Springer, 2005.

\bibitem{ettinger2008lvn}
B.~Ettinger, N.~Sarig, and Y.~Yomdin.
\newblock {Linear versus Non-Linear Acquisition of Step-Functions}.
\newblock {\em Journal of Geometric Analysis}, 18(2):369--399, 2008.

\bibitem{gautschi1974nei}
W.~Gautschi.
\newblock {Norm estimates for inverses of Vandermonde matrices}.
\newblock {\em Numerische Mathematik}, 23(4):337--347, 1974.

\bibitem{gelb-segmentation}
A.~Gelb and D.~Cates.
\newblock {Segmentation of Images from Fourier Spectral Data}.
\newblock {\em Commun. Comput. Phys.}, 5:326--349, 2009.

\bibitem{gelb1999detection}
A.~Gelb and E.~Tadmor.
\newblock {Detection of edges in spectral data}.
\newblock {\em Applied and computational harmonic analysis}, 7(1):101, 1999.

\bibitem{golub2000snm}
G.H. Golub, P.~Milanfar, and J.~Varah.
\newblock {A Stable Numerical Method for Inverting Shape from Moments}.
\newblock {\em SIAM Journal on Scientific Computing}, 21(4):1222--1243, 2000.

\bibitem{gottlieb1997gpa}
D.~Gottlieb and C.W. Shu.
\newblock {On the Gibbs phenomenon and its resolution}.
\newblock {\em SIAM Review}, pages 644--668, 1997.

\bibitem{guilpin2004eps}
C.~Guilpin, J.~Gacougnolle, and Y.~Simon.
\newblock {The $\varepsilon$-algorithm allows to detect Dirac delta functions}.
\newblock {\em Applied Numerical Mathematics}, 48(1):27--40, 2004.

\bibitem{gustafsson2000rpd}
B.~Gustafsson, C.~He, P.~Milanfar, and M.~Putinar.
\newblock {Reconstructing planar domains from their moments}.
\newblock {\em Inverse Problems}, 16(4):1053--1070, 2000.

\bibitem{kantokryl62}
L.V. Kantorovich and V.I. Krylov.
\newblock {\em Approximate Methods of Higher Analysis}.
\newblock Fizmatgiz, Moscow, 1962.

\bibitem{kvernadze2004ajd}
G.~Kvernadze.
\newblock {Approximating the jump discontinuities of a function by its
  Fourier-Jacobi coefficients}.
\newblock {\em Mathematics of Computation}, 73(246):731--752, 2004.

\bibitem{YaronLipman07082009}
Yaron Lipman and David Levin.
\newblock {Approximating piecewise-smooth functions}.
\newblock {\em IMA J Numer Anal}, page drn087, 2009.

\bibitem{lyubich2004sylvester}
Y.I. Lyubich.
\newblock {The Sylvester-Ramanujan System of Equations and The Complex Power
  Moment Problem}.
\newblock {\em The Ramanujan Journal}, 8(1):23--45, 2004.

\bibitem{mhaskar2000polynomial}
HN~Mhaskar and J.~Prestin.
\newblock {Polynomial frames for the detection of singularities}.
\newblock In {\em Wavelet Analysis and Multiresolution Methods: Proceedings of
  the Conference Held at University of Illinois at Urbana-Champaign, Illinois},
  page 273. CRC, 2000.

\bibitem{Natanson1949}
I.P. Natanson.
\newblock {\em Constructive Function Theory (in Russian)}.
\newblock Gostekhizdat, 1949.

\bibitem{Plaskota2009227}
L.~Plaskota and G.W. Wasilkowski.
\newblock The power of adaptive algorithms for functions with singularities.
\newblock {\em Journal of Fixed Point Theory and Applications}, 6(2):227--248,
  2009.

\bibitem{prony1795essai}
R.~Prony.
\newblock {Essai experimental et analytique}.
\newblock {\em J. Ec. Polytech.(Paris)}, 2:24--76, 1795.

\bibitem{sig_ack}
N.~Sarig and Y.~Yomdin.
\newblock {Signal Acquisition from Measurements via Non-Linear Models}.
\newblock {\em Mathematical Reports of the Academy of Science of the Royal
  Society of Canada}, 29(4):97--114, 2008.

\bibitem{szegHo1975op}
G.~Szeg\H{o}.
\newblock {\em {Orthogonal Polynomials}}.
\newblock American Mathematical Society, 1975.

\bibitem{vetterli2002ssf}
M.~Vetterli, P.~Marziliano, and T.~Blu.
\newblock {Sampling signals with finite rate of innovation}.
\newblock {\em IEEE Transactions on Signal Processing}, 50(6):1417--1428, 2002.

\bibitem{vindas2009local}
J.~Vindas.
\newblock {\em {Local Behavior of Distributions and Applications}}.
\newblock PhD thesis, 2009.

\bibitem{wei2007detection}
M.~Wei, A.G. Mart{\'\i}nez, and A.R. De~Pierro.
\newblock {Detection of edges from spectral data: New results}.
\newblock {\em Applied and Computational Harmonic Analysis}, 22(3):386--393,
  2007.

\bibitem{wilkinson1994rounding}
J.H. Wilkinson.
\newblock {\em {Rounding errors in algebraic processes}}.
\newblock Dover Pubns, 1994.

\bibitem{zygmund1959trigonometric}
A.~Zygmund.
\newblock {\em {Trigonometric Series. Vols. I, II}}.
\newblock Cambridge University Press, New York, 1959.

\end{thebibliography}

\appendix

\section{\label{app:difference-calculus-and-related-results}Discrete difference
calculus and related results}

In this appendix we provide proofs of several combinatorial auxiliary
results.

Let $\shift$ denote the discrete ``shift'' operator in $k$, i.e.
for every function $g(k):\reals\to\reals$ we have\[
\shift g(k)\isdef g(k+1)\]
Furthermore, let $\Delta$ denote the discrete difference operator,
i.e. $\Delta=\shift-\id$ where $\id$ is the identity operator. Then
by the binomial theorem we have \begin{equation}
\Delta^{d}g(k)=\left(\shift-\id\right)^{d}g(k)=(-1)^{d}\sum_{j=0}^{d}(-1)^{j}\binom{d}{j}g(k+j)\label{eq:difference-operator-power}\end{equation}

\begin{lem}
\label{lem:polynomial-delta}Let $p(k)=a_{0}k^{n}+a_{1}k^{n-1}+\dots+a_{n}$
be a polynomial of degree $n$. Then\[
\begin{split}\Delta^{n}p(k) & =a_{0}n!\\
\Delta^{n+1}p(k) & =0\end{split}
\]
\end{lem}
\begin{proof}
See e.g. \cite{elaydi2005ide}.
\end{proof}
Now assume that $g\left(k\right):\reals^{+}\to\reals$ is a given
function. Let us perform a change of variable $y=\frac{1}{k}$, and
define $G(y)\isdef g\left(\frac{1}{y}\right)=g(k)$. With this notation,
we have\[
\Delta g\left(k\right)=g(k+1)-g(k)=g\left(\frac{1}{y}+1\right)-g\left(\frac{1}{y}\right)=G\left(\frac{y}{1+y}\right)-G\left(y\right)\]
We subsequently define the ``dual'' operator $\combop$ as:\[
\combop\left\{ G(y)\right\} \isdef G\left(\frac{y}{1+y}\right)-G\left(y\right)\]

The operator $\combop$ has an interesting property of ``killing''
the lowest-order Taylor coefficient at 0.
\begin{prop}
\label{prop:L-property}Let $H(y)$ be analytic at $y=0$, such that
$H(y)=h_{m}y^{m}+h_{m+1}y^{m+1}+\dots$. Then for $n\in\naturals$,
the function $\combop^{n}\left\{ H\left(y\right)\right\} $ is analytic
at $0$ with Taylor expansion\[
\combop^{n}\left\{ H\left(y\right)\right\} =h_{m+n}^{*}y^{m+n}+h_{m+n+1}^{*}y^{m+n+1}+\dots\]
\end{prop}
\begin{proof}
The proof is by induction on $n$. The basis $n=0$ is given. Assuming
that $U\left(y\right)=\combop^{n-1}\left\{ H(y)\right\} $ is analytic
at 0 and $U\left(y\right)=u_{m+n-1}^{*}y^{m+n-1}+u_{m+n-1}^{*}y^{m+n-1}+\dots$,
consider the function $\combop\left\{ U\left(y\right)\right\} =\combop^{n}\left\{ H(y)\right\} $.
Let $z=\frac{y}{1+y}$, and so for $\left|y\right|<1$ we have\[
z\left(y\right)=y\left(1-y+y^{2}+\dots\right)=y-y^{2}+\dots\]
Making the analytic change of coordinates $y\to z\left(y\right)$
we conclude that the Taylor expansion of $U\left(z\left(y\right)\right)$
at the origin is \[
U\left(z\left(y\right)\right)=u_{m+n-1}^{*}z^{m+n-1}+\dots=u_{m+n-1}^{*}y^{m+n-1}+\dots\]
That is, the leading coefficient is the same as in the Taylor expansion
of $U\left(y\right)$. Therefore\[
\begin{split}\combop\left\{ U\left(y\right)\right\}  & =U\left(z\right)-U\left(y\right)\\
 & =u_{m+n}^{**}y^{m+n}+u_{m+n+1}^{**}y^{m+n+1}+\dots\end{split}
\]
This expansion holds in some neighborhood of the origin.\end{proof}
\begin{lem}
\label{lem:binomial-sum-fractions}Let $l,d\in\naturals$. Then there
exist positive constants $\Cl{dd-bound},\Cl[kk]{kk-dd-bound}$ such
that for all $k>\Cr{kk-dd-bound}$\[
\left|\sum_{j=0}^{d}(-1)^{j}\binom{d}{j}\frac{1}{(k+j)^{l}}\right|<\frac{\Cr{dd-bound}}{k^{d+l}}\]
\end{lem}
\begin{proof}
Let $g(k)=\frac{1}{k^{l}}$. It is easy to check (see \eqref{eq:difference-operator-power})
that\[
\begin{split}A_{l,d}(k)\isdef\sum_{j=0}^{d}(-1)^{j}\binom{d}{j}\frac{1}{(k+j)^{l}} & =\Delta^{d}g(k)\end{split}
\]

The proof of the claim is in two steps. First, we shall develop the
expression $A_{l,d}(k)$ into power series in $\frac{1}{k}$ converging
for sufficiently large $k$. Then, based on this representation we
shall establish the required bound.

Let $y=\frac{1}{k}$. According to our notation, let $G\left(y\right)=g\left(\frac{1}{y}\right)=g\left(k\right)$
and so $A_{l,d}(k)=\combop^{d}\left\{ G\left(y\right)\right\} \isdef F(y)$.
Furthermore, for all $j=0,1,\dots$ we have\[
k+j=\frac{1}{y}+j=\frac{1+jy}{y}\]
and so\[
F\left(y\right)=\sum_{j=0}^{d}\left(-1\right)^{j}{d \choose j}G\left(\frac{y}{1+jy}\right)\]
Substituting $G\left(y\right)=y^{l}$, we conclude that $F(y)$ is
a real analytic function of $y$ in the disk $\left|y\right|<\frac{1}{d}$,
and so it can be written as a converging power series\[
F(y)=\sum_{i=0}^{\infty}f_{i}y^{i}\]
Applying \prettyref{prop:L-property} to $G\left(y\right)$ we conclude
that $f_{0}=\dots=f_{l+d-1}=0$. Therefore the expansion\begin{equation}
A_{l,d}\left(k\right)=\sum_{i=l+d}^{\infty}\frac{f_{i}}{k^{i}}\label{eq:ald-exp-k}\end{equation}
holds for $k>d$.

Let us now estimate the magnitude of the coefficients $f_{i}$. Since
\eqref{eq:ald-exp-k} is valid for $k=d+1$, then there exists a constant
$\Cl{ald-coeffs-bound}$ such that $\left|f_{i}\left(d+1\right)^{-i}\right|<\Cr{ald-coeffs-bound}$
for all $i\in\naturals$ and therefore\[
\left|f_{i}\right|<\Cr{ald-coeffs-bound}\left(d+1\right)^{i}\]
But then for arbitrary $k\geq d+2$ we have\[
\begin{split}\left|A_{l,d}\left(k\right)\right| & =\left|\sum_{i=0}^{\infty}\frac{f_{l+d+i}}{k^{l+d+i}}\right|<\frac{\Cr{ald-coeffs-bound}\left(d+1\right)^{l+d}}{k^{l+d}}\sum_{i=0}^{\infty}\frac{\left(d+1\right)^{i}}{k^{i}}\\
 & \leq\frac{\Cr{ald-coeffs-bound}\left(d+1\right)^{l+d}}{k^{l+d}}\cdot\frac{1}{1-\frac{d+1}{d+2}}\leq\frac{\Cr{dd-bound}}{k^{l+d}}\qedhere\end{split}
\]
\end{proof}
\begin{lem}
\label{lem:basic-recurrence}Let $\omega\in\complexfield$, $a_{0},\dotsc,a_{n}\in\complexfield$
and $n,k\in\naturals$. Denote $b_{k}\isdef\omega^{k}\cdot\left(a_{0}+a_{1}k+\dotsc a_{n}k^{n}\right)$.
Then\[
\sum_{j=0}^{n+1}\left(-1\right)^{j}\binom{n+1}{j}b_{k+j}\omega^{n+1-j}\equiv0\]
\end{lem}
\begin{proof}
Denote\[
p(k)\isdef a_{0}+a_{1}k+\dotsc a_{n}k^{n}\]
Then we rewrite the given expression as\begin{eqnarray*}
\begin{split}\sum_{j=0}^{n+1}\left(-1\right)^{j}\binom{n+1}{j}b_{k+j}\omega^{n+1-j} & =\sum_{j=0}^{n+1}(-1)^{j}\binom{n+1}{j}p(k+j)\omega^{n+k+1}\\
 & =(-1)^{d+1}\omega^{n+k+1}\Delta^{n+1}p(k)\\
\explntext{Lemma\;\ref{lem:polynomial-delta}} & =0\end{split}
\end{eqnarray*}
The claim is therefore proved.\end{proof}
\begin{lem}
\label{lem:dimdim-special-values}Let\[
\dimdim(d,t,s)\isdef\sum_{j=s}^{d+1}(-1)^{j}\binom{j}{s}\binom{d+1}{j}j^{d-t}\]
The following statements are true for all $s=0,1,\dots,d+1$:
\begin{enumerate}
\item If $t\geq s$ then $\dimdim(d,t,s)=0$
\item $\dimdim(d,s-1,s)=(-1)^{d+1}(d+1-s)!\binom{d+1}{s}$
\end{enumerate}
\end{lem}
\begin{proof}
Let $\alpha\isdef d-s$ and $\beta\isdef d-t$. Now\[
\begin{split}\dimdim(d,t,s) & =\sum_{j=0}^{d+1-s}(-1)^{j+s}\binom{j+s}{s}\binom{d+1}{j+s}(j+s)^{d-t}\\
 & =\sum_{j=0}^{\alpha+1}(-1)^{j+s}\frac{(j+s)!}{s!j!}\cdot\frac{(d+1)!}{(j+s)!(d+1-j-s)!}(j+s)^{\beta}\\
 & =(-1)^{d-\alpha}\sum_{j=0}^{\alpha+1}(-1)^{j}\frac{(d+1)!}{(d-\alpha)!j!(\alpha+1-j)!}(j+s)^{\beta}\cdot\frac{(\alpha+1)!}{(\alpha+1)!}\\
 & =(-1)^{d-\alpha}\binom{d+1}{\alpha+1}\sum_{j=0}^{\alpha+1}(-1)^{j}\binom{\alpha+1}{j}(j+s)^{\beta}\end{split}
\]
Now let $g(s)\isdef s^{\beta}$ be a polynomial of degree $\beta$.
By \eqref{eq:difference-operator-power}, the above expression can
be rewritten as\begin{equation}
\dimdim(d,t,s)=(-1)^{d+1}\binom{d+1}{\alpha+1}\Delta^{\alpha+1}g(s)\label{eq:dimdim-through-delta}\end{equation}

\begin{enumerate}
\item Assume $t\geq s$. Then $\alpha+1\geq\beta+1$, and so by \prettyref{lem:polynomial-delta}
we have $\Delta^{\alpha+1}g(s)=0$. This completes the proof of the
first part.
\item Let $t=s-1$. By \prettyref{lem:polynomial-delta} we have $\Delta^{\alpha+1}g(s)=(\alpha+1)!$
and so by \eqref{eq:dimdim-through-delta} \[
\dimdim(d,t,s)=(-1)^{d+1}\binom{d+1}{\alpha+1}(\alpha+1)!=(-1)^{d+1}(d-s+1)!\binom{d+1}{s}\]
This completes the proof of the second part.\qedhere
\end{enumerate}
\end{proof}

\section{\label{app:misc-results}Miscellaneous auxiliary results}
\begin{thm}[Rouche's theorem]
\label{thm:rouche}Let the polynomial $q(z)\in\complexfield[z]$
be a sum $q(z)=p(z)+e(z)$. Let $z_{0}$ be a simple zero of $p(z)$.
If there exists $\rho\in\reals^{+}$ such that \[
\left|p(z)\right|>\left|e(z)\right|\qquad\forall z\in\partial B_{\rho}(z_{0})\]
then $q(z)$ has a simple zero inside $B_{\rho}(z_{0})$.\end{thm}
\begin{lem}
\label{lem:lower-bound-values-circle-via-derivatives}Let there be
given a sequence of polynomials $P_{k}(z):\complexfield\to\complexfield$
and a point $z_{0}\in\complexfield$ such that
\begin{enumerate}
\item $P_{k}(z_{0})=0$ for all $k\in\naturals$;
\item $\bigl|P'_{k}(z_{0})\bigr|\geq\Cl{derivative-bound}$ for all $k\in\naturals$
and some constant $\Cr{derivative-bound}$ independent of $k$;
\item For every fixed $k$ the following inequality holds for all $z\in B_{k^{-1}}(z_{0})$\[
\bigl|P''_{k}(z)\bigr|\leq\Cl{second-derivative-bound}k\]
where $\Cr{second-derivative-bound}$ is a constant independent of
$k$.
\end{enumerate}
Let $\rho(k):\naturals\to\reals$ satisfy\[
0<\rho(k)<\min\left(\frac{1}{k},\frac{\Cr{derivative-bound}}{\Cr{second-derivative-bound}k}\right)\]

Then there exists a constant $\Cl{generic-1}$ independent of $\rho(k)$
such that for all $k$, the following holds for every $z\in\partial B_{\rho(k)}(z_{0})$:\[
\bigl|P_{k}(z)\bigr|\geq\Cr{generic-1}\rho(k)\]

\end{lem}
\begin{proof}
Let us write the truncated Taylor expansion of $P_{k}$ around $z_{0}$
with remainder in Lagrange form:\[
P_{k}\left(z_{0}+\rho(k)e^{\imath\theta}\right)=P_{k}(z_{0})+\underbrace{P'_{k}(z_{0})\rho(k)\ee^{\imath\theta}}_{=E_{1}}+\underbrace{\frac{P''_{k}(\xi)}{2}\rho^{2}(k)\ee^{2\imath\theta}}_{=E_{2}}\]
for some $\xi\in B_{\rho}(z_{0})$. Now since $\rho(k)\leq\frac{1}{k}$
we have\[
\begin{split}\left|E_{1}\right| & \geq\Cr{derivative-bound}\rho(k)\\
\left|E_{2}\right| & \leq\frac{\Cr{second-derivative-bound}k\rho^{2}(k)}{2}\end{split}
\]

On the other hand,\[
\begin{split}\rho(k) & \leq\frac{\Cr{derivative-bound}}{\Cr{second-derivative-bound}k}\\
\frac{\Cr{second-derivative-bound}k\rho^{2}(k)}{2} & \leq\frac{\Cr{derivative-bound}\rho(k)}{2}\end{split}
\]

Therefore $\frac{\left|E_{1}\right|}{2}\geq\left|E_{2}\right|$ and
so by taking $\Cr{generic-1}\isdef\frac{\Cr{derivative-bound}}{2}$
we have $\bigl|P_{k}(z)\bigr|\geq\Cr{generic-1}\rho(k)$.\end{proof}
\begin{prop}
\label{prop:neg-exp-estimate}Let $n\in\naturals$ be given. Then
for $\left|x\right|<\frac{3}{n+2}$ the following estimate holds:\[
\left(1+x\right)^{-n}=1-nx+\frac{n\left(n+1\right)}{2}x^{2}\Cl[rr]{rr-geomtail}\left(x\right)\qquad\text{where}\quad\left|\Cr{rr-geomtail}\left(x\right)\right|<\frac{1}{1-\frac{x\left(n+2\right)}{3}}\]
In general, for approximation of order $d$, we have for $\left|x\right|<\frac{d+2}{n+d+1}$\[
\left(1+x\right)^{-n}=1-nx+\frac{n\left(n+1\right)}{2}x^{2}+\dots+\left(-1\right)^{d}\frac{n\times\cdots\times\left(n+d-1\right)}{d!}x^{d}+\left(-1\right)^{d+1}\frac{n\times\cdots\times\left(n+d\right)}{\left(d+1\right)!}x^{d+1}\Cl[rr]{rr-geomtail2}\left(x\right)\]
where\[
\left|\Cr{rr-geomtail2}\left(x\right)\right|<\frac{1}{1-\frac{\left(n+d+1\right)}{d+2}x}\]
\end{prop}
\begin{proof}
Standard majorization of the Taylor series tail by a geometric series.
\end{proof}

\section{\label{app:mollifier-explicit-construction}Explicit construction
of a bump}

In this appendix we present an explicit construction of the bump function
which we used in our numerical simulations. We also derive an explicit
bound for the size of its Fourier coefficients, to be used in the
proof of localization accuracy.

Let there be given two parameters $t$ and $E$ with $2E>t$, together
with the point $\jp\in\reals$. Our goal is to build a function $g=g_{E,t}(x;\jp$)
which satisfies the following conditions:
\begin{itemize}
\item [(G1)]$g\equiv0$ for $x\notin\left[\jp-E,\jp+E\right]$;
\item [(G2)]$g\equiv1$ for $x\in\left[\jp-\frac{t}{2},\jp+\frac{t}{2}\right]$;
\item [(G3)]$g\in C^{\infty}\left(\reals\right)$;
\item [(G4)]the Fourier coefficients of $g$ decay as rapidly as possible.
\end{itemize}
The idea is to take a standard $C^{\infty}$ mollifier, scale it and
convolve with a box function.

We define two new parameters: the scaling factor $s$ and the width
of the box $r$. Note that our construction implies $r\geq2s$, because
otherwise the result of the convolution will not have a flat segment
in the middle.

Let us therefore take the standard $C^{\infty}$ mollifier\[
\Psi(x)=\begin{cases}
\ee^{-1/(1-x^{2})} & \mbox{ for }|x|<1\\
0 & \mbox{ otherwise}\end{cases}\]
and scale it between $-s$ and $s$ for some $s>0$:\[
m_{s}(x)=\frac{1}{s\Delta}\Psi\left(\frac{x}{s}\right)\]
where\[
\Delta=\frac{1}{s}\int_{-s}^{s}\Psi\left(\frac{x}{s}\right)\dd x=\int_{-1}^{1}\Psi(y)\dd y\sim0.443994\]

Now we take a box function centered at $\jp$, having width $r$:\[
b_{r}(x;\xi)=\begin{cases}
1 & \mbox{ for }-\frac{r}{2}\leq x-\xi\leq\frac{r}{2}\\
0 & \mbox{ otherwise}\end{cases}\]

Finally we convolve the two and get a smooth bump:\[
g=g_{r,s}(x;\xi)=b_{r}(x;\xi)*m_{s}(x)=\frac{1}{s\Delta}\int_{\xi-\frac{r}{2}}^{\xi+\frac{r}{2}}\Psi\left(\frac{x-t}{s}\right)\dd t\]

The new parameters $s,r$ should be compatible with the original $E,t$.
In particular, we want to have a strip of width $t$ in the center,
and the extent of the whole bump should not exceed $E$. Therefore
we have the following compatibility conditions:\begin{equation}
\begin{split}s+\frac{t}{2} & <\frac{r}{2}\\
2s+\frac{r}{2} & <E\end{split}
\label{eq:bump-params-compatibility}\end{equation}

The function $g$ so constructed clearly satisfies conditions $\left(G1\right)-\left(G3\right)$
above. Let us now maximize the decay of its Fourier coefficients.
By definition:\[
c_{k}(g)=\frac{1}{2\pi s\Delta}\int_{-\pi}^{\pi}\ee^{-\imath kx}\Biggl\{\int_{\xi-\frac{r}{2}}^{\xi+\frac{r}{2}}\Psi\left(\frac{x-t}{s}\right)\dd t\Biggr\}\dd x\]

Notice first that $\Psi(z)$ is zero outside the region $-1\leq z\leq1$,
therefore we can make the change of variables $z\to t-x,t\to t$ and
rewrite the integral as \[
c_{k}(g)=\frac{1}{2\pi s\Delta}\int_{-s}^{s}\ee^{\imath kz}\Psi\left(\frac{z}{s}\right)\Biggl\{\int_{\xi-\frac{r}{2}}^{\xi+\frac{r}{2}}\ee^{-\imath kt}\dd t\Biggr\}\dd z\]

So now the two integrals are completely separated. Explicit calculation
gives\[
\int_{\xi-\frac{r}{2}}^{\xi+\frac{r}{2}}\ee^{-\imath kt}\dd t=-\frac{\imath\ee^{-\frac{1}{2}\imath k(r+2\xi)}\left(-1+\ee^{\imath kr}\right)}{k}\]

Now we scale back: $z=sy$ and obtain the explicit formula\begin{equation}
\begin{split}c_{k}(g) & =-\frac{\imath\ee^{-\frac{1}{2}\imath k(r+2\xi)}\left(-1+\ee^{\imath kr}\right)}{k}\cdot\frac{1}{2\pi\Delta}\int_{-1}^{1}\ee^{\imath ksy}\Psi(y)\dd y\\
 & =-\frac{\imath\ee^{-\frac{1}{2}\imath k(r+2\xi)}\left(-1+\ee^{\imath kr}\right)}{2\pi\Delta k}c_{-ks}(\Psi)\end{split}
\label{eq:bump-fc-formula}\end{equation}

Finally we would like to determine the optimal values for $s$ and
$r$ so that $\left|\fc\left(g\right)\right|$ decrease as rapidly
as possible with $k\to\infty$.

First note that since $\Psi\in C^{\infty}$, then for every $\alpha>1$
there exists a constant $\Cl{mollifier-fc-bound}\left(\alpha\right)$
such that\[
\left|\fc\left(\Psi\right)\right|\leq\Cr{mollifier-fc-bound}\cdot\left|k\right|^{-\alpha}\]

The formula \eqref{eq:bump-fc-formula} suggests that we should take
$s$ to be as large as possible. Applying the conditions \eqref{eq:bump-params-compatibility}
we get that the following values maximize $s$:\[
\begin{split}s^{*} & =\frac{1}{3}\left(E-\frac{t}{2}\right)\\
r^{*} & =\frac{2}{3}\left(E+t\right)\end{split}
\]

We have thus proved the following result:
\begin{thm}
\label{thm:bump-fourier-decay}Given $E,t$ with $2E>t$ and a point
$\jp$, let $g_{E,t}\left(x;\jp\right)$ be the bump constructed above.
Then it satisfies the conditions $\left(G1\right)-\left(G4\right)$
such that for every $\alpha>1$ there exists a constant $\Cl{bump-final}=\Cr{bump-final}\left(\alpha\right)$
such that for all $k\in\naturals$\[
\left|\fc\left(g_{E,t}\right)\right|\leq\Cr{bump-final}\cdot\left(2E-t\right)^{-\alpha}k^{-1-\alpha}\]

\end{thm}
\begin{figure}
\subfloat[The standard mollifier $\Psi(x)$]{\includegraphics{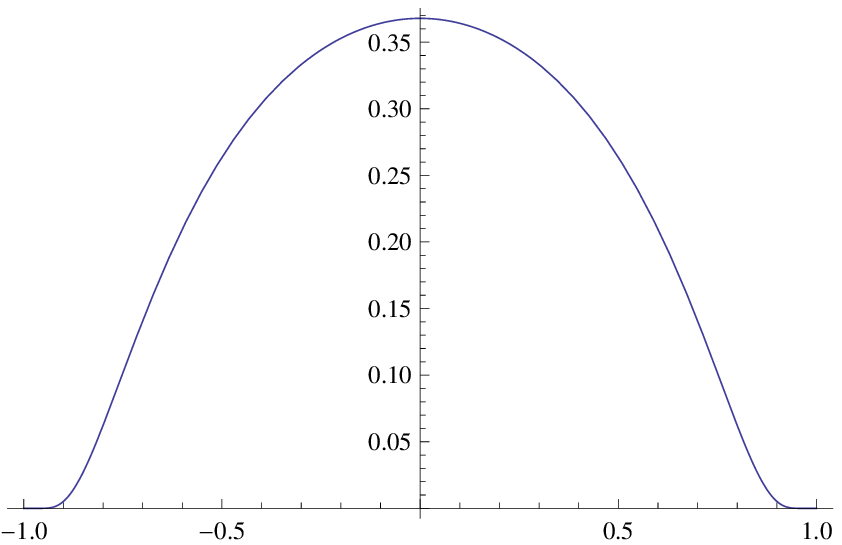}

}\subfloat[The final bump for $\jp=0$]{\includegraphics[bb=148bp 570bp 365bp 719bp,clip]{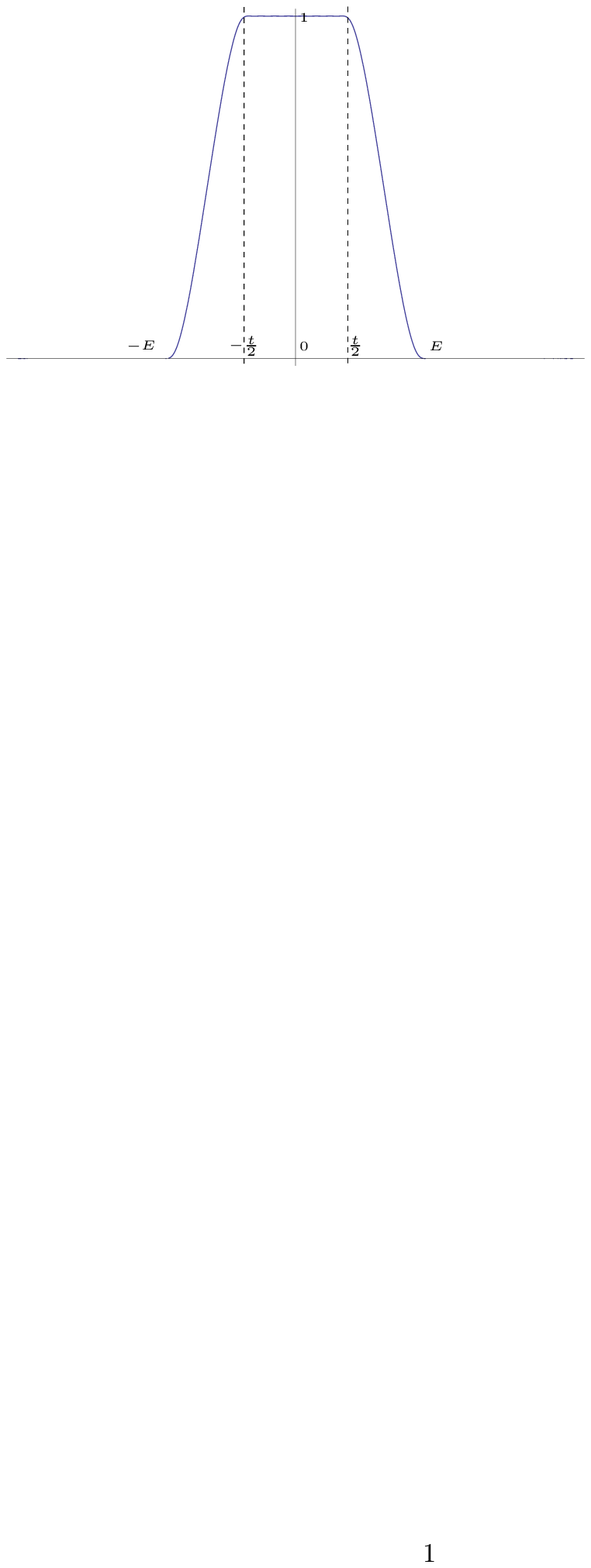}

}

\caption{The construction of the bump}
\label{fig:bump-construction}
\end{figure}

\section{\label{app:initial-estimates-prony}Initial estimates via Prony's
method}

In this appendix we present a rigorous proof that the Eckhoff's method
of order zero produces sufficiently accurate estimates of the jump
locations $\left\{ \jp_{j}\right\} $, to be used in \prettyref{alg:localization-alg}.
Denote $\omega_{j}=\ee^{-\imath\jp_{j}}$. For $\ord=0$, the system
\eqref{eq:singular-fourier-explicit} becomes\begin{equation}
\underbrace{\sum_{j=1}^{\np}\jc_{0,j}\omega_{j}^{k}}_{=m_{k}}\approx2\pi\left(\imath k\right)\fc(\fun)\label{eq:prony-system}\end{equation}

\eqref{eq:prony-system} is a well-known system of equations which
is sometimes called the Prony system (\cite{sig_ack}) or Sylvester-Ramanujan
system (\cite{lyubich2004sylvester}). The original method of solution
(due to Baron de Prony, \cite{prony1795essai}) is to exploit the
following fact.
\begin{lem}
The sequence $\left\{ m_{k}\right\} $ satisfies the recurrence relation
with constant coefficients\[
\sum_{i=0}^{\np}m_{k+i}q_{i}=0\]
where $\left\{ q_{i}\right\} $ are the coefficients of the polynomial\[
Q(z)\isdef\prod_{j=1}^{\np}\left(z-\omega_{j}\right)=\sum_{i=0}^{\np}q_{i}z^{i}\]
\end{lem}
\begin{proof}
We have $Q(\omega_{j})=0$ for all $j=1,\dots,\np$. Therefore\[
\sum_{i=0}^{\np}q_{i}m_{k+i}=\sum_{i=0}^{\np}q_{i}\sum_{j=1}^{\np}\jc_{0,j}\omega_{j}^{k+i}=\sum_{j=1}^{\np}\jc_{0,j}\omega_{j}^{k}\sum_{i=0}^{\np}q_{i}\omega_{j}^{i}=\sum_{j=1}^{\np}\jc_{0,j}\omega_{j}^{k}Q\left(\omega_{j}\right)=0\qedhere\]
\end{proof}
\begin{cor}
Let $q_{\np}=1$ for normalization. Then for all $k\in\naturals$
the coefficient vector $\left\{ q_{i}\right\} _{i=0}^{\np-1}$ is
the solution of the linear system\begin{equation}
\underbrace{\left[\begin{array}{cccc}
m_{k} & m_{k+1} & \cdots & m_{k+\np-1}\\
m_{k+1} & m_{k+2} & \cdots & m_{k+\np}\\
\vdots & \vdots & \vdots & \vdots\\
m_{k+\np-1} & m_{k+\np} & \cdots & m_{k+2\np-2}\end{array}\right]}_{\isdef H_{k}}\times\left[\begin{array}{c}
q_{0}\\
q_{1}\\
\vdots\\
q_{\np-1}\end{array}\right]=-\left[\begin{array}{c}
m_{k+\np}\\
m_{k+\np+1}\\
\vdots\\
m_{k+2\np-1}\end{array}\right]\label{eq:prony-exact-system}\end{equation}

\end{cor}
After this preparation, we can now describe the algorithm for obtaining
initial estimates using Prony's method. Recall that our a-priori bounds
are given by $\jmin,\jmax,\jdist$ and $T$ - see \prettyref{sec:localization}.
\begin{algorithm}
\label{alg:prony-initial-estimates}Let us be given the first $\startcoeff+2\np-1$
Fourier coefficients $\fc(\fun)$ of a function $\fun$ with $\np$
unknown discontinuities $\left\{ \jp_{j}\right\} _{j=1}^{\np}$, which
is continuously differentiable between these discontinuities. Denote
the magnitudes of the jumps by $\left\{ \jc_{j}\right\} _{j=1}^{\np}$.
\begin{enumerate}
\item Calculate the sequence\[
r_{k}=2\pi\left(\imath k\right)\fc(\fun)\]

\item Solve the system\begin{equation}
\underbrace{\left[\begin{array}{cccc}
r_{\startcoeff} & r_{\startcoeff+1} & \cdots & r_{\startcoeff+\np-1}\\
r_{\startcoeff+1} & r_{\startcoeff+2} & \cdots & r_{\startcoeff+\np}\\
\vdots & \vdots & \vdots & \vdots\\
r_{\startcoeff+\np-1} & r_{\startcoeff+\np} & \cdots & r_{\startcoeff+2\np-2}\end{array}\right]}_{\isdef\nn{H_{k}}}\times\left[\begin{array}{c}
\nn q_{0}\\
\nn q_{1}\\
\vdots\\
\nn q_{\np-1}\end{array}\right]=-\left[\begin{array}{c}
r_{\startcoeff+\np}\\
r_{\startcoeff+\np+1}\\
\vdots\\
r_{\startcoeff+2\np-1}\end{array}\right]\label{eq:prony-approximate-system}\end{equation}

\item Take the estimated $\left\{ \nn{\omega}_{j}\right\} $ to be the roots
of the polynomial\[
\nn Q(z)=z^{\np}+\sum_{i=0}^{\np-1}\nn q_{i}z^{i}\]
and then set\[
\nn{\jp}_{j}=-\arg\nn{\omega}_{j}\]

\end{enumerate}
\end{algorithm}
Now we would like to analyze the accuracy of \prettyref{alg:prony-initial-estimates}.
First, we need to estimate the error in solving the system \eqref{eq:prony-approximate-system}.
We use standard result from numerical linear algebra.
\begin{lem}
\label{lem:condition-number}Consider the linear system $Ax=b$ and
let $\vec{x_{0}}$ be the exact solution. Let this system be perturbed:\[
\left(A+\Delta A\right)x=b+\Delta b\]
and let $x_{0}+\Delta x$ denote the exact solution of this perturbed
system. Denote\[
\delta x=\frac{\|\Delta x\|}{\|x_{0}\|}\qquad\delta A=\frac{\|\Delta A\|}{\|A\|}\qquad\delta b=\frac{\|\Delta b\|}{\|b\|}\qquad\kappa=\|A\|\|A^{-1}\|\enskip\text{(condition number)}\]

for some vector norm $\|\cdot\|$ and the induced matrix norm. Then\begin{equation}
\delta x\leq\frac{\kappa}{1-\kappa\cdot\delta A}\left(\delta A+\delta b\right)\label{eq:accuracy-through-condition-number}\end{equation}
\end{lem}
\begin{proof}
See e.g. \cite{wilkinson1994rounding}.
\end{proof}
Consider \eqref{eq:prony-approximate-system}. The error in the right-hand
side is given by \eqref{eq:order0-fc-bound}. Therefore we now need
to estimate the condition number of the matrix $H_{k}$. Although
all the entries are bounded, it may still happen%
\footnote{Consider for instance $H_{k}=\left[\begin{array}{cc}
1 & 1+\frac{1}{k}\\
1 & 1-\frac{1}{k}\end{array}\right]$%
} that $\kappa\left(H_{k}\right)$ is unbounded. Fortunately, this
is not the case. To see this, we are going to factorize $H_{k}$ into
a component which depends on $k$, and a component which doesn't.
\begin{lem}
Let $V=V\left(\jp_{1},\dots,\jp_{\np}\right)$ denote the Vandermonde
matrix on the nodes $\left\{ \omega_{j}\right\} $, i.e.\[
V=\left[\begin{array}{cccc}
1 & 1 & \dots & 1\\
\omega_{1} & \omega_{2} & \dots & \omega_{\np}\\
\vdots & \vdots & \ddots & \vdots\\
\omega_{1}^{\np-1} & \omega_{2}^{\np-1} & \dots & \omega_{\np}^{\np-1}\end{array}\right]\]

Then for all $k\in\naturals$\[
H_{k}=V\times\diag\left\{ \jc_{0,j}\omega_{j}^{k}\right\} \times V^{T}\]
\end{lem}
\begin{proof}
Direct computation from the definitions \eqref{eq:prony-system} and
\eqref{eq:prony-exact-system}.\end{proof}
\begin{cor}
For all $k\in\naturals$\[
\kappa\left(H_{k}\right)\leq\frac{\jmax}{\jmin}\kappa\left(V\right)\]
\end{cor}
\begin{rem}
$\kappa\left(V\right)$ is well-studied in e.g. \cite{gautschi1974nei}.
It essentially depends on the minimal distance between the nodes.
In particular:\[
\|V^{-1}\|\sim\max_{1\leq i\leq\np}\prod_{j=1,j\neq i}^{n}\frac{1}{|\omega_{j}-\omega_{i}|}\]
\end{rem}
\begin{lem}
\label{lem:Qbig-coeffs-pert}There exist constants $\Cl{q-coeffs-error-bound},\Cl[kk]{kk-q-coeffs-error-bound}$
such that for all $i=0,1,\dots,\ord$ and for all $k>\Cr{kk-q-coeffs-error-bound}\frac{T\jmax}{\jmin^{2}}\kappa\left(V\right)$\[
\left|q_{i}-\nn q_{i}\right|\leq\Cr{q-coeffs-error-bound}\frac{T\jmax}{\jmin^{2}}\kappa\left(V\right)k^{-1}\]
\end{lem}
\begin{proof}
In the context of \prettyref{lem:condition-number}, our original
system is $H_{k}\vec q=\vec m$ \eqref{eq:prony-exact-system} and
the perturbed system is $\nn H_{k}\vec{\nn q}=\vec{\nn m}$ \eqref{eq:prony-approximate-system}.
Note that $\left|m_{k}\right|\geq\jmin\cdot\Cl{mk-bound}$ for some
$\Cr{mk-bound}$. From previous considerations we therefore have\[
\begin{split}\delta H_{k} & =\frac{\|\nn H_{k}-H_{k}\|}{\|H_{k}\|}\leq\Cl{delta-hk}\frac{T}{\jmin}\cdot\frac{1}{k}\\
\delta\vec m & \leq\Cl{delta-m}\cdot\frac{T}{\jmin}\cdot\frac{1}{k}\\
\kappa\left(H_{k}\right) & \leq\frac{\jmax}{\jmin}\kappa\left(V\right)\end{split}
\]
We would like to estimate $\delta\vec q$ according to \eqref{eq:accuracy-through-condition-number}.
If\[
k>\underbrace{2\Cr{delta-hk}}_{\isdef\Cr{kk-q-coeffs-error-bound}}\frac{T}{\jmin}\cdot\frac{\jmax}{\jmin}\kappa\left(V\right)\]
then $\kappa\left(H_{k}\right)\delta H_{k}\leq\frac{1}{2}$ and so\[
\delta\vec q\leq\underbrace{2\left(\Cr{delta-hk}+\Cr{delta-m}\right)}_{\isdef\Cr{q-coeffs-error-bound}}\frac{\jmax}{\jmin}\kappa\left(V\right)\frac{T}{\jmin}\cdot\frac{1}{k}\qedhere\]

\end{proof}
We can finally estimate the accuracy of \prettyref{alg:prony-initial-estimates}.
To shorten notation, let\[
F=F\left(\magbounds,\geom,T\right)\isdef\frac{T\jmax}{\jmin^{2}}\kappa\left(V\right)\]

\begin{thm}
\label{thm:initial-estimates-prony} For every $0<\alpha<1$ there
exist constants $\Cl{prony-final}\left(\alpha\right),\Cl[kk]{kk-prony-final}\left(\alpha,F\right)$
such that for every $k>\Cr{kk-prony-final}$ \prettyref{alg:prony-initial-estimates}
reconstructs the locations of the jumps with accuracy\[
\left|\jp_{j}-\nn{\jp}_{j}\right|\leq\Cr{prony-final}\cdot F\cdot k^{\alpha-1}\]
\end{thm}
\begin{proof}
We have shown that the perturbation $Q-\nn Q$ has coefficients of
magnitude $k^{-1}$. We will use the same reasoning as in \prettyref{sec:accuracy-single-jump}
in order to estimate the quantity $\left|\omega_{j}-\nn{\omega}_{j}\right|$
- which is the perturbation of the roots of $Q\left(z\right)$. Let
$\left|z\right|<2$. Since the coefficients of $Q\left(z\right)$
do not depend on $k$, we can obviously find constants $\Cl{Qbig-first-der}$
and $\Cl{Qbig-second-der}$ such that
\begin{enumerate}
\item $\left|Q'\left(\omega_{j}\right)\right|\geq\Cr{Qbig-first-der}$
\item $\left|Q''\left(z\right)\right|<\Cr{Qbig-second-der}$
\end{enumerate}
By a reasoning similar to \prettyref{lem:lower-bound-values-circle-via-derivatives}
we conclude that there exist constants $\Cl{Qbig-rho-lower}<1,\Cl{Qbig-lower-bound}$
such that for every function $0<\rho\left(k\right)<\Cr{Qbig-rho-lower}$
we have\begin{equation}
\left|Q\left(z\right)\right|>\Cr{Qbig-lower-bound}\rho\left(k\right)\qquad\forall z\in\partial B_{\rho\left(k\right)}\left(\omega_{j}\right)\label{eq:Qbig-lower}\end{equation}

Now let \[
\rho\left(k\right)=F\cdot k^{\alpha-1}\]
If $k>\left(\frac{F}{\Cr{Qbig-rho-lower}}\right)^{\frac{1}{1-\alpha}}$,
then $\rho\left(k\right)<\Cr{Qbig-rho-lower}$ and so \eqref{eq:Qbig-lower}
holds. On the other hand, by \prettyref{lem:Qbig-coeffs-pert} we
have that\[
\left|\left(Q-\nn Q\right)\left(z\right)\right|\leq\Cl{Qbig-diff-upper}\cdot F\cdot k^{-1}\]

Now finally we require that $k>\left(\frac{\Cr{Qbig-diff-upper}}{\Cr{Qbig-lower-bound}}\right)^{\frac{1}{\alpha}}$,
in which case\[
\left|\left(Q-\nn Q\right)\left(z\right)\right|\leq\Cr{Qbig-diff-upper}\cdot F\cdot k^{-1}<\Cr{Qbig-lower-bound}Fk^{\alpha-1}<\left|Q\left(z\right)\right|\]
and therefore $\nn Q(z)$ has a simple zero inside $B_{\rho\left(k\right)}\left(\omega_{j}\right)$.

Thus we have shown that $\left|\nn{\omega}_{j}-\omega_{j}\right|\leq F\cdot k^{\alpha-1}$.
Write $\nn{\omega}_{j}=\omega_{j}+\beta\left(k\right)k^{\alpha-1}$
where $\left|\beta\left(k\right)\right|<F$. Then by Taylor expansion
of the logarithm we will have (recall $\left|\omega_{j}\right|=1$)
for large enough $k>\Cr{kk-prony-final}\left(\alpha\right)$

\[
\left|\nn{\jp}_{j}-\jp_{j}\right|=\left|\arg\omega_{j}-\arg\nn{\omega}_{j}\right|=\left|\arg\left(\frac{\omega_{j}}{\nn{\omega}_{j}}\right)\right|\leq\Cr{prony-final}\left(\alpha\right)\cdot F\cdot k^{\alpha-1}\qedhere\]

\end{proof}

\end{document}